\documentclass[10pt]{amsart}

\usepackage{latexsym,exscale,enumerate,amsfonts,amssymb, ulem, xparse, mathtools}
\usepackage{amsmath,amsthm,amsfonts,amssymb,amscd, stmaryrd,textcomp}
\usepackage{hyperref}
\usepackage{ulem}
\usepackage{bbold}
\usepackage[usenames,dvipsnames]{xcolor}
\colorlet{green}{black!30!green} 

\usepackage[all]{xy}
\SelectTips{cm}{}

\usepackage{graphicx}

\usepackage{tikz}
\usetikzlibrary{calc}
\usetikzlibrary{decorations.markings}
\usetikzlibrary{decorations.pathreplacing}
\usetikzlibrary{arrows,shapes,positioning}
\tikzstyle directed=[postaction={decorate,decoration={markings,
    mark=at position #1 with {\arrow{>}}}}]
\tikzstyle rdirected=[postaction={decorate,decoration={markings,
    mark=at position #1 with {\arrow{<}}}}]
\tikzset{anchorbase/.style={baseline={([yshift=-0.5ex]current bounding box.center)}}}

\tikzset{
    partial ellipse/.style args={#1:#2:#3}{
        insert path={+ (#1:#3) arc (#1:#2:#3)}
    }
}

\newcommand{\rep}{\Rep(\glnn{N})}
\newcommand{\reptwo}{\Rep(\slnn{2})}

\newcommand{\sessAWeb}{N\cat{A}\cat{Web}_\mathrm{s}^{+,\mathrm{ess}}}
\newcommand{\repp}{\Rep^{+}(\glnn{N})}
\newcommand{\reph}{\Rep(\h)}
\newcommand{\rephp}{\Rep^{+}(\h)}
\newcommand{\Sym}{\mathrm{Sym}}

\def\X{{\mathbbm X}}
\def\h{{\mathfrak h}}

\def\1{\mathbbm{1}}%

\newtheorem{theorem}{Theorem}

\newtheorem{corollary}[theorem]{Corollary}

\newtheorem{definition}[theorem]{Definition}

\newtheorem{lemma}[theorem]{Lemma}

\newtheorem{proposition}[theorem]{Proposition}
\newtheorem{question}[theorem]{Question}
\newtheorem{remark}[theorem]{Remark}

\newcommand{\slnn}[1]{\mf{sl}_{#1}}

\newcommand{\glnn}[1]{\mf{gl}_{#1}}

\newcommand{\sh}{w}
\newcommand{\cat}[1]{\ensuremath{\mbox{\bfseries {\upshape {#1}}}}}
\newcommand{\diagrep}{\phi} 
\newcommand{\pTr}{\mathrm{pTr}}
\newcommand{\bVn}{\bigwedge}
\newcommand{\Kar}{\operatorname{Kar}}
\newcommand{\sym}{\mathrm{Sym}}

\newcommand{\Web}[1][]{#1\cat{Web}}
\newcommand{\Webp}[1][]{#1\cat{Web}^{+}}
\newcommand{\Webq}[1][]{#1\cat{Web}_{q}}

\newcommand{\wrap}{D}
\newcommand{\wrapi}{D^{-1}}

\newcommand{\AWeb}[1][]{#1\cat{A}\cat{Web}}
\newcommand{\AWebq}[1][]{#1\cat{A}\cat{Web}_{q}}
\newcommand{\AWebp}[1][]{#1\AWeb^+}

\newcommand{\essAWeb}[1][]{#1\AWeb^{\mathrm{ess}}}
\newcommand{\essAWebp}[1][]{#1\AWeb^{\mathrm{ess},+}}

\newcommand{\essbAWeb}[1][]{\overline{#1\AWeb}^{\mathrm{ess}}}
\newcommand{\essbAWebp}[1][]{\overline{#1\AWeb}^{\mathrm{ess},+}}

\newcommand{\cev}[1]{\reflectbox{\ensuremath{\vec{\reflectbox{\ensuremath{#1}}}}}}

\newcommand{\Rep}{\cat{Rep}}

\newcommand{\Hom}{{\rm Hom}}

\renewcommand{\to}{\rightarrow}

\newcommand{\id}{{\rm id}}

\def\mf{\mathfrak}
\def\shuffle{\,\raise 1pt\hbox{$\scriptscriptstyle\cup{\mskip
               -4mu}\cup$}\,}

\numberwithin{equation}{section}

\def\comm#1{}%
\def\emph#1{{\sl #1\/}}

\let\tilde=\widetilde

\let\phi=\varphi
\let\theta=\vartheta
\let\epsilon=\varepsilon

\usepackage{bbm}
\def\C{{\mathbbm C}}
\def\N{{\mathbbm N}}
\def\R{{\mathbbm R}}
\def\Z{{\mathbbm Z}}

\def\1{\mathbbm{1}}%

\def\la{\langle}
\def\ra{\rangle}

\begin{document}
%

\title[Extremal weight projectors II]
{
Extremal weight projectors II}

\author{Hoel Queffelec}
 \address{IMAG\\ Univ. Montpellier\\ CNRS \\ Montpellier \\ France}
\email{hoel.queffelec@umontpellier.fr}

 \author{Paul Wedrich}
\address{Mathematical Sciences Institute\\ The Australian National University \\ Australia}
\email{p.wedrich@gmail.com}
\urladdr{paul.wedrich.at}

\begin{abstract}
In previous work, we have constructed diagrammatic idempotents in an affine extension of the Temperley--Lieb category, which describe extremal weight projectors for $\slnn{2}$, and which categorify Chebyshev polynomials of the first kind. In this paper, we generalize the construction of extremal weight projectors to the case of $\glnn{N}$ for $N \geq 2$, with a view towards categorifying the corresponding torus skein algebras via Khovanov--Rozansky link homology. As by-products, we obtain compatible diagrammatic presentations of the representation categories of $\glnn{N}$ and its Cartan subalgebra, and a categorification of power-sum symmetric polynomials.
\end{abstract}

\maketitle

\section{Introduction}

The topological motivation for this article is the search for an extension of Khovanov--Rozansky link homologies \cite{Kh1, KhR} to invariants of links in 3-manifolds other than $\R^3$. Since quantum link homologies, in their mode of definition and computation, currently depend on the presentation of links as 2-dimensional projections, the most accessible 3-manifolds in this endeavor are thickened surfaces $\Sigma\times I$. Just as Khovanov homology categorifies the Jones polynomial, the surface link homologies should categorify surface skein modules \cite{Prz1,Tur, APS}. These admit an algebra structure induced by stacking links, with distinguished bases (conjecturally) satisfying strong integrality and positivity properties \cite{FG, Thu, Le, MS}. They provide quantizations of surface character varieties that play an important role in quantum Teichm{\"u}ller theory \cite{BW1}. Both aspects make such \textit{skein algebras} prime targets for categorification via link homology technology.
\medskip

In order to categorify quantum invariants, it is useful to have explicit, combinatorial or diagrammatic descriptions of underlying representation categories. For example, Khovanov homology can be built from a categorification of the Temperley--Lieb category, which describes the representation category of $U_q(\slnn{2})$, and all incarnations of Khovanov--Rozansky homology implicitly employ a categorification of the \textit{MOY} or \textit{web calculus} for the representation category of $U_q(\glnn{N})$ \cite{MOY, CKM}.
\medskip

The main purpose of the present paper is to provide representation-theoretic tools for categorifying skein algebras. The key novelty when working with skein algebras is that their proposed distinguished bases are obtained from links colored not by irreducible representations of the corresponding quantum group, but colored only by their \textit{extremal weight spaces}. In \cite[Section 1.3]{QW} we have proposed a strategy for lifting these colorings to the categorified level of toric link homologies, which is inspired by Khovanov's categorification of the colored Jones polynomial~\cite{Kh7}. The main tool necessary in this approach is a diagrammatic presentation of the representation category of a Cartan subalgebra $U(\mathfrak{h})\subset U(\glnn{N})$, which is the first result in this paper.

\begin{theorem}[Corollary~\ref{cor:diagrep}]
\label{thm:1} There exists a diagrammatic presentation for the representation category of $U(\mathfrak{h})$, given by an affine extension $\essAWeb[N]$ of the web calculus for $U(\glnn{N})$ in which $\bVn^k(V)$-labeled essential circles are set to zero for $0<k<N$.
\end{theorem}

We emphasize that we deal with universal enveloping algebras, rather than their quantizations. Those are intended to control the $\C$-linear \textit{morphism spaces} in our skein module categorifications \cite[Section 1.3]{QW}, whereas the expected quantum parameter $q$ is promoted to a grading on \textit{objects}. This is related to the fact that \textit{annular Khovanov homology} \cite{GLW} has a natural action of $U(\slnn{2})$, not of $U_q(\slnn{2})$. An additional quantization seems possible, see Remark~\ref{rem:q}, but will not play a role here. Note that the diagrammatic presentation from Theorem~\ref{thm:1} is compatible with one that follows from the work of Cautis--Kamnitzer \cite[Section 2.6]{CK_ann}.
\medskip

In \cite{QW}, we have constructed a diagrammatic presentation as in Theorem~\ref{thm:1} for the case of $\slnn{2}$, and identified  idempotent morphisms that encode the projections onto extremal weight spaces in finite-dimensional $U(\slnn{2})$-representations. These \textit{extremal weight projectors} are analogous to, but finer than Jones--Wenzl projectors~\cite{Jon2,Wen}, and they can also be defined recursively. 
In this article, we identify and study extremal weight projectors for $\glnn{N}$. 

\begin{theorem} The diagrammatic category $\essAWeb[N]$ contains recursively defined idempotents that correspond to projections onto the extremal weight spaces in the $U(\glnn{N})$-representations $\sym^k(V)$. 
\end{theorem} 
In fact we prove a slightly stronger version of this theorem in a central extension of $\essAWeb[N]$, which has an additional grading by winding number, that will be important for categorifying skein modules, see Theorem~\ref{thm:Tm}.
\medskip 

The $\slnn{2}$ extremal weight projectors can be considered as categorifications of Chebyshev polynomials of the first kind by decategorifying their images to elements of the representation ring $K_0(\reptwo)\cong \Z[X]$. Analogously, the extremal weight projectors for $\glnn{N}$ categorify power-sum symmetric polynomials in the representation ring of $\glnn{N}$. Such categorifications of classical orthogonal polynomials are of independent interest, see e.g. \cite{KSa}. Motivated by this, we prove a categorified Newton identity.

\begin{theorem}[Theorem~\ref{thm:newton}] The extremal weight projectors satisfy a categorified version of the Newton identity relating power-sum symmetric and elementary symmetric polynomials.
\end{theorem} 

The main application for extremal weight projectors, however, is in categorifying toric skein modules. In a separate paper \cite{QW3}, we will construct a categorification of the $\glnn{2}$ skein module of the thickened torus via a toric $\glnn{2}$ foam category. The category $\essAWeb[2]$ describes morphism spaces in this foam category, with affine webs corresponding to rotationally symmetric foams. In particular, the toric foam category contains primitive idempotents given by rotation foams obtained from extremal weight projectors, and it is the target of a toric link homology functor. The use of $\glnn{2}$ foams, as opposed to Bar-Natan cobordisms \cite{BN2}, is necessary to guarantee the functoriality of the resulting link homology, see \cite{Blan, ETW}. In preparation for this construction, we prove a delooping lemma for $\glnn{2}$ webs and decomposition formulas for tensor products of $\glnn{2}$ extremal weight projectors in Section~\ref{sec:two}.

\begin{remark} Affine web categories have appeared before in work of the first-named author \cite{Queff_aff} on skein modules, and of Cautis--Kamnitzer \cite{CK_ann} on a K-theoretic version of the derived geometric Satake correspondence for $SL_N$. The main differences are that here we work at $q=1$, which makes the affine web categories symmetric monoidal, and that we take a quotient by $\bVn^k(V)$-labeled essential circles for $0<k<N$. It is unclear to us how to define an analogous quotient for generic $q$ that would admit extremal weight projectors.
\end{remark}
\begin{remark}
\label{rem:q}
Affine web categories at generic $q$ describe morphism spaces in quantized toric foam categories, which can be defined using a quantized horizontal trace construction. This is analogous to the quantized annular Bar-Natan cobordisms of Beliakova-Putyra-Wehrli \cite{BPW}. However, such a quantization involves a non-canonical choice of a simple closed curve on the torus, that breaks a natural mapping class group action which is desirable for categorified skein modules. In \cite{QW3}, we thus proceed with affine webs at $q=1$ and unquantized toric foam categories.
\end{remark}

\subsection*{Acknowledgements}  We would like to thank Anna Beliakova, C\'edric Bonnaf\'e, Eugene Gorsky, Joel Kamnitzer and Radmila Sazdanovi\'c for interesting discussions. Part of this work was done during the programme ``Homology theories in low dimensional topology'' at the Isaac Newton Institute for Mathematical Sciences, which was supported by the UK EPSRC [Grant Number EP/K032208/1]. We thank the Newton Institute for providing the ideal working environment for categorifying Newton identities.

\subsection*{Funding}
The work of H.~Q. was partially supported by a PEPS Jeunes Chercheuses et Jeunes Chercheurs, the ANR Quantact and by the CNRS-MSI partnership ``LIA AnGe''. The work of P.~W. was supported by the Leverhulme Trust [Research Grant RP2013-K-017] and the Australian Research Council Discovery Projects ``Braid groups and higher representation theory'' and ``Low dimensional categories'' [DP140103821, DP160103479].

\section{Affine \texorpdfstring{$\glnn{N}$}{gl(N)} webs and extremal weight projectors}
We start by recalling the diagrammatic calculus of $\glnn{N}$ webs, which describes the category of representations of $U_q(\glnn{N})$ that is monoidally generated by exterior powers of the vector representation and their duals. 
\subsection{The category of \texorpdfstring{$\glnn{N}$}{gl(N)} webs}
\label{sec:webs}
The category $\Webq[N]$ of $\glnn{N}$ webs is the $\C(q)$-linear pivotal tensor category with objects generated by points on the line $\R$ that are labeled with integers in the set $\{1,\dots,N\}$ and that carry an orientation \textit{up} or \textit{down}. We may consider the tensor unit as a $0$-labeled point without orientation. The morphisms are generated by trivalent graphs, properly embedded in the strip $\R\times [0,1]$, with edges oriented and labeled in the same set, with a flow condition at each vertex imposing that the sum of incoming labels equals the sum of the outgoing ones. These graphs are interpreted as mapping from the bottom sequence of boundary points (with labels and orientations) to that at the top. The morphisms are considered modulo isotopy relative to the boundary and subject to the local relations

\begin{gather}
\label{eq:webrel}  
\begin{tikzpicture}[anchorbase,scale=.7]
    \draw [very thick,directed=.55] (0,0) to (0,.5); 
    \draw [very thick,directed=.55] (0,1.5) to (0,2); 
   \draw [very thick,directed=.55] (0,.5) to  [out=45,in=315] (0,1.5);
    \draw [very thick,directed=.55] (0,.5) to  [out=135,in=225] (0,1.5);
\node at (.5,1) {\tiny$l$} ;
\node at (-.5,1) {\tiny$k$} ;   
  \end{tikzpicture}
  \;\;=\;\; {k+l \brack k}\;
  \begin{tikzpicture}[anchorbase,scale=.7]
    \draw [very thick] (0,0) -- (0,2);
    \node at (.45,.25) {\tiny$k+l$} ;    
  \end{tikzpicture}\quad,\quad
  \begin{tikzpicture}[anchorbase,scale=.7]
    \draw [very thick,directed=.55,directed=.9,directed=.2] (0,0) to (0,2); 
   \draw [very thick,directed=.55] (0,1.25) to  [out=90,in=180] (.25,1.5) to [out=0,in=90] (.5,1.25) to (.5,.75) to  [out=-90,in=0] (.25,.5) to [out=180,in=-90] (0,.75);
\node at (.75,1) {\tiny$l$} ;
\node at (.25,.25) {\tiny$k$} ;   
  \end{tikzpicture}
  \;\;=\;\; {N-k \brack l}\;
  \begin{tikzpicture}[anchorbase,scale=.7]
    \draw [very thick] (0,0) -- (0,2);
    \node at (.25,.25) {\tiny$k$} ;    
  \end{tikzpicture}
  \quad,\quad  
\begin{tikzpicture}[anchorbase,scale=.7]
    \draw [very thick,directed=.55] (0,0) to [out=90,in=225] (0.5,.75);
        \draw [very thick,directed=.55] (1,0) to [out=90,in=315] (0.5,.75);   
    \draw [very thick,directed=.55] (2,0) to [out=90,in=315] (1,1.5);
     \draw [very thick,directed=.55] (0.5,0.75) to [out=90,in=225] (1,1.5);
     \draw [very thick,directed=.55] (1,1.5) to (1,2);
  \end{tikzpicture}
  \;\;=\;\;
  \begin{tikzpicture}[anchorbase,scale=.7]
    \draw [very thick,directed=.55] (0,0) to [out=90,in=225] (1,1.5);
        \draw [very thick,directed=.55] (1,0) to [out=90,in=225] (1.5,.75);   
    \draw [very thick,directed=.55] (2,0) to [out=90,in=315] (1.5,.75);
     \draw [very thick,directed=.55] (1.5,0.75) to [out=90,in=315] (1,1.5);
     \draw [very thick,directed=.55] (1,1.5) to (1,2);
  \end{tikzpicture}
\\
\nonumber
\begin{tikzpicture}[anchorbase,scale=.75]
  \draw [very thick,->] (0,0) -- (0,1.5);
  \draw [very thick, ->] (1,0) -- (1,1.5);
  \draw [very thick, directed=.5] (0,.4) -- (1,.6);
  \draw [very thick, directed=.5] (1,.9) -- (0,1.1);
  \node at (.5,.2) {\tiny $b$};
  \node at (.5,1.2) {\tiny $a$};
  \node at (0,-.2) {\tiny $k$};
  \node at (1,-.2) {\tiny $l$};
\end{tikzpicture} 
\;\;=\;\;
\sum_{t} {k-l+a-b\brack t}\;\;
\begin{tikzpicture}[anchorbase,scale=.75]
  \draw [very thick,->] (0,0) -- (0,1.5);
  \draw [very thick, ->] (1,0) -- (1,1.5);
  \draw [very thick, directed=.5] (1,.4) -- (0,.6);
  \draw [very thick, directed=.5] (0,.9) -- (1,1.1);
  \node at (.5,.2) {\tiny $a-t$};
  \node at (.5,1.2) {\tiny $b-t$};
  \node at (0,-.2) {\tiny $k$};
  \node at (1,-.2) {\tiny $l$};
\end{tikzpicture}
\quad,\quad
\begin{tikzpicture}[anchorbase,scale=.75]
  \draw [very thick,->] (0,0) -- (0,1.5);
  \draw [very thick, <-] (1,0) -- (1,1.5);
  \draw [very thick, rdirected=.5] (0,.6) to [out=330,in=210](1,.6);
  \draw [very thick, rdirected=.5] (1,.9) to [out=150,in=30]  (0,.9);
  \node at (.5,.1) {\tiny $1$};
  \node at (.5,1.3) {\tiny $1$};
  \node at (0,-.2) {\tiny $k$};
  \node at (1,-.2) {\tiny $l$};
\end{tikzpicture}
\;\;=\;\;
\begin{tikzpicture}[anchorbase,scale=.75]
  \draw [very thick,->] (0,0) -- (0,1.5);
  \draw [very thick, <-] (1,0) -- (1,1.5);
  \draw [very thick, rdirected=.5] (1,.4) to [out=150,in=30] (0,.4);
  \draw [very thick, rdirected=.5] (0,1.1) to [out=330,in=210] (1,1.1);
  \node at (.5,.2) {\tiny $1$};
  \node at (.5,1.2) {\tiny $1$};
  \node at (0,-.2) {\tiny $k$};
  \node at (1,-.2) {\tiny $l$};
\end{tikzpicture}
+
[N-k-l] \begin{tikzpicture}[anchorbase,scale=.75]
  \draw [very thick,->] (0,0) -- (0,1.5);
  \draw [very thick, <-] (1,0) -- (1,1.5);
  \node at (0,-.2) {\tiny $k$};
  \node at (1,-.2) {\tiny $l$};
\end{tikzpicture}
\end{gather}
as well as their reflections and orientation reversals. The following relations are useful consequences of the ones above:
\begin{equation}
\label{eq:webrel2}
\begin{tikzpicture}[anchorbase,scale=.75]
  \draw [very thick,directed=.5] (0,0) circle (.5);
  \node at (.6,-.25) {\tiny $k$};
\end{tikzpicture}
\;\; = \;\;
\begin{tikzpicture}[anchorbase,scale=.75]
  \draw [very thick,rdirected=.5] (0,0) circle (.5);
  \node at (.6,-.25) {\tiny $k$};
\end{tikzpicture}
\;\;=\;\;{N \brack k}\quad, \quad 
\begin{tikzpicture}[anchorbase,scale=.5]
\draw [very thick,->] (0,0) to  (0,2);
\draw [very thick,->] (1,2) to (1,0);
  \node at (0,-.2) {\tiny $N$};
  \node at (1,-.2) {\tiny $N$};
\end{tikzpicture}
\;\; =\;\;
\begin{tikzpicture}[anchorbase,scale=.5]
\draw [very thick,->] (0,0) to (0,.5) to [out=90,in=90] (1,.5) to (1,0);
\draw [very thick,->] (1,2) to (1,1.5) to [out=270,in=270] (0,1.5) to (0,2);
  \node at (0,-.2) {\tiny $N$};
  \node at (1,-.2) {\tiny $N$};
\end{tikzpicture}
\quad , \quad
\begin{tikzpicture}[anchorbase,scale=.5]
\draw [very thick,->] (0,0) to  (0,2);
\draw [very thick,->] (1,2) to (1,0);
  \node at (0,-.2) {\tiny $k$};
  \node at (1,-.2) {\tiny $N$};
\end{tikzpicture}
\;\; = \;\;
\begin{tikzpicture}[anchorbase,scale=.5]
\draw [very thick,->] (0,0) to (0,0.25) to [out=90,in=240] (0.5,0.75) to [out=120,in=240] (0.5,1.25) to[out=120,in=270] (0,1.75) to  (0,2);
\draw [very thick,<-] (1,0) to (1,0.25) to [out=90,in=0] (.5,.75) ;
\draw [very thick] (1,2) to (1,1.75) to [out=270,in=0] (0.5,1.25);
  \node at (0,-.2) {\tiny $k$};
  \node at (1,-.2) {\tiny $N$};
   \node at (1.25,1) {\tiny $N-k$};
\end{tikzpicture}
\end{equation}

\subsection{Link with \texorpdfstring{$U_q(\glnn{N})$}{Uq(gl(N)}-representation theory}
Let $\Rep(U_q(\glnn{N}))$ denote the $\C(q)$-linear pivotal tensor category of $U_q(\glnn{N})$-representations that is monoidally generated by exterior powers of the vector representation and their duals. The main purpose of the diagrammatic calculus of $\glnn{N}$ webs is to describe this category.

\begin{theorem}\label{thm:CKM} There exists an equivalence of $\C(q)$-linear pivotal tensor categories
\[\phi \colon \Webq[N] \to \Rep(U_q(\glnn{N}))\]
that sends $k$-labeled upward points to $k$-fold exterior powers of the vector representation of $U_q(\glnn{N})$.
\end{theorem}
Essentially, this theorem is due to Cautis--Kamnitzer--Morrison~\cite{CKM}, although they state it for $U_q(\slnn{N})$. Versions for $U_q(\glnn{N})$ have appeared in \cite{QS2,TVW}. We now describe the functor $\phi$ explicitly.

Recall that $U_q(\glnn{N})$ is the $\C(q)$-algebra generated by $E_i,F_i$ for $1\leq i \leq N-1$ and $L^{\pm 1}_j$ for $1\leq j \leq N$ subject to the following relations:
\begin{gather}
L_i E_i=q E_i L_i,\quad L_iF_i=q^{-1}F_iL_i,\quad L_{i+1} E_i=q^{-1} E_iL_{i+1},\quad L_{i+1}F_i=q F_i L_{i+1}
\\
[E_i,F_j]=\delta_{i,j} \frac{L_iL_{i+1}^{-1}-L_{i+1}L_{i}^{-1}}{q-q^{-1}},\quad [L_i,L_j]=0.
\\
E_i^2 E_j -[2] E_i E_j E_i + E_j E_i^2 =0 \text{ if } |i-j|=1  \text{ and } [E_i,E_j]=0 \text{ otherwise; analogously for } Fs.
\end{gather}
It is a Hopf algebra with coproduct, antipode and counit as follows:
\begin{gather*}
\Delta(E_i)=E_i\otimes L_iL_{i+1}^{-1} +1\otimes E_i, \quad \Delta(F_i)=F_i\otimes 1 +L_i^{-1}L_{i+1}\otimes F_i, \quad \Delta(L_i^{\pm 1})=L_i^{\pm 1}\otimes L_i^{\pm 1}\\
S(L_i^{\pm 1})=L_i^{\mp}, \quad S(E_i)=-E_i L_i^{-1}L_{i+1}, \quad S(F_i)= - L_iL_{i+1}^{-1}F_i \\
\epsilon(L_i^{\pm 1})= 1, \quad \epsilon(E_i)=0, \quad \epsilon(F_i)=0
\end{gather*}
Let $V=\C(q)\langle v_1,v_2,\dots v_N\rangle$ denote the vector representation of $U_q(\glnn{N})$ and $V^*=\C(q)\langle v_1^*,v_2^*,\dots,v_N^* \rangle$ its dual. 

To leftward oriented cups and caps, the functor $\phi$ associates the natural evaluation and co-evaluation maps for duals:
\begin{align*}
\begin{tikzpicture}[anchorbase, scale=.5]
\draw [very thick,->] (1,2) to (1,1.7)to [out=270,in=0] (.5,1.2) to [out=180,in=270] (0,1.7) to (0,2);
\end{tikzpicture}
\quad \xmapsto{\phi} \quad
\begin{cases}
\C \to  V \otimes V^* \\
1\mapsto  \sum_{k=1}^N v_k\otimes v_k^*
\end{cases}
\quad ,\quad
\begin{tikzpicture}[anchorbase, scale=.5]
\draw [very thick,->] (1,0) to (1,0.3)to [out=90,in=0] (.5,.8) to [out=180,in=90] (0,.3) to (0,0);
\end{tikzpicture}
\quad \xmapsto{\phi} \quad
\begin{cases}
V^*\otimes V \to \C \\
v_{k}^*\otimes v_{l} \mapsto \delta_{k,l}
\end{cases}
\end{align*}
Let $\bVn^k V$ denote the k-th exterior power of $V$. This has a basis indexed by subsets $S=\{i_1,\dots,i_k \}\subset \{1,2,\dots, N\}$ of size $k$. If $1\leq i_1<\cdots < i_k\leq N$, we use the following notation for the corresponding basis vector: $v_S:=v_{i_1}\wedge \cdots \wedge v_{i_k}$. The dual $(\bVn^k V)^*\cong \bVn^k V^*$ then has the dual basis given by vectors $v_S^*$ and we have corresponding thickness $k$ cap and cup morphisms as above. 

\begin{align*}
\begin{tikzpicture}[anchorbase, scale=.5]
\draw [very thick,->] (1,2) to (1,1.7)to [out=270,in=0] (.5,1.2) to [out=180,in=270] (0,1.7) to (0,2);
\node at (.5,1.5) {\tiny$k$} ;
\end{tikzpicture}
\quad \xmapsto{\phi} \quad
\begin{cases}
\C \to  \bVn^k V \otimes \bVn^k V^* \\
1\mapsto  \sum_{|S|=k} v_S\otimes v_S^*
\end{cases}
\quad ,\quad
\begin{tikzpicture}[anchorbase, scale=.5]
\draw [very thick,->] (1,0) to (1,0.3)to [out=90,in=0] (.5,.8) to [out=180,in=90] (0,.3) to (0,0);
\node at (.5,.5) {\tiny$k$} ;
\end{tikzpicture}
\quad \xmapsto{\phi} \quad
\begin{cases}
\bVn^k V^*\otimes \bVn^k V \to \C \\
v_{S}^*\otimes v_{T} \mapsto \delta_{S,T}
\end{cases}
\end{align*}

The other duality maps, that is, rightward oriented caps and cups, are perturbed by powers of $q$. 
\begin{align*}
\begin{tikzpicture}[anchorbase, scale=.5]
\draw [very thick,<-] (1,2) to (1,1.7)to [out=270,in=0] (.5,1.2) to [out=180,in=270] (0,1.7) to (0,2);
\node at (.5,1.5) {\tiny$k$} ;
\end{tikzpicture}
\quad \xmapsto{\phi} \quad
\begin{cases}
\C \to  \bVn^k V^* \otimes \bVn^k V \\
1\mapsto  \sum_{|S|=k} q^{-\epsilon_S}v_S^*\otimes v_S
\end{cases}
\quad ,\quad
\begin{tikzpicture}[anchorbase, scale=.5]
\draw [very thick,<-] (1,0) to (1,0.3)to [out=90,in=0] (.5,.8) to [out=180,in=90] (0,.3) to (0,0);
\node at (.5,.5) {\tiny$k$} ;
\end{tikzpicture}
\quad \xmapsto{\phi} \quad
\begin{cases}
\bVn^k V\otimes \bVn^k V^* \to \C \\
v_{S}\otimes v_{T}^* \mapsto \delta_{S,T} q^{\epsilon_S}
\end{cases}
\end{align*}
Here $\epsilon_S= \sum_{i\in S}(N+1-2i)$.

Merges of thick strands act as exterior product:
\begin{align}
\begin{tikzpicture}[anchorbase, scale=.5]
\draw [very thick, directed=.55] (0,0) to [out=90,in=225] (.5,.7);
\draw [very thick, directed=.55] (1,0) to [out=90,in=315]  (.5,.7);
\draw [very thick,->] (.5,.7) -- (.5,1.4);
\node at (-.2,-.2) {\tiny$k$} ;
\node at (1.2,-.2) {\tiny$l$} ;
\node at (.5,1.6) {\tiny$k+l$} ;
\end{tikzpicture}
\quad \xmapsto{\phi} \quad
\begin{cases}
\bVn^k V\otimes \bVn^l V \to \bVn^{k+l} V \\
v_{S}\otimes v_{T} \mapsto 0 \quad \quad\quad \quad\text{ if } S\cap T\neq \emptyset \\
v_{S}\otimes v_{T} \mapsto (-q)^{\epsilon_{S,T}} v_{S\cup T} \quad\text{ otherwise } 
\end{cases}
\end{align}
Here $\epsilon_{S,T}$ is the number of inversions in the concatenation of the ordered lists of elements of $S$ and $T$. The split vertex acts as follows:
\begin{align}
\begin{tikzpicture}[anchorbase,scale=.5]
\draw [very thick, directed=.5] (.5,0) -- (.5,.7);
\draw [very thick, directed=1] (.5,.7) to [out=135,in=270]  (0,1.3)to (0,1.4);
\draw [very thick, directed=1] (.5,.7) to [out=45,in=270]  (1,1.3) to (1,1.4);
\node at (-.2,1.6) {\tiny$k$} ;
\node at (1.2,1.6) {\tiny$l$} ;
\node at (.5,-.2) {\tiny$k+l$} ;
\end{tikzpicture}
\quad \xmapsto{\phi}\quad
\begin{cases}
\bVn^{k+l} V  \to \bVn^{k} V \otimes \bVn^{l} V  \\
v_S \mapsto (-1)^{k l}\sum_{T\subset S, |T|=k} (-q)^{-\epsilon_{S\setminus T,T}} v_T\otimes v_{S\setminus T}
\end{cases}
\end{align}

Analogous formulas hold for merges and splits of duals, which implies that merges and splits can be slid around caps and cups:
\begin{align}
\begin{tikzpicture}[anchorbase, scale=.5]
\draw [very thick, <-] (0,0) to [out=90,in=225] (.5,.7);
\draw [very thick, <-] (1,0) to [out=90,in=315]  (.5,.7);
\draw [very thick,rdirected=.55] (.5,.7) -- (.5,1.4);
\node at (-.2,-.2) {\tiny$k$} ;
\node at (1.2,-.2) {\tiny$l$} ;
\node at (.5,1.6) {\tiny$k+l$} ;
\end{tikzpicture}
\quad \xmapsto{\phi} \quad
\begin{cases}
\bVn^k V^*\otimes \bVn^l V^* \to \bVn^{k+l} V^* \\
v_{S}^*\otimes v_{T}^* \mapsto 0 \quad \quad\quad \quad\text{ if } S\cap T\neq \emptyset \\
v_{S}^*\otimes v_{T}^* \mapsto (-1)^{k l}(-q)^{-\epsilon_{S,T}} v_{S\cup T}^* \quad\text{ otherwise } 
\end{cases}
\end{align}
and:
\begin{align}
\begin{tikzpicture}[anchorbase,scale=.5]
\draw [very thick, <-] (.5,0) -- (.5,.7);
\draw [very thick, rdirected=.55] (.5,.7) to [out=135,in=270]  (0,1.3)to (0,1.4);
\draw [very thick, rdirected=.55] (.5,.7) to [out=45,in=270]  (1,1.3) to (1,1.4);
\node at (-.2,1.6) {\tiny$k$} ;
\node at (1.2,1.6) {\tiny$l$} ;
\node at (.5,-.2) {\tiny$k+l$} ;
\end{tikzpicture}
\quad \xmapsto{\phi}\quad
\begin{cases}
\bVn^{k+l} V^*  \to \bVn^{k} V^* \otimes \bVn^{l} V^*  \\
v_S^* \mapsto \sum_{T\subset S, |T|=k} (-q)^{\epsilon_{S\setminus T,T}} v_T^*\otimes v_{S\setminus T}^*
\end{cases}
\end{align}
This concludes the description of the functor $\phi$.

The category $\Rep(U_q(\glnn{N}))$ is braided, and by virtue of Theorem~\ref{thm:CKM}, so is $\Webq[N]$. The diagrammatic description of the braiding of two fundamental $U_q(\glnn{N})$-representations in $\Webq[N]$ is given as follows:
\begin{equation} \label{eq:crossing}
\begin{tikzpicture}[anchorbase,scale=.75]
  \draw [very thick, ->] (1,0) -- (0,1.5);
  \draw [white, line width=.15cm] (0,0) -- (1,1.5);
  \draw [very thick, ->] (0,0) -- (1,1.5);
  \node at (0,-.3) {\tiny $k$};
  \node at (1,-.3) {\tiny $l$};
\end{tikzpicture}
\quad = \quad
(-q)^{kl}\sum_{\substack{a,b\geq 0\\ b-a=k-l}} (-q)^{b-k}\quad
\begin{tikzpicture}[anchorbase,scale=.75]
  \draw [very thick,->] (0,0) -- (0,1.5);
  \draw [very thick, ->] (1,0) -- (1,1.5);
  \draw [very thick, directed=.5] (0,.4) -- (1,.6);
  \draw [very thick, directed=.5] (1,.9) -- (0,1.1);
  \node at (.5,.2) {\tiny $b$};
  \node at (.5,1.2) {\tiny $a$};
  \node at (0,-.2) {\tiny $k$};
  \node at (1,-.2) {\tiny $l$};
\end{tikzpicture}
\end{equation}

In particular, a crossing of two $1$-labeled strands is given by:
\[
\begin{tikzpicture}[anchorbase,scale=.75]
  \draw [very thick, ->] (1,0) -- (0,1.5);
  \draw [white, line width=.15cm] (0,0) -- (1,1.5);
  \draw [very thick, ->] (0,0) -- (1,1.5);
\end{tikzpicture}
\quad = \quad
\begin{tikzpicture}[anchorbase,scale=.75]
  \draw [very thick, ->] (0,0) -- (0,1.5);
  \draw [very thick, ->] (1,0) -- (1,1.5);
\end{tikzpicture}
\quad - \;\;q\;\;
\begin{tikzpicture}[anchorbase,scale=.75]
  \draw [very thick] (0,0) to [out=90,in=-90] (.5,.5);
  \draw [very thick] (1,0) to [out=90,in=-90] (.5,.5);
  \draw [very thick, directed=.5] (.5,.5) -- (.5,1);
  \node at (.7,.75) {\tiny $2$};
  \draw [very thick,->] (.5,1) to [out=90,in=-90] (0,1.5);
  \draw [very thick,->] (.5,1) to [out=90,in=-90] (1,1.5);
\end{tikzpicture}
\]
For negative crossings, one uses the above formulas with $q$ inverted.

\begin{lemma} \label{lem:Reid} The following analogs of Reidemeister moves hold in $\Webq[N]$, where strands can carry all possible orientations and labels.
\begin{gather*}
  q^{-k(N-1)} \begin{tikzpicture}[anchorbase,scale=.7]
   \draw [very thick] (0,0) to [out=90,in=180] (.75,1.25) to [out=0,in=90] (1.25,.75); 
     \draw [white, line width=.15cm]  (1.25,.75) to [out=-90,in=0] (.75,.25) to [out=180,in=-90] (0,1.5);     
      \draw [very thick]  (1.25,.75) to [out=-90,in=0] (.75,.25) to [out=180,in=-90] (0,1.5);     
        \node at (0,-.3) {\tiny $k$};
  \end{tikzpicture}
  \;\;=\;\;
  \begin{tikzpicture}[anchorbase,scale=.7]
    \draw [very thick] (0,0) -- (0,1.5); 
     \node at (0,-.3) {\tiny $k$};
  \end{tikzpicture}
  \;\;=\;\;
  q^{k(N-1)} 
    \begin{tikzpicture}[anchorbase,scale=.7]
      \draw [very thick]  (1.25,.75) to [out=-90,in=0] (.75,.25) to [out=180,in=-90] (0,1.5);     
       \draw [white,line width=.15cm] (0,0) to [out=90,in=180] (.75,1.25) to [out=0,in=90] (1.25,.75);   
        \draw [very thick] (0,0) to [out=90,in=180] (.75,1.25) to [out=0,in=90] (1.25,.75);   
         \node at (0,-.3) {\tiny $k$};
  \end{tikzpicture}, 
  \\
  \begin{tikzpicture}[anchorbase,scale=.5]
    \draw [very thick] (0,0) to [out=90,in=-90] (1,1) to [out=90,in=-90] (0,2);
    \draw [white, line width=.15cm] (1,0) to [out=90,in=-90] (0,1) to [out=90,in=-90] (1,2);
    \draw [very thick] (1,0) to [out=90,in=-90] (0,1) to [out=90,in=-90] (1,2);
  \end{tikzpicture}
\quad=\quad
  \begin{tikzpicture}[anchorbase,scale=.5]
    \draw [very thick] (0,0) -- (0,2);
    \draw [very thick] (1,0) -- (1,2);
  \end{tikzpicture}
  \quad,\quad
  \begin{tikzpicture}[anchorbase,scale=.5]
    \draw [very thick] (0,0) to [out=90,in=-90] (.7,.7) to [out=90,in=-90] (1.4,1.4) -- (1.4,2.1);
        \draw [white, line width=.15cm] (.7,0) to [out=90,in=-90] (0,.7) -- (0,1.4) to [out=90,in=-90] (.7,2.1);
    \draw [very thick] (.7,0) to [out=90,in=-90] (0,.7) -- (0,1.4) to [out=90,in=-90] (.7,2.1);
        \draw [white, line width=.15cm] (1.4,0) -- (1.4,.7) to [out=90,in=-90] (.7,1.4) to [out=90,in=-90] (0,2.1);
    \draw [very thick] (1.4,0) -- (1.4,.7) to [out=90,in=-90] (.7,1.4) to [out=90,in=-90] (0,2.1);
  \end{tikzpicture}
\quad=\quad
  \begin{tikzpicture}[anchorbase,scale=.5]
    \draw [very thick] (0,0) -- (0,.7) to [out=90,in=-90] (.7,1.4) to [out=90,in=-90] (1.4,2.1);    
    \draw [white, line width=.15cm] (.7,0) to [out=90,in=-90] (1.4,.7) -- (1.4,1.4) to [out=90,in=-90] (.7,2.1);
    \draw [very thick] (.7,0) to [out=90,in=-90] (1.4,.7) -- (1.4,1.4) to [out=90,in=-90] (.7,2.1);
        \draw [white, line width=.15cm] (1.4,0) to [out=90,in=-90] (.7,.7) to [out=90,in=-90] (0,1.4) -- (0,2.1);
            \draw [very thick] (1.4,0) to [out=90,in=-90] (.7,.7) to [out=90,in=-90] (0,1.4) -- (0,2.1);
  \end{tikzpicture}
  \quad,\quad
    \begin{tikzpicture}[anchorbase,scale=.5]
    \draw [very thick] (.5,0) to  (.5,.7) to [out=45,in=-90] (1,1.4) -- (1,2.1);
       \draw [very thick] (.5,.7) to [out=135,in=-90] (0,1.4) -- (0,2.1);
       \draw [white, line width=.15cm] (-.4,1.4) to  (1.4,1.4);
    \draw [very thick] (-.4,1.4) to  (1.4,1.4);
  \end{tikzpicture}
  \quad =\quad
  \begin{tikzpicture}[anchorbase,scale=.5]
    \draw [very thick] (.5,0) to  (.5,1.4) to [out=45,in=-90] (1,2.1);
       \draw [very thick] (.5,1.4) to [out=135,in=-90] (0,2.1);
    \draw [white, line width=.15cm] (-.4,0.7) to  (1.4,0.7);
    \draw [very thick] (-.4,0.7) to  (1.4,0.7);
  \end{tikzpicture}
  \end{gather*}
  We will refer to the last relation as a \textit{forkslide move}.
\end{lemma}
\begin{proof} The Reidemeister II, III and forkslide moves follow from the property of a braiding, and our braiding convention is only a minor rescaling of the one in \cite[Corollary 6.2.3]{CKM}, see also \cite[Section 2.4]{TVW}. The Reidemeister I moves can be verified inductively as in \cite[Lemma 2.9]{Queff_aff}.
\end{proof}
  
\begin{definition}
We denote by $\Webp[N]$ the full subcategory of $\Web[N]$ with objects given by arbitrary sequences with exclusively upward pointing orientations. 
\end{definition}
In the following we will use the same superscript $+$ to indicate analogous full subcategories of other web categories, consisting of those objects with upward (or outward) pointing orientations. The next lemma is a well-known consequence the proof of Theorem~\ref{thm:CKM} using quantum skew Howe duality, see \cite{CKM,TVW}.

\begin{lemma}\label{lem:upward} The morphism spaces of $\Webp[N]$ are spanned by upward-pointing webs, i.e. webs whose edges have no horizontal tangent vectors.
\end{lemma} 
  
In the following, we will consider \emph{skein modules} of isotopy classes of webs embedded in different surfaces, modulo the local relations from \eqref{eq:webrel}, and we will also vary the ground ring. In the following sections we deal with webs over $\C$, whose defining relations are obtained from \eqref{eq:webrel} by specializing $q=1$. We indicate categories of webs at $q=1$ by the omission of the $q$-subscript, e.g. $\Web[N]$ instead of $\Webq[N]$. The functor $\phi$ also specializes to $q=1$ and then relates $\Web[N]$ to the symmetric monoidal category of $U(\glnn{N})$-representations.

Note that setting $q=1$ identifies the evaluation of positive and negative crossings in terms of webs in \eqref{eq:crossing}, and so we sometimes omit to display any over- or under-crossing information in graphics. In particular, the braid group action induced by 1-labeled crossings becomes a symmetric group action.

\subsection{Affinization at \texorpdfstring{$q=1$}{q=1}}
In \cite{QW}, we considered an affine extension of the Temperley-Lieb category, and extended an analog of the functor $\phi$ to this more general category. Just as in this simpler $\slnn{2}$ case, we will consider a more general affine web category that will give us the freedom to extend the diagrammatic presentation of the representation category of $U(\glnn{N})$ to a Cartan subalgebra.

We define the category $\AWeb[N]$ to be the $\C$-linear category of webs drawn on an annulus subject to the same local relations as in $\Web[N]$, i.e. relations \eqref{eq:webrel} at $q=1$. It is easy to see that the morphisms of $\AWeb[N]$ can be obtained by gluing the strip in which diagrams in $\Web[N]$ live into an annulus, and adding new \textit{wrapping morphisms}. 

\[
\wrap 
\;\;=\;\;
\begin{tikzpicture}[anchorbase, scale=.3]
\draw (0,0) circle (1);
\draw (0,0) circle (3);
\draw [very thick] (-.8,.6) .. controls (-2.6,1.2) and (-2,-2) .. (0,-2) .. controls (2,-2) and (1.2,.9) ..  (2.4,1.8);
\draw [very thick] (-.4,.9) .. controls (-.8,1.8) and (-.9,1.2) .. (-1.8,2.4);
\node at (0,2) {$\dots$};
\draw [very thick] (.8,.6) .. controls (1.2,.9) and (.75,1.3) .. (1.5,2.6);
\node at (0,-1) {$*$};
\node at (0,-3) {$*$};
\draw [dashed] (0,-1) to (0,-3);
\end{tikzpicture}
\quad,\quad
\wrapi
\;\;=\;\;
\begin{tikzpicture}[anchorbase, scale=.3]
\draw (0,0) circle (1);
\draw (0,0) circle (3);
\draw [very thick] (.8,.6) .. controls (2.6,1.2) and (2,-2) .. (0,-2) .. controls (-2,-2) and (-1.2,.9) ..  (-2.4,1.8);
\draw [very thick] (.4,.9) .. controls (.8,1.8) and (.9,1.2) .. (1.8,2.4);
\node at (0,2) {$\dots$};
\draw [very thick] (-.8,.6) .. controls (-1.2,.9) and (-.75,1.3) .. (-1.5,2.6);
\node at (0,-1) {$*$};
\node at (0,-3) {$*$};
\draw [dashed] (0,-1) to (0,-3);
\end{tikzpicture}
\]

Note that we remember where the ends of the strip have been glued by drawing a dashed line between two base points on the boundary components of the annulus. We denote this segment by $\alpha$. We also stress that in $\AWeb[N]$ webs can come with any orientation on the boundary. 

\subsection{Link with representation theory}
We will extend the domain of the functor $\phi$ from $\Web[N]$ to $\AWeb[N]$ by sending the wrapping morphisms to maps between $U(\glnn{N})$-representations, which respect the weight space decomposition but break the $U(\glnn{N})$-action. This will allow us to build new diagrammatic projectors, and we will now explain how to choose this preferred extension. Note that other extensions via evaluation representations would allow us to preserve the $U(\glnn{N})$ action, but this would not serve our purpose of finding a diagrammatic presentation of the representation category of a Cartan subalgebra.

We first consider a single counterclockwise wrap morphism $\wrap=\wrap_1$ of a single 1-labeled outward pointing strand.\vspace{-3mm}
\[
\begin{tikzpicture}[anchorbase, scale=.3]
\draw (0,0) circle (1);
\draw (0,0) circle (3);
\draw [very thick,->] (0,1) .. controls (0,1.7) and (-1.6,1) .. (-1.6,0) to [out=-90,in=180] (0,-1.8) to [out=0,in=-90] (2.1,0) to [out=90,in=-90] (0,3);
\node at (-.2,2) {\tiny$1$};
\node at (0,-3) {$*$};
\draw [dashed] (0,-1) to (0,-3);
\end{tikzpicture}
\]
The requirement that $\phi(\wrap)$ respects the weight space decomposition of $V$ implies that $\phi(\wrap)(v_k)= \gamma_k v_k$ for some $\gamma_k\in \C$ and the desired invariance under ambient isotopy forces these scalars to be invertible. In fact, this choice of scalars determines the action of $\phi(\wrap_k)$, the $k$-labeled version of the wrap: by opening a $k$-blister and sliding one vertex around the wrap $\wrap_k$, the eigenvalues of $\phi(\wrap_k)$ can be seen to be $k$-fold products of the eigenvalues of $\phi(\wrap)$: $\phi(\wrap_k)(v_S)= (\prod_{i\in S}\gamma_i) v_S$. Furthermore, inverse wraps have inverse eigenvalues: $\phi(\wrap_k^{-1})(v_S)= (\prod_{i\in S}\gamma_i)^{-1} v_S$. Next, we would like to have relations of the form:\vspace{-3mm}

\begin{equation}
\label{eq:capslide}
\phi
\left(
\begin{tikzpicture}[anchorbase, scale=.3]
\draw (0,0) circle (1);
\draw (0,0) circle (3);
\draw [very thick] (-.8,.6) to [out=150,in=90] (-1.75,0) to[out=270,in=180] (0,-2) to [out=0,in=270] (2,0) to [out=90,in
=225] (2.25,.9) to [out=45,in=315] (2.25,1.35)to[out=135,in=45] (1.8,1.35) to (.8,.6) ;
\draw [very thick] (-.4,.9)to  (-1.2,2.7);
\node at (0,2) {$\dots$};
\draw [very thick] (.4,.9)  to (1.2,2.7);
\node at (0,-1) {$*$};
\node at (0,-3) {$*$};
\draw [dashed] (0,-1) to (0,-3);
\end{tikzpicture}
\right)
\;=\;
\phi\left(
\begin{tikzpicture}[anchorbase, scale=.3]
\draw (0,0) circle (1);
\draw (0,0) circle (3);
\draw [very thick] (.8,.6) to [out=30,in=90] (1.75,0) to[out=270,in=0] (0,-2) to [out=180,in=270] (-2,0) to [out=90,in
=315] (-2.25,.9) to [out=135,in=225] (-2.25,1.35)to[out=45,in=135] (-1.8,1.35) to (-.8,.6) ;
\draw [very thick] (-.4,.9)to  (-1.2,2.7);
\node at (0,2) {$\dots$};
\draw [very thick] (.4,.9)  to (1.2,2.7);
\node at (0,-1) {$*$};
\node at (0,-3) {$*$};
\draw [dashed] (0,-1) to (0,-3);
\end{tikzpicture}
\right)
\end{equation}
To ensure that $\phi$ respects such isotopy relations for sliding cups and caps around the annulus, we need to have
$\phi(\wrap_k)(v_S^*)= (\prod_{i\in S}\gamma_i)^{-1} v_S^*$ and $\phi(\wrap_k^{-1})(v_S^*)= (\prod_{i\in S}\gamma_i) v_S^*$, which determine the maps assigned to inward pointing versions of $D$ and $D^{-1}$.

In order to be able to project onto the 1-dimensional spaces spanned by specific standard basis vectors in $V$, we would like $\phi(\wrap)$ to have distinct eigenvalues on the $v_k$. Furthermore, we would like to find a set of diagrammatic relations in the annular web category that enforces a choice of $\phi(\wrap)$ with distinct eigenvalues, or in other words, with a separable characteristic polynomial $\prod_{i=1}^N(X-\gamma_i)= \sum_{k=0}^N X^{N-k} (-1)^{k} e_k(\vec{\gamma})$. Here $e_k(\vec{\gamma})=e_k(\{\gamma_1,\dots,\gamma_N\})$ denotes the $k$-th elementary symmetric polynomial evaluated at the complex numbers $\gamma_1,\dots,\gamma_N$.

\begin{lemma} The coefficients of the characteristic polynomial of $\phi(\wrap)$ are determined by the image of $\phi$ on essential circles in the annulus. More precisely:

 \begin{equation}
\label{eqn:essentialideal}
  \phi\left(
  \begin{tikzpicture}[anchorbase, scale=.3]
    \draw (0,0) circle (1);
    \draw (0,0) circle (3);
    \draw [very thick, directed=.55] (0,0) circle (2);
    \node at (0,-1) {$*$};
    \node at (-2.5,0) {\tiny$k$};
    \node at (0,-3) {$*$};
    \draw [dashed] (0,-1) to (0,-3);
  \end{tikzpicture}
  \right)
  = \sum_{|S|=k} (\prod_{i\in S} \gamma_i) \id_\C = e_k(\vec{\gamma})
\end{equation}
\end{lemma}

\begin{proof} The morphism can be written as the composition of a $k$-cup, a $k$-wrap and a $k$-cap. The sum $\sum_{|S|=k}$ comes from the cup and the factors from the action of the wrap on $v_S$.
\end{proof}

We now prescribe $\phi(\wrap)$ to have the separable characteristic polynomial $X^N-1$, and we may index the roots as $\gamma_k=\zeta^k=e^{k 2 \pi i /N}$. This choice of relation is homogeneous with respect to a $\Z/N\/Z$-grading by winding number, see \ref{def:grading}. We now extend the definition of $\phi$ to the new generators $\wrap$ in the general case, that is, allowing more than one strand.
\begin{definition}
  Let $V\otimes W$ be the image under $\phi$ of the domain of $\wrap$ and $W\otimes V$ its co-domain. Then we define $\phi(\wrap)$ to be the linear map determined by $v_{k}\otimes w\mapsto \zeta^{k} w\otimes v_k
  $
  for $v_{k}\in V$ and any $w\in W$. Furthermore we set $\phi(\wrapi)=\phi(\wrap)^{-1}$ and analogously for the duals.
\end{definition}

Let $\h$ denote a Cartan subalgebra in $\glnn{N}$ and consider $U(\h)=\langle L_1^{\pm 1},\dots,L_N^{\pm 1} \rangle \subset U(\glnn{N})$. We denote by $\reph$ the category of finite-dimensional $U(\h)$-representations of integral weights. Note that the inclusion $\h\hookrightarrow \glnn{N}$ induces a restriction functor $\Rep(\glnn{N})\to\reph$ and that $\phi(\wrap)$ and $\phi(\wrapi)$ are morphisms in $\reph$.

\begin{lemma}\label{lem:welldef} The functor $\phi\colon \AWeb[N] \to \reph$ is well-defined. 
\end{lemma}
\begin{proof} All morphisms in $\AWeb[N]$ are compositions of caps or cups between adjacent strands, vertices, as well as the morphisms $\wrap$ and $\wrapi$. Any relation satisfied by compositions of these generating morphisms is either supported in $\Web[N]\subset \AWeb[N]$ (and is thus respected by $\phi$) or involves some generators $\wrap$ and $\wrapi$. Since $\phi$ maps $\wrap$ and $\wrapi$ to inverse isomorphisms, it suffices to check that $\phi$ respects the isotopy relations that caps, cups and vertices can be slid along wraps around the annulus. However, the images under $\phi$ of the wrap morphisms were precisely chosen for these relations to hold. 
\end{proof}

\subsection{The tensor product on annular webs}
Let $\otimes\colon \AWeb[N]\times \AWeb[N]\to \AWeb[N]$ denote the bi-functor given on objects by $(m,n)\mapsto m+n$ and on morphisms by superposing a pair of annular webs $(W_1,W_2)$:

\[
 \begin{tikzpicture}[anchorbase, scale=.25]
\draw[thick] (0,0) circle (2.5);
\fill[black,opacity=.2] (0,0) circle (2.5);
\draw[thick,fill=white] (0,0) circle (1.5);
\draw[thick] (0,0) circle (4.5);
\fill[black,opacity=.2] (4.5,0) arc (0:360:4.5) -- (3.5,0) arc (360:0:3.5);
\draw [thick] (0,0) circle (3.5);
\draw (0,0) circle (1);
\draw (0,0) circle (5);
\draw[dotted] (-2.29,2.29) to [out=225,in=90] (-3.25,0) to [out=270,in=135] (-2.29,-2.29);
\draw[dotted] (-3.4,3.4) to [out=225,in=90] (-4.75,0) to [out=270,in=135] (-3.4,-3.4);
\draw[dotted] (-1.93,1.93) to [out=225,in=90] (-2.75,0) to [out=270,in=135] (-1.93,-1.93);
\draw[dotted] (-0.88,0.88) to [out=225,in=90] (-1.25,0) to [out=270,in=135] (-0.88,-0.88);
\draw [white,line width=.15cm] (-.6,.8) to (-1.8,2.4);
\draw [white,line width=.15cm] (-.6,-.8) to (-1.8,-2.4);
\draw [very thick] (-.6,.8) to (-2.1,2.8);
\draw [very thick] (-.6,-.8) to (-2.1,-2.8);
\draw [very thick] (-2.7,3.6) to (-3,4);
\draw [very thick] (-2.7,-3.6) to (-3,-4);
\draw[dotted] (2.29,2.29) to [out=315,in=90] (3.25,0) to [out=270,in=45] (2.29,-2.29);
\draw[dotted] (3.4,3.4) to [out=315,in=90] (4.75,0) to [out=270,in=45] (3.4,-3.4);
\draw[dotted] (1.93,1.93) to [out=315,in=90] (2.75,0) to [out=270,in=45] (1.93,-1.93);
\draw[dotted] (0.88,0.88) to [out=315,in=90] (1.25,0) to [out=270,in=45] (0.88,-0.88);
\draw [very thick] (1.5,2) to (1.98,2.64);
\draw [very thick] (2.79,3.72) to (3,4);
\draw [very thick] (1.5,-2) to (1.98,-2.64);
\draw [very thick] (2.79,-3.72) to (3,-4);
\draw [very thick] (.6,-.8) to (.9,-1.2);
\draw [very thick] (.6,.8) to (.9,1.2);
\node at (0,-1) {$*$};
\node at (0,-5) {$*$};
\draw [dashed] (0,-1) to (0,-5);
\node at (0,1.95) {\tiny$W_2$};
\node at (0,3.95) {\tiny$W_1$};
\end{tikzpicture}
\]
and resolving all crossings via \eqref{eq:crossing}. This is well-defined thanks to Lemma \ref{lem:Reid} and it induces a symmetric monoidal structure on $\AWeb[N]$ such that both the inclusion $\Web[N]\to \AWeb[N]$ and  $\phi \colon \AWeb[N] \to \reph$ become symmetric monoidal functors. Here we have used that $q=1$.

\begin{definition} \label{def:si} For $m\geq 0$ we denote by $\AWeb[N](m)$ the endomorphism algebra in $\AWeb[N]$ of the object consisting of a sequence of $m$ points with label $1$ and outward orientation. We denote by $s_i$ for $1\leq i \leq m$ the element of this endomorphism algebra that is given by the crossing between the strands in positions $i$ and $i+1$, with positions understood modulo $m$. We also write $u_i=\id_2-s_i$ for the corresponding \textit{dumbbell web}.

More generally, for two objects $\vec{k}$ and $\vec{l}$ we use the shorthand $\AWeb[N](\vec{k},\vec{l}):= \Hom_{\AWeb[N]}(\vec{k},\vec{l})$. We will also use these notation conventions in other web categories. 
\end{definition}

The following lemma will allow us to freely express webs in terms of images of $1$-labeled tangles, which will be very useful in a number of proofs. 

\begin{lemma} \label{lem:resolveastangles}Every element of $\AWeb[N](m)$ for $m\geq 0$ can be written as a $\C$-linear combination of $1$-labeled annular tangles.
\end{lemma}
As we will see later, we may assume that the closed components are essential circles (possibly carrying higher labels) and the non-closed components are oriented $1$-labeled arcs from the inner to the outer boundary circle, that are everywhere outward pointing. 
\begin{proof} We include this classical proof for completeness. It is well-known that every closed web can be written as a $\C$-linear combination of webs $W$ with only $1$-labeled edges, interacting at most in $2$-labeled dumbbells, see e.g. \cite[Proof of Lemma 4.1]{TVW}. Indeed, the argument can be inductively built from the following two operations:
\[
\begin{tikzpicture}[anchorbase]
  \draw [very thick] (0,0) -- (0,2);
  \node at (.3,1) {\tiny $k$};
\end{tikzpicture}
\quad = \frac{1}{k!}
\begin{tikzpicture}[anchorbase]
  \draw [very thick] (0,0) -- (0,.5);
  \draw [very thick] (0,.5) .. controls (-.5,.75) and (-.5,1.25) .. (0,1.5);
  \draw [very thick] (0,.5) .. controls (-.3,.75) and (-.3,1.25) .. (0,1.5);
  \node at (.15,1) {\small $\dots$};
  \draw [very thick] (0,.5) .. controls (.5,.75) and (.5,1.25) .. (0,1.5);
  \draw [very thick] (0,1.5) --(0,2);
\end{tikzpicture}
\quad \text{and}\quad
\begin{tikzpicture}[anchorbase]
  \draw [very thick] (-.5,0) .. controls (-.5,.25) .. (0,.5);
  \draw [very thick] (-.3,0) .. controls (-.3,.25) .. (0,.5);
  \node at (.1, .05) {\small $\dots$};
  \draw [very thick] (.5,0) .. controls (.5,.25) .. (0,.5);
  \draw [very thick] (0,.5) -- (0,1);
  \draw [very thick] (0,1) .. controls (-.5,1.25) .. (-.5,1.5);
  \node at (-.1,1.45) {\small $\dots$};
  \draw [very thick] (0,1) .. controls (.3,1.25) .. (.3,1.5);
  \draw [very thick] (0,1) .. controls (.5,1.25) .. (.5,1.5);
  \draw[decorate,decoration={brace}] (.5,-.1) -- node[midway,font=\small,yshift=-.25cm] {$k$} (-.5,-.1);
  \draw[decorate,decoration={brace}] (-.5,1.6) -- node[midway,font=\small,yshift=.25cm] {$k$} (.5,1.6);
\end{tikzpicture}
\quad = \;(-1)^k\;\;
\begin{tikzpicture}[anchorbase]
  \draw [very thick] (-.5,0) .. controls (-.5,.75) and (.5,.75) .. (.5,1.5);
  \draw [very thick] (-.3,0) .. controls (-.3,.25) .. (.1,.5);
  \node at (.1, .05) {\small $\dots$};
  \draw [very thick] (.5,0) .. controls (.5,.25) .. (.1,.5);
  \draw [very thick] (.1,.5) .. controls (.1,.75) and (-.1,.75) .. (-.1,1);
  \draw [very thick] (-.1,1) .. controls (-.5,1.25) .. (-.5,1.5);
  \node at (-.1,1.45) {\small $\dots$};
  \draw [very thick] (-.1,1) .. controls (.3,1.25) .. (.3,1.5);
  \draw[decorate,decoration={brace}] (.5,-.1) -- node[midway,font=\small,yshift=-.25cm] {$k-1$} (-.3,-.1);
  \draw[decorate,decoration={brace}] (-.5,1.6) -- node[midway,font=\small,yshift=.25cm] {$k-1$} (.3,1.6);
\end{tikzpicture}
\quad + \; (-1)^{k+1}\;\;
\begin{tikzpicture}[anchorbase]
  \draw [very thick] (-.5,0) .. controls (-.5,.25) and (-.2,.5) .. (-.2,.85);
  \draw [very thick] (-.2,.85) .. controls (-.2,.9) and (-.1,.9) .. (-.1,1);
  \draw [very thick] (-.3,0) .. controls (-.3,.25) .. (.1,.5);
  \node at (.1, .05) {\small $\dots$};
  \draw [very thick] (.5,0) .. controls (.5,.25) .. (.1,.5);
  \draw [very thick] (.1,.5) .. controls (.1,.6) and (.2,.6) .. (.2,.65);
  \draw [very thick] (.2,.65) .. controls (.1,.75) and (-.1,.75) .. (-.2,.85);
  \draw [very thick] (.2,.65) .. controls (.2,1) and (.5,1.25) .. (.5,1.5);
  \draw [very thick] (-.1,1) .. controls (-.5,1.25) .. (-.5,1.5);
  \node at (-.1,1.45) {\small $\dots$};
  \draw [very thick] (-.1,1) .. controls (.3,1.25) .. (.3,1.5);
  \draw[decorate,decoration={brace}] (.5,-.1) -- node[midway,font=\small,yshift=-.25cm] {$k-1$} (-.3,-.1);
  \draw[decorate,decoration={brace}] (-.5,1.6) -- node[midway,font=\small,yshift=.25cm] {$k-1$} (.3,1.6);
\end{tikzpicture}
\] 
The remaining $2$-labeled dumbbells can now be expanded in terms of crossings and their oriented resolutions, resulting in a linear combination of $1$-labeled annular tangles.
\end{proof}
An analogous result is true at generic $q$ over $\C(q)$.

\begin{lemma}\label{lem:endempty} The endomorphism algebra of the empty object in $\AWeb[N]$ is isomorphic to $\C[c_1,\dots, c_{N-1}, c_N^{\pm 1}]$, where $c_i$ denotes the counter-clockwise oriented $i$-labeled circle and $c_N^{-1}$ the $N$-labeled clockwise oriented circle.
\end{lemma}
\begin{proof}
The proof proceeds in two steps. First we show that any closed web in the annulus can be written as a $\C$-linear combination of collections of essential circles from the set $\{c_1,c_2, \dots, c_N, c_N^{-1}\}$. The essential circles commute, which is easily seen by applying Reidemeister II moves, and $c_N$ and $c_N^{-1}$ are indeed mutually inverse. Second, we check that the counter-clockwise oriented circles are algebraically independent. 

By Lemma~\ref{lem:resolveastangles} any closed web may be expressed as a linear combination of annular links. Moreover, we may assume that every link component wraps exclusively clockwise or counter-clockwise around the annulus. The annular evaluation algorithm from \cite{QR2} implies that any such component can be resolved in terms of concentric circles of various labels. Hence, the same is true for the original web. To finish the first part, note that the clockwise-oriented circle of label $i$ is equal to $c_{N-i}c_N^{-1}$. This follows from the last relation in \eqref{eq:webrel2}. 

In order to prove algebraic independence, we use an extension of the functor $\phi$ from Lemma~\ref{lem:welldef}. Let $R=\C[\X^{\pm 1} ]=\C[X_1^{\pm 1},\dots, X_N^{\pm 1}]$ be a Laurent polynomial ring in $N$ variables and consider the category $R\reph$ obtained by tensoring all morphism spaces in $\reph$ by $R$. Then $\phi_R\colon \AWeb[N] \to R\reph$ can be defined just as $\phi$ was above, except that the $X_i$ now play the role of the eigenvalues of the $1$-labeled counter-clockwise wrap: $\phi_R(D)(v_k):=X_k v_k$. 

For a closed annular web $W$, it follows that $\phi(W)\in R$. In particular, counter-clockwise oriented circles $c_i$ evaluate to elementary symmetric polynomials $e_i(\X)$ in the variables $X_k$ and their clockwise counterparts evaluate to $e_i(\X^{-1})$. As a consequence of the first part of this proof, the evaluation of closed annular webs $\phi_R(W)$ takes values in the symmetric part $R^{\mathfrak{S}_N} \cong \C[e_1(\X),\dots, e_{N-1}(\X), e_{N}(\X)^{\pm 1}]$ of $R$. Now the algebraic independence of the $c_i$ follows from the algebraic independence of their images $e_i(\X)$ under $\phi_R$.
\end{proof}

An analogous version of this result holds for $\AWebq[N]$ over $\C(q)$. The second part in its $q=1$ version is already sufficient to establish algebraic independence of the counter-clockwise essential circles in $\AWebq[N]$.

Just as in the non-annular case, we denote by $\AWebp[N]$ the full subcategory of $\AWeb[N]$ with objects given by all upward, or \emph{outward}, pointing boundary sequences. The following is the analog of Lemma~\ref{lem:upward}.

\begin{lemma}\label{lem:upwardannulus} The morphism spaces of $\AWebp[N]$ are spanned by webs with all edges outward oriented, but potentially superposed with essential circles.
\end{lemma}

Note that Lemma \ref{lem:endempty} allows us to restrict to counter-clockwise essential circles , except for the $N$-labeled ones.

\begin{proof}
 We will prove the claim for webs $W$ whose source and target objects are only $1$-labeled. The general claim follows by a usual merging argument.

Lemma~\ref{lem:resolveastangles} allows us to write $W$ as a linear combination of annular tangles. The closed components of these tangles evaluate to essential circles thanks to Lemma~\ref{lem:endempty}, while the non-closed components can be isotoped to be outward oriented arcs. (This shortcut is possible since we work at $q=1$.) The superposition of such outward arcs is itself a linear combination of outward webs, as desired. 
\end{proof}

\subsection{Equivalences between blocks}\label{sec:lambdaequiv}

The categories $\Web[N]$, $\AWeb[N]$ and all of their specializations and quotients considered in the following decompose into blocks (i.e. full subcategories) indexed by $m\in \Z$, which consist of those objects, whose signed sum of labels equals $m$. Here we count upward oriented boundary points positively, and downward pointing ones negatively. We indicate such blocks by the subscript $m$, e.g. $\Web[N]_{m}$.

We will use the notation $\lambda$ for the endofunctors of these categories that act on objects by tensoring with a single $N$-labeled upward boundary point and on morphisms by superposing with an upward-oriented $N$-labeled edge. We denote by $\lambda^*$ the analogous operation with downward orientations. We give an example for $\AWeb[N]$:
\[
\begin{tikzpicture}[anchorbase, scale=.4]
\draw[thick] (0,0) circle (2.5);
\fill[black,opacity=.2] (0,0) circle (2.5);
\draw[thick,fill=white] (0,0) circle (1.5);
\draw[dotted] (-1.93,1.93) to [out=45,in=180] (0,2.75) to [out=0,in=135] (1.93,1.93);
\draw[dotted] (-0.88,0.88) to [out=45,in=180] (0,1.25) to [out=0,in=135] (0.88,0.88);
\draw (0,0) circle (1);
\draw (0,0) circle (3);
\draw [very thick] (.8,.6) to (1.2,.9);
\draw [very thick] (-.8,.6) to (-1.2,.9);
\draw [very thick] (2,1.5) to (2.4,1.8);
\draw [very thick] (-2,1.5) to (-2.4,1.8);
\node at (0,-1) {$*$};
\node at (0,-3) {$*$};
\draw [dashed] (0,-1) to (0,-3);
\end{tikzpicture}
\quad
\xrightarrow{\lambda}
\quad 
 \begin{tikzpicture}[anchorbase, scale=.4]
\draw[thick] (0,0) circle (2.5);
\fill[black,opacity=.2] (0,0) circle (2.5);
\draw[thick,fill=white] (0,0) circle (1.5);
\draw[dotted] (-1.93,1.93) to [out=45,in=180] (0,2.75) to [out=0,in=135] (1.93,1.93);
\draw[dotted] (-0.88,0.88) to [out=45,in=180] (0,1.25) to [out=0,in=135] (0.88,0.88);
\draw [very thick] (.8,.6) to (1.2,.9);
\draw [very thick] (-.8,.6) to (-1.2,.9);
\draw [very thick] (2,1.5) to (2.4,1.8);
\draw [very thick] (-2,1.5) to (-2.4,1.8);
\draw [white, line width=.12cm] (1,0) to (3,0);
\draw [very thick,->] (1,0) to (3,0);
\node at (3.3,0) {\tiny $N$};
\node at (0,-1) {$*$};
\node at (0,-3) {$*$};
\draw [dashed] (0,-1) to (0,-3);
\draw (0,0) circle (1);
\draw (0,0) circle (3);
\end{tikzpicture}
\quad
\xrightarrow{\lambda^*}
\quad 
 \begin{tikzpicture}[anchorbase, scale=.4]
\draw[thick] (0,0) circle (2.5);
\fill[black,opacity=.2] (0,0) circle (2.5);
\draw[thick,fill=white] (0,0) circle (1.5);
\draw[dotted] (-1.93,1.93) to [out=45,in=180] (0,2.75) to [out=0,in=135] (1.93,1.93);
\draw[dotted] (-0.88,0.88) to [out=45,in=180] (0,1.25) to [out=0,in=135] (0.88,0.88);
\draw [very thick] (.8,.6) to (1.2,.9);
\draw [very thick] (-.8,.6) to (-1.2,.9);
\draw [very thick] (2,1.5) to (2.4,1.8);
\draw [very thick] (-2,1.5) to (-2.4,1.8);
\draw [white, line width=.12cm] (1,0) to (3,0);
\draw [very thick,->] (1,0) to (3,0); 
\node at (3.3,0) {\tiny $N$};
\draw [white, line width=.12cm] (.8,-.6) to (2.4,-1.8);
\draw [very thick,<-] (.8,-.6) to (2.4,-1.8);
\node at (2.64,-1.98) {\tiny $N$};
\node at (0,-1) {$*$};
\node at (0,-3) {$*$};
\draw [dashed] (0,-1) to (0,-3);
\draw (0,0) circle (1);
\draw (0,0) circle (3);
\end{tikzpicture}
\] 
Here we display the $N$-labeled strand as crossing over the remaining web for better visibility, even though this has no significance at $q=1$.

\begin{lemma} The endofunctors $\lambda$ and $\lambda^*$ restrict to mutually quasi-inverse equivalences between the blocks $\AWeb[N]_m$ and $\AWeb[N]_{m+N}$.
\end{lemma}
\begin{proof}
Using isotopy relations, it is easy to see that the following types of webs provide natural isomorphisms between the identity functor on $\AWeb[2]_m$ and the endofunctor given by the composition $\lambda^*\circ \lambda$:
\[ \begin{tikzpicture}[anchorbase, scale=.4]
\draw[dotted] (-1.93,1.93) to [out=45,in=180] (0,2.75) to [out=0,in=135] (1.93,1.93);
\draw[dotted] (-0.88,0.88) to [out=45,in=180] (0,1.25) to [out=0,in=135] (0.88,0.88);
\draw [very thick] (.8,.6) to (2.4,1.8);
\draw [very thick] (-.8,.6) to (-2.4,1.8);
\draw [white, line width=.12cm] (2.4,-1.8) to (1.2,-.9) to  [out=135, in=180] (1.5,0)to (3,0);
\draw [very thick,->] (2.4,-1.8) to (1.2,-.9) to  [out=135, in=180] (1.5,0)to (3,0);
\node at (3.3,0) {\tiny $N$};
\node at (0,-1) {$*$};
\node at (0,-3) {$*$};
\draw [dashed] (0,-1) to (0,-3);
\draw (0,0) circle (1);
\draw (0,0) circle (3);
\end{tikzpicture}
\] 
Analogously, there are natural transformations between $\lambda \circ \lambda^*$ and the identity functor on $\AWeb[2]_{m}$. 
\end{proof}

\subsection{The quotients by essential circles}
It is a key observation that the functor $\phi \colon \AWeb[N] \to \reph$ is not faithful.
\begin{proposition}\label{prop:nff}~\vspace{-.5cm}
 \begin{equation}
\label{eqn:essentialidealtwo}
  \phi\left(
  \begin{tikzpicture}[anchorbase, scale=.3]
    \draw (0,0) circle (1);
    \draw (0,0) circle (3);
    \draw [very thick,directed=.55] (0,0) circle (2);
  \draw [very thick] (-.9,.4) to (-2.8,1.2);
    \node at (0,1.35) {$\cdots$};
    \node at (0,2.5) {$\cdots$};
    \draw [very thick] (.9,.4) to (2.8,1.2);
    \node at (0,-1) {$*$};
    \node at (0,-3) {$*$};
    \draw [dashed] (0,-1) to (0,-3);
    \node at (-2.5,0) {\tiny $k$};
  \end{tikzpicture}
  \right)
  =
   \phi\left(
  \begin{tikzpicture}[anchorbase, scale=.3]
    \draw (0,0) circle (1);
    \draw (0,0) circle (3);
    \draw [very thick,rdirected=.55] (0,0) circle (2);
  \draw [very thick] (-.9,.4) to (-2.8,1.2);
    \node at (0,1.35) {$\cdots$};
    \node at (0,2.5) {$\cdots$};
    \draw [very thick] (.9,.4) to (2.8,1.2);
    \node at (0,-1) {$*$};
    \node at (0,-3) {$*$};
    \draw [dashed] (0,-1) to (0,-3);
    \node at (-2.5,0) {\tiny $k$};
  \end{tikzpicture}
  \right)
  =\begin{cases} (-1)^{N-1} \id \quad \text{ if } k=N\\
  0 \quad\quad\quad\quad \text{ otherwise }
  \end{cases}.
\end{equation}
\end{proposition}
\begin{proof} This follows from \eqref{eqn:essentialideal}.
\end{proof}

\begin{definition} We let $\essbAWeb[N]$ denote the quotient of $\AWeb[N]$ by the ideal generated by the relations for $k<N$ shown in \eqref{eqn:essentialideal}, with any number of through-strings. We define $\essAWeb[N]$ to be quotient of $\AWeb[N]$ by the ideal generated by all relations in \eqref{eqn:essentialideal}. 
\end{definition} 

Note that the monoidal structure $\otimes$ and the functor $\phi$ descend to the quotients $\essbAWeb[N]$ and $\essAWeb[N]$. One of the key results of this paper will be to prove that this category $\essAWeb[N]$ is equivalent to the full subcategory of $\reph$ generated by the objects of $\Rep(\glnn{N})$.

The category $\essbAWeb[N]$ is a central extension of $\essAWeb[N]$ by the invertible element $c_N$ given by $N$-labeled counter-clockwise oriented essential circle, which we denote by $c_N$. Conversely, $\essAWeb[N]$ is obtained from $\essbAWeb[N]$ by setting $c_N=(-1)^{N-1}$.

\begin{definition}\label{def:grading} Let $W$ be a web in $\AWeb[N]$. Then the \emph{flow winding number} $w_f(W)$ of $W$ is given by the algebraic intersection number of the web with the segment $\alpha$ (assuming no trivalent vertex occurs on it), where $k$-labeled edges crossing $\alpha$ count as $\pm k$.
\end{definition}
It is clear that all web relations in $\AWeb[N]$ and $\essbAWeb[N]$ preserve the flow winding number. This $\Z$-grading on the morphism spaces of these categories descends to a $\Z/N\Z$ grading on $\essAWeb[N]$.

The following corollaries are implied by Lemma~\ref{lem:endempty} and Lemma~\ref{lem:upwardannulus}.

\begin{corollary}\label{cor:end} The endomorphism algebra of the empty object in $\essAWeb[N]$ is $\C$ and in $\essbAWeb[N]$ it is isomorphic to the Laurent polynomial ring over $\C$ generated by an essential $N$-labeled circle. 
\end{corollary}
\comm{
\begin{proof}
It is not hard to see that every closed web in the annulus can be written in terms of concentric essential circles using relations in $\AWeb[N]$. 
All $k$-labeled essential circles for $0<k<N$ are set to zero. The remaining $N$-labeled circles of opposite orientation are inverse to each other. These elements are non-zero, since they act non-trivially under the functor to representation theory. Finally, integer powers of $N$-labeled circles are linearly independent because they are of different flow winding number.
\end{proof}
}

\begin{corollary}\label{cor:out} The morphism spaces of $\essAWebp[N]$ are spanned by outward pointing webs. The same is true in $\essbAWebp[N]$ up to superposition with an integer power of the $N$-labeled essential circle $c_N$.
\end{corollary}

\begin{lemma} In $\essbAWeb[N](1)$ we have $\wrap^N=(-1)^{N-1}c_N$. In $\essAWeb[N](1)$ this specializes to $\wrap^N = \id_1$.
\label{lem:DNone}
\end{lemma}
\begin{proof} We compute:
 \begin{align*} 
 (-1)^{N-1} c_N
  \; &= (-1)^{N-1}\; \begin{tikzpicture}[anchorbase, scale=.3]
    \draw (0,0) circle (1);
    \draw (0,0) circle (3);
    \draw [very thick,->] (0,1) to (0,3);
    \draw [very thick,directed=.55] (0,0) circle (2);
    \node at (0,-1) {$*$};
    \node at (0,-3) {$*$};
        \node at (0.5,-2.4) {\tiny $N$};
    \draw [dashed] (0,-1) to (0,-3);
  \end{tikzpicture}
  \; = \; \begin{tikzpicture}[anchorbase, scale=.3]
    \draw (0,0) circle (1);
    \draw (0,0) circle (3);
    \draw [very thick,directed=.65] (0,1) to [out=90,in=270] (-.8,1.8);
     \draw [very thick,->] (.8,1.8) to [out=90,in=270] (0,2.8) to (0,3);
    \draw [very thick,directed=.55,directed=.25] (0,0) circle (2);
    \node at (0,-1) {$*$};
    \node at (0,-3) {$*$};
            \node at (0.5,-2.4) {\tiny $N$};
    \draw [dashed] (0,-1) to (0,-3);
  \end{tikzpicture}
    \; = \; \begin{tikzpicture}[anchorbase, scale=.3]
    \draw (0,0) circle (1);
    \draw (0,0) circle (3);
    \draw [very thick,directed=.65] (0,1) to [out=90,in=270] (.8,1.8);
     \draw [very thick,->] (-.8,1.8) to [out=90,in=270] (0,2.8) to (0,3);
    \draw [very thick,directed=.55,directed=.25] (0,0) circle (2);
    \node at (0,-1) {$*$};
    \node at (0,-3) {$*$};
            \node at (.8,-2.4) {\tiny $N$-$1$};
    \draw [dashed] (0,-1) to (0,-3);
  \end{tikzpicture} \circ \wrap
  \\
  &= (-1)^{N-1} \; \begin{tikzpicture}[anchorbase, scale=.3]
    \draw (0,0) circle (1);
    \draw (0,0) circle (3);
    \draw [very thick,->] (0,1) to (0,3);
    \draw [very thick,directed=.55] (0,0) circle (2);
    \node at (0,-1) {$*$};
    \node at (0,-3) {$*$};
        \node at (0.8,-2.4) {\tiny $N$-$1$};
    \draw [dashed] (0,-1) to (0,-3);
  \end{tikzpicture} \circ \wrap
  \;+\; 
  \begin{tikzpicture}[anchorbase, scale=.3]
    \draw (0,0) circle (1);
    \draw (0,0) circle (3);
    \draw [very thick,directed=.65] (0,1) to [out=90,in=270] (-.8,1.8);
     \draw [very thick,->] (.8,1.8) to [out=90,in=270] (0,2.8) to (0,3);
    \draw [very thick,directed=.55,directed=.25] (0,0) circle (2);
    \node at (0,-1) {$*$};
    \node at (0,-3) {$*$};
            \node at (0.8,-2.4) {\tiny $N$-$1$};
    \draw [dashed] (0,-1) to (0,-3);
  \end{tikzpicture}
  \circ \wrap
   \; = \; \begin{tikzpicture}[anchorbase, scale=.3]
    \draw (0,0) circle (1);
    \draw (0,0) circle (3);
    \draw [very thick,directed=.65] (0,1) to [out=90,in=270] (.8,1.8);
     \draw [very thick,->] (-.8,1.8) to [out=90,in=270] (0,2.8) to (0,3);
    \draw [very thick,directed=.55,directed=.25] (0,0) circle (2);
    \node at (0,-1) {$*$};
    \node at (0,-3) {$*$};
            \node at (.8,-2.4) {\tiny $N$-$2$};
    \draw [dashed] (0,-1) to (0,-3);
  \end{tikzpicture} \circ \wrap^2
  \\
  &= \cdots = - \; \begin{tikzpicture}[anchorbase, scale=.3]
    \draw (0,0) circle (1);
    \draw (0,0) circle (3);
    \draw [very thick,->] (0,1) to (0,3);
    \draw [very thick,directed=.55] (0,0) circle (2);
    \node at (0,-1) {$*$};
    \node at (0,-3) {$*$};
        \node at (0.5,-2.4) {\tiny $1$};
    \draw [dashed] (0,-1) to (0,-3);
  \end{tikzpicture} \circ \wrap^{N-1} 
  \;+\; \wrap^N = \wrap^N
 \end{align*}
The second equality arises by resolving the crossing while the third is an isotopy. The equalities then alternate between such which hold by expanding a crossing and such that use isotopies and \eqref{eqn:essentialidealtwo}.
\end{proof}
\begin{remark}Analogous one shows $\sum_{i=0}^N D^{N-i} (-1)^i c_i =0$ in $\AWeb[N]$, c.f. \cite[Section 8.2]{CK_ann}.
\end{remark}

Consider the algebra $\C[\wrap^{\pm 1}]/\la \wrap^N-1\ra$. Lemma~\ref{lem:DNone} implies that this surjects onto the subalgebra of $\essAWeb[N](1)$ generated by wraps, and the flow winding grading implies that the surjection is an isomorphism. The Chinese remainder theorem implies $\C[\wrap^{\pm 1}]/\la \wrap^N-1\ra\cong \bigoplus_{j=1}^N \C[\wrap]/\la \wrap-e^{j 2 \pi i /N}\ra$ and we denote by $P_k\in \C[\wrap]$ representatives for the idempotents that project onto the direct summands $\C[\wrap]/\la \wrap-e^{k 2 \pi i /N}\ra$. By abuse of notation we also write $P_k$ for the corresponding orthogonal idempotents in $\essAWeb[N](1)$. It is a straightforward but crucial observation that $\phi(P_k(\wrap))$ is the projection $V\twoheadrightarrow \C\la v_k\ra \hookrightarrow V$.

\begin{theorem} \label{thm:fullness}
  The functor $\phi:\AWeb[N]\mapsto \reph$ is full. 
\end{theorem}

\begin{proof} We show that the induced functor on the quotient $\essAWeb[N]$ is full. For this, let $\vec{k}$ and $\vec{l}$ be objects in $\essAWeb[N]$ and $\cev{k}$ be obtained from $\vec{k}$ by inverting orientations and the order of the sequence. We consider the isomorphisms of the form $f\colon \essAWeb[N](\vec{k},\vec{l}) \to \essAWeb[N](\emptyset,\vec{l} \otimes \cev{k})$ and their inverses, which are given by the following operations on diagrams.
\[  f\;\;= \;\begin{tikzpicture}[anchorbase, scale=.3]
\draw[thick] (0,0) circle (3.5);
\fill[black,opacity=.2] (0,0) circle (3.5);
\draw[thick,fill=white] (0,0) circle (2.5);
\draw (0,0) circle (1);
\draw (0,0) circle (4);
\draw[dotted] (-2.65,2.65) to [out=225,in=90] (-3.75,0) to [out=270,in=135] (-2.65,-2.65);
\draw[dotted] (-1.59,1.59) to [out=225,in=90] (-2.25,0) to [out=270,in=135] (-1.59,-1.59);
\draw [very thick] (-2.4,3.2) to (-2.1,2.8);
\draw [very thick] (-2.4,-3.2) to (-2.1,-2.8);
\draw [white,line width=.15cm] (0,2.4) to (0,3.6);
\draw[very thick] (-1.5,2) to [out=300,in=270] (0,2.5) to (0,4);
\draw [white,line width=.15cm] (1.97,1.47) to (3.16,2.37);
\draw[very thick] (-1.5,-2) to [out=60,in=270] (-1.75,0) to [out=90,in=180] (0,1.75) to [out=0,in=210] (2,1.5) to (3.2,2.4);
\node at (0,-1) {$*$};
\node at (0,-4) {$*$};
\draw [dashed] (0,-1) to (0,-4);
\end{tikzpicture}  
\quad,\quad 
f^{-1}\;\;=\;
 \begin{tikzpicture}[anchorbase, scale=.3]
\draw[thick] (0,0) circle (2.5);
\fill[black,opacity=.2] (0,0) circle (2.5);
\draw[thick,fill=white] (0,0) circle (1.5);
\draw (0,0) circle (1);
\draw (0,0) circle (4);
\draw[dotted] (-2.29,2.29) to [out=225,in=90] (-3.25,0) to [out=270,in=135] (-2.29,-2.29);
\draw [very thick] (-2.4,3.2) to (-1.5,2);
\draw [very thick] (-2.4,-3.2) to (-1.5,-2);
\draw [white,line width=.15cm] (1,0) to (2.6,0);
\draw [white,line width=.15cm] (.6,-.8) to (1.8,-2.4);
\draw [very thick] (1,0) to (2.5,0) to [out=0,in=30] (1.5,2);
\draw [very thick] (.6,-.8) to (1.8,-2.4) to[out=300,in=225]  (2.46,-2.46) to [out=45,in=285] (3.38,.9) to[out=105,in=345] (.9,3.38) to [out=165,in=90] (0,3) to (0,2.5);
\draw[dotted] (1.94,1.94) to [out=135,in=0] (0,2.75) ;
\draw[dotted] (1.94,-1.94) to [out=45,in=270] (2.75,0) ;
\draw[dotted]  (1.25,0) to [out=270,in=45] (0.88,-0.88);
\node at (0,-1) {$*$};
\node at (0,-4) {$*$};
\draw [dashed] (0,-1) to (0,-4);
\end{tikzpicture}
\]
Since these operations are given by tensoring with an identity morphism and then pre-composing with cups, or post-composing with caps, there are corresponding isomorphisms $\phi(f^{\pm 1})$ between the relevant morphism spaces in the target category $\reph$. To prove the theorem, it suffices to check that $\phi$ restricts to a surjective map from the morphism space $\essAWeb[N](\emptyset,\vec{l} \otimes \cev{k})$ to $\Hom_{\reph}(\C,\phi(\vec{l}\otimes \cev{k})$, which is isomorphic to the zero weight space in $\phi(\vec{l}\otimes \cev{k})$ by evaluating at $1\in \C$. 

    It is not hard to see that it suffices to prove this in the case where $\vec{l}\otimes \cev{k}$ consists entirely of entries $1$, the first $n$ of which point outward and the last $n$ inward. Indeed, if all weight zero vectors in this tensor product are hit, then composing with merge morphisms and permutations, one can hit any weight zero vector in a tensor product of fundamentals and their duals.
    
Actually, it is sufficient to find $v_k\otimes v_k^*$ in the image of $\phi(-)(1)$ applied to $\essAWeb[N](\emptyset,(1,1^*))$, as then we can take diagrammatic tensor products of such generators, composed with permutations to find any standard basis vector of weight zero. Now we invert the bending process via the isomorphisms $f^{\pm 1}$ and $\phi(f^{\pm 1})$, and the remaining problem becomes equivalent to finding the projection $V \twoheadrightarrow \C\la v_k\ra\hookrightarrow V$ in the image of $\phi$. But this we have already seen; it is given by $\phi(P_k)$.
\end{proof}

Later, we will prove that $\phi$ induces a functor from $\essAWeb[N]$ to $\reph$ that is not only full, but also faithful (see Theorem \ref{thm:faithfulness}). Nevertheless, we will continue to work in the more general framework of $\essbAWeb[N]$ whenever possible.

\subsection{Extremal weight projectors}
In \cite{QW}, we defined the concept of extremal weight projectors in the context of (affine) $\slnn{2}$ skein theory. This involved finding a suitable quotient of the affine Temperley-Lieb category, in which we identified a family of idempotents akin to Jones-Wenzl projector and corresponding, on the representation-theoretic side, to projections onto the direct sum of the top and bottom weight spaces in the tensor powers of the vector representation of $U(\slnn{2})$. The same question naturally extends beyond the $\slnn{2}$ case, and the definition can be adapted and generalized to the $\glnn{N}$ case as follows.

\begin{definition}\label{def:T} The elements $T_m\in\essbAWeb[N](m)$ are recursively defined via:
\begin{itemize} 
\item $T_1=\id_1$,
\item $T_2=\frac{1}{N}\sum_{k=0}^{N-1} \wrap^{-k}\otimes \wrap^k$,
\item $T_{m+1}=(\id_{m-1}\otimes T_2)(T_m\otimes \id_1)$ for $m\geq 2$.
\end{itemize}
\end{definition}

\begin{theorem} \label{thm:extweightproj}
  The element $\diagrep(T_m)$ is the endomorphism of $V^{\otimes m}$ projecting onto the extremal weight space $\C\la v_{i\cdots i}|i\in \{1,\dots,N\}\ra$ in $\Sym^m(V)\subset V^{\otimes m}$.
\end{theorem}
\begin{proof} For $m=1$ this is tautological. For $m=2$ we compute $\phi( \wrap^{-1}\otimes \id_1)(v_{i j})=\zeta^{-i} v_{i j}$ and $\phi(\id_1 \otimes \wrap)(v_{i j})=\zeta^{j} v_{i j}$. Thus we have:
\[\phi(T_2)(v_{i j}) =  \frac{1}{N}\sum_{k=0}^{N-1} \zeta^{k(j-i)} v_{i j} = \begin{cases} v_{i i} &\text{ if } i=j\\
0 &\text{ if } i \neq j
\end{cases}\]
For the vanishing, recall that $0=(X^N-1)=(X-1)(1+X+\cdots+X^{N-1})$, so if $X^N=1$ but $X\neq 1$, then $X$ is a zero of the cyclotomic polynomial. This is precisely the case for $\zeta^{(j-i)}$ for $i\neq j$.

For the induction step, we see immediately from the recursion that $\phi(T_{m+1})$ annihilates $v_{\epsilon_1 \epsilon_2\cdots \epsilon_{m+1}}$ unless $\epsilon_1=\cdots=\epsilon_{m}=:\epsilon$. In the remaining cases we have:
\begin{align*}
\phi(\id_{m-1}\otimes T_2)\phi(T_m\otimes \id_1)(v_{\epsilon\cdots \epsilon k}) &= \phi(\id_{m-1}\otimes T_2)(v_{\epsilon\cdots \epsilon k}) =\begin{cases} v_{\epsilon \cdots \epsilon \epsilon} &\text{ if } k=\epsilon\\ 0 &\text{ if } k\neq \epsilon \end{cases}
\end{align*} 
So $\phi(T_{m+1})$ is the extremal weight projector.
\end{proof}

In the following we show that the $T_m$ are idempotents that satisfy a number of properties analogous to the extremal weight projectors $\phi(T_m)$. This will lead to a proof that $\phi$ is indeed faithful on $\essAWeb[N]$. We start by studying properties of $T_2$.

\begin{lemma} \label{lem_T2idem}$T_2$ is an idempotent in $\essbAWeb[N]$ and thus also in the quotient $\essAWeb[N]$.
\end{lemma}
\begin{proof}
We compute: 
\begin{align*}T_2^2 =  \frac{1}{N^2}\sum_{k=0}^{N-1}\sum_{l=0}^{N-1} \wrap^{-k-l} \otimes \wrap^{k+l} = \frac{1}{N}\sum_{l=0}^{N-1} \wrap^{-l} \otimes \wrap^l =T_2
\end{align*}
Here we have used that $\wrap^{-N-i}\otimes\wrap^{N+i} = (-1)^{2N-2}{c_N}^{-1}c_N (\wrap^{-i}\otimes\wrap^{i}) =\wrap^{-i}\otimes\wrap^{i}$.
\end{proof}

\begin{lemma} \label{lem:T2comm}
In $\essbAWeb[N](3)$ we have that $\id_1 \otimes T_2$ and $T_2\otimes \id_1$ commute. 
\end{lemma}
\begin{proof}
$(\id_1 \otimes T_2)(T_2\otimes \id_1)= \frac{1}{N^2}\sum_{k,l=0}^{N-1}\wrap^{-k} \otimes \wrap^{k-l}\otimes \wrap^{l} = (\id_1 \otimes T_2)(T_2\otimes \id_1)$
\end{proof}

The following lemma states a relation that is trivially satisfied on the representation-theoretic side, i.e. after applying $\phi$, but which is non-obvious in $\essbAWeb[N]$ and $\essAWeb[N]$. For the latter, the result can be deduced from \cite[Equation (37)]{CK_ann}.

\begin{lemma}\label{lem:crossabs} The idempotent $T_2$ absorb the crossing $s=s_1$ between its two strands. More precisely $s T_2 = T_2 s = T_2$  in  $\essbAWeb[N](2)$ and thus also in $\essAWeb[N](2)$.
\end{lemma}

The diagrammatic proof is involved and we postpone it until Section~\ref{sec:proofs}.

\begin{lemma} In $\essbAWeb[N](2)$ we have $\wrapi T_2 \wrap=T_2$.
\end{lemma}
\begin{proof} 
For this we rewrite $T_2$ in terms of $\wrapi\otimes \id_1= \wrapi s$ and $\id_1 \otimes \wrap= \wrap s$. Using Lemma~\ref{lem:crossabs} and $(\wrapi s)^N (\wrap s)^N = \wrap^{-N}\otimes \wrap^N= 1$, we compute:
\[\wrapi T_2 \wrap= \wrapi s T_2 \wrap= \frac{1}{N}\sum_{k=0}^{N-1} \wrapi s ( \wrapi s)^{k}(\wrap s)^k\wrap s s=T_2 s= T_2\]
\vspace{-5mm}
\end{proof}

\begin{theorem} \label{thm:Tm}
  The elements $T_m$ of $\essbAWeb[N]$ satisfy the following properties: 
  \begin{enumerate}
  \item $T_m^2=T_m$;
  \item $T_m (\id_{k}\otimes T_n\otimes \id_{m-n-k}) =(\id_{k}\otimes T_n\otimes \id_{m-n-k}) T_m = T_m$ for $1\leq n< m$ and $0\leq k\leq m-n$;
  \item $(T_k\otimes \id_{m-k})(\id_{m-l}\otimes T_l)= (\id_{m-l}\otimes T_l) (T_k\otimes \id_{m-k})  = T_{m}$ for $k+l> m$; \label{item:overlap}
\item $T_m s_i= s_i T_m=T_m$ for $m\geq 2$; 
\item $T_m u_i= u_i T_m=0$ for $m\geq 2$;
 \item $\wrapi T_m \wrap=T_m$. 
  \end{enumerate}
Here, $s_i$ and $u_i$ again refer to crossings and dumbbell webs between the strands in position $i$ and $i+1$, see Definition~\ref{def:si}.
\end{theorem}
\begin{proof} We have already checked in Lemma~\ref{lem_T2idem} that $T_2$ is idempotent. From the definition, it is clear that $T_m$ is a product of $m-1$ distinct factors of the form $\id_{k}\otimes T_2 \otimes \id_{m-2-k}$ for $0\leq k\leq m-2$. Lemma~\ref{lem:T2comm} implies that these factors commute and so $T_m$ is an idempotent (1) that absorbs smaller $T_n$, i.e. (2). It is also clear that overlapping projectors $T_l$ and $T_k$ combine as in (3). The crossing absorption property of $T_2$ now implies the one for $T_m$ and crossings $s_i$ for $1\leq i \leq m-1$. 

Using crossing absorption, we obtain the rotation conjugation invariance (6) from the $T_2$  case:
\begin{align*}
\wrapi T_m \wrap &=\wrapi (T_{m-1}\otimes \id_1)(\id_{m-2}\otimes T_2)(T_{m-1}\otimes \id_1) \wrap \\
&= (\id_1 \otimes T_{m-1})\wrapi s_{m-1}\cdots s_{2}s_{1}(T_2\otimes \id_{m-2})s_{1}s_{2}\cdots s_{m-1} \wrap (\id_1 \otimes T_{m-1})
\\
&= (\id_1 \otimes T_{m-1})((\wrapi T_2 \wrap)\otimes \id_{m-2}) (\id_1 \otimes T_{m-1}) 
\\&= (\id_1 \otimes T_{m-1})( T_2 \otimes \id_{m-2}) (\id_1 \otimes T_{m-1}) = T_{m} 
\end{align*}

This implies the missing crossing absorption relation (4)
\[T_m s_m= T_m \wrapi s_{m-1} \wrap = \wrapi T_m  s_{m-1} \wrap = \wrapi T_m  \wrap =T_m .\] 
Finally, the $u_i$ annihilation property (5) is equivalent to $s_i$ crossing absorption.
\end{proof}

Now we can give an alternative recursion relation for $T_m$ for $m\geq 3$. This is the direct generalization of the defining recursive relation in \cite[Definition 15]{QW}, and it is reminiscent of the Jones-Wenzl projectors.

\begin{corollary} The idempotents $T_m$ satisfy the following recursion for $m\geq 3$:
\[T_m= (T_{m-1}\otimes \id_1)s_{m-1}(T_{m-1}\otimes \id_1)\]
Graphically, we write this as: 
\begin{equation}
\label{eq:recursive}
 \begin{tikzpicture}[anchorbase, scale=.3]
\draw[thick] (0,0) circle (4);
\fill[black,opacity=.2] (0,0) circle (4);
\draw[thick,fill=white] (0,0) circle (2);
\draw[dotted] (-1.05,1.05) to [out=45,in=180] (0,1.5) to [out=0,in=135] (1.05,1.05);
\draw [thick] (0,0) circle (2);
\draw[dotted] (-3.16,3.16) to [out=45,in=180] (0,4.5) to [out=0,in=135] (3.16,3.16);
\draw (0,0) circle (1);
\draw (0,0) circle (5);
\draw [very thick] (.8,.6) to (1.6,1.2);
\draw [very thick] (1,0) to (2,0);
\draw [very thick,->] (4,0) to (5,0);
\draw [very thick] (-.8,.6) to (-1.6,1.2);
\draw [very thick,->] (3.2,2.4) to (4,3);
\draw [very thick,->] (-3.2,2.4) to (-4,3);
\node at (0,-1) {$*$};
\node at (0,-5) {$*$};
\draw [dashed] (0,-1) to (0,-5);
\node at (0,2.95) {\small$T_{m}$};
\end{tikzpicture}
\;\; :=\;\;
 \begin{tikzpicture}[anchorbase, scale=.3]
\draw[thick] (0,0) circle (2.5);
\fill[black,opacity=.2] (0,0) circle (2.5);
\draw[thick,fill=white] (0,0) circle (1.5);
\draw[dotted] (-1.93,1.93) to [out=45,in=180] (0,2.75) to [out=0,in=135] (1.93,1.93);
\draw[dotted] (-0.88,0.88) to [out=45,in=180] (0,1.25) to [out=0,in=135] (0.88,0.88);
\draw[thick] (0,0) circle (4.5);
\fill[black,opacity=.2] (4.5,0) arc (0:360:4.5) -- (3.5,0) arc (360:0:3.5);
\draw [thick] (0,0) circle (3.5);
\draw[dotted] (-3.4,3.4) to [out=45,in=180] (0,4.75) to [out=0,in=135] (3.4,3.4);
\draw[dotted] (-2.29,2.29) to [out=45,in=180] (0,3.25) to [out=0,in=135] (2.29,2.29);
\draw (0,0) circle (1);
\draw (0,0) circle (5);
\draw [very thick] (.8,.6) to (1.2,.9);
\draw [very thick] (-.8,.6) to (-1.2,.9);
\draw [very thick] (-2,1.5) to (-2.8,2.1);
\fill [white] (1.4,.2) -- (2.6,.2) -- (2.6,-.2) -- (1.4,-.2);
\fill [white] (3.4,.2) -- (4.6,.2) -- (4.6,-.2) -- (3.4,-.2);
\draw [draw =white, double=black, thick, double distance=1.25pt] (1,0) -- (2.5,0) to [out=0,in=210] (2.8,2.1);
\draw [very thick,->] (2,1.5) to [out=30,in=180] (3.5,0) to (5,0);
\draw [very thick] (2.8,2.1) -- (2.8,2.1);
\draw [very thick,->] (3.6,2.7) to (4,3);
\draw [very thick,->] (-3.6,2.7) to (-4,3);
\node at (0,-1) {$*$};
\node at (0,-5) {$*$};
\draw [dashed] (0,-1) to (0,-5);
\node at (.1,1.92) {\tiny$T_{m-1}$};
\node at (0,3.95) {\tiny$T_{m-1}$};
\end{tikzpicture}
\end{equation}
\end{corollary}
\begin{proof} We check this identity as follows.
\begin{align*}
(T_{m-1}\otimes \id_1)s_{m-1}(T_{m-1}\otimes \id_1) &= (T_{m-1}\otimes \id_1)s_{m-2}\cdots s_{2}s_{1}(\id_1\otimes T_{m-1}) s_1 s_2 \cdots s_{m-1}\\ 
 &= (T_{m-1}\otimes \id_1)(\id_1\otimes T_{m-1} ) s_1 s_2 \cdots s_{m-1} = T_m s_1 s_2 \cdots s_{m-1} =T_m
\end{align*}
The first equation holds by isotopy, the third by item (3) of Theorem~\ref{thm:Tm}, and the others follows from crossing absorption.
\end{proof}

\begin{lemma}\label{lem:linkedproj} Let $m,n\in \N$ with $m+n\geq 3$, then $(T_m\otimes T_n)s_m(T_m\otimes T_n)=T_{m+n}$. This means, crossing-connected projectors can be combined. 
\end{lemma}
\begin{proof} We may assume that $m\geq 2$ and compute:
\[(\id_m\otimes T_n)(T_m\otimes \id_n)s_m(T_m\otimes \id_n)(\id_m\otimes T_n)= (\id_m\otimes T_n)(T_{m+1}\otimes \id_{n-1})(\id_m\otimes T_n) = T_{m+n}\] Here we have used the projector recursion \eqref{eq:recursive} and the fact that overlapping projectors can be combined, i.e. (\ref{item:overlap}) in Theorem \ref{thm:Tm}.
\end{proof}

Next we consider the images of the idempotents $T_m$ in the quotient category $\essAWeb[N]$. Recall the morphisms $\{P_a\}_{a\in \{1,\dots, N}$ that were introduced just before Theorem \ref{thm:fullness} as diagrammatic versions of projectors on eigenspaces. We will see in Lemma \ref{lem:projexpand} that they can be combined to give an alternate definition of the extremal weight projectors, which amounts to saying that in the quotient category $\essAWeb[N]$, extremal weight spaces can be broken into individual weight spaces.

\begin{lemma}\label{lem:extrabs} In $\essAWeb[N]$ we have $(P_a\otimes P_b)\circ T_2 = T_2\circ (P_a\otimes P_b) = \delta_{a,b}(P_a\otimes P_b)$.
\end{lemma}
\begin{proof} The $\C$-algebra $R:=\C[X^{\pm 1},Y^{\pm 1}]/\la X^N-1,Y^N-1\ra$ surjects onto the subalgebra of $\essAWeb[N](2)$ generated by wraps and their inverses via the map $1 \mapsto \id_2$, $X\mapsto \wrap\otimes \id_1=sD$ and $Y\mapsto \id_1\otimes \wrap=Ds$. We will check the desired equalities in $R$, where $T_2$ is represented by $\frac{1}{N}\sum_{k=0}^{N-1} X^{-k}Y^k$ and $P_a\otimes P_b$ is represented by $P_a(X)P_b(Y)$, which then implies that these equalities also hold in $\essAWeb[N](2)$.

Note that the idempotents $P_a(X)P_b(Y)$ decompose $R\cong \bigoplus_{a,b} P_a(X)P_b(Y) R$ into $1$-dimensional summands, which precisely consist of simultaneous eigenvectors for multiplication by $X$ and $Y$ with eigenvalues $\zeta^a$ and $\zeta^b$ respectively. Thus we can compute the action of $T_2$ on such an idempotent as:
\[P_a(X)P_b(Y)T_2 = P_a(X)P_b(Y) \frac{1}{N}\sum_{k=0}^{N-1} X^{-k}Y^k = P_a(X)P_b(Y) \frac{1}{N}\sum_{k=0}^{N-1} \zeta^{k(b-a)} = \delta_{a,b}P_a(X)P_b(Y).\vspace{-.6cm}\]

\end{proof}
\begin{corollary}\label{cor:Ttwo} In $\essAWeb[N]$ we have 
 $T_2=\sum_{k=1}^N P_k\otimes P_k$.
\end{corollary}

\begin{corollary}\label{cor:crossabs} In $\essAWeb[N]$ we have $s(P_k\otimes P_k)=(P_k\otimes P_k)s=P_k\otimes P_k $.
\end{corollary}
\begin{proof} We only consider composing with $s$ on the left:
\[s(P_k\otimes P_k)=s T_2(P_k\otimes P_k)=T_2(P_k\otimes P_k)=P_k\otimes P_k\]
Here we have used Lemma~\ref{lem:extrabs} twice and Lemma~\ref{lem:crossabs} in between.
\end{proof}

\begin{lemma}\label{lem:projexpand} In $\essAWeb[N]$ we have $T_m=\sum_{k=1}^N \bigotimes_m P_k$.
\end{lemma}
\begin{proof} We have observed this for $m=2$ in Corollary~\ref{cor:Ttwo}. For $m=1$ this follows from the decomposition $1=\sum_{k=1}^N P_k$ in $\C[\wrap]/\la \wrap^N-1\ra$. Also, it is not hard to see that the $P_k$'s slide through crossings, e.g. $s \circ (P_k\otimes \id_1)=(\id_1 \otimes P_k) s$ in $\essAWeb[N](2)$. Now we proceed inductively for $m\geq 2$:
\begin{align*}
T_{m+1}=(T_m\otimes \id_1)s_m(T_m\otimes \id_1) &= \sum_{k=1}^N\sum_{l=1}^N (P_k\otimes\cdots \otimes P_k\otimes \id_1)s_m(P_l\otimes\cdots \otimes P_l\otimes \id_1) 
\\&= \sum_{k=1}^N (\id_{m-1} \otimes P_k\otimes \id_1) s_m(P_k\otimes\cdots \otimes P_k\otimes \id_1)
\\&= \sum_{k=1}^N s_m (P_k\otimes\cdots \otimes P_k\otimes P_k) =\sum_{k=1}^N \bigotimes_m P_k
\end{align*}
Here we have used the orthogonality of the idempotents $P_k$ to proceed to the second line and the sliding property to proceed to the third line. The final crossing absorption follows from Corollary~\ref{cor:crossabs}. 
\end{proof}

\subsection{Faithfulness of the diagrammatic presentation}
We will now combine the previous results to prove the following theorem.

\begin{theorem} \label{thm:faithfulness}
    The functor $\phi:\essAWeb[N]\mapsto \reph$ is faithful.
\end{theorem}

We partition the proof of the theorem into three parts.

\begin{proposition}\label{prop:faithful} $\phi$ is injective when restricted to the endomorphism algebra $\essAWeb[N](n)$.
\end{proposition}
\begin{proof}
To see this result, we will exhibit a spanning set in $\essAWeb[N](n)$ that is sent under $\phi$ to a linear basis. Consider $\epsilon,\epsilon'\in \{1,\cdots,N\}^n$ so that $|\{i,\epsilon_i=k\}|=|\{i,\epsilon'_i=k\}|$ for all $k=1,\cdots, N$. Choose $\sigma_\epsilon^{\epsilon'}\in \mathfrak{S}_n$ to be the smallest in length so that $\epsilon'_{\sigma(i)}=\epsilon_i$. For example, it can be inductively defined by assigning to $1$ the smallest $r$ so that $\epsilon'_r=\epsilon_1$, etc. Recall the notation $P_k$ for the polynomial such that $P_k(D)\in \essAWeb[N](1)$ is the projector onto the $\zeta^k$ eigenspace of $D$. In $\essAWeb[N](n)$, denote $w_i=\id_{i-1}\otimes \wrap \otimes \id_{n-i}$ the complete wrap of the $i$-th strand. We define:
  \[
\phi_\epsilon^{\epsilon'}:=\sigma P_{\epsilon_n}(w_n)\cdots P_{\epsilon_1}(w_1).
\]
It is easy to see that the set $\{\phi_\epsilon^{\epsilon'}\}$ is linearly independent, because it is so under $\phi$.

On the other hand, we can deduce from Lemma~\ref{lem:upwardannulus} and its proof that this set spans $\essAWeb[N](n)$. Indeed, given the essential circle relations, we first deduce that any element $W\in \essAWeb[N](n)$ is made of a composition of elements from $\mathfrak{S}_n\rightarrow \essAWeb[N](n)$ and the wraps $w_i$. From there, using far-commutation and the formulas:
\[
w_is_{i-1}=s_{i-1}w_{i-1},\quad w_is_{i}=s_{i}w_{i+1},
\]
one can see that this gives an algebra epimorphism $ \langle w_i\rangle\rtimes \mathfrak{S}_n \twoheadrightarrow \essAWeb[N](n)$.

Now, by construction, the polynomials $\{P_k(w)\}_{k\in \{0,\dots,N-1\}}$ form a basis of $\C[w]/(w^N-1)$, and since the $w_i$'s commute, it follows that the elements $P_{\epsilon_n}(w_n)\cdots P_{\epsilon_1}(w_1)$ span $\langle w_i\rangle_{i\in \{1,\dots,n-1\}}\subset \essAWeb[N](n)$. Thus, any element in $\essAWeb[N](n)$ can be written as a linear combination of terms of the kind $\tilde{\sigma}\phi_{\epsilon}^{\epsilon}$ with $\tilde{\sigma}$ in the image of $\mathfrak{S}_n$. It remains to see that $\tilde{\sigma}$ can be assumed to be of minimal length, which is equivalent to saying that no two strands corresponding to the same value in $\epsilon$ cross.
Via isotopies, this reduces to the identities proven in Corollary~\ref{cor:crossabs}.

This proves that the set $\{\phi_{\epsilon}^{\epsilon'}\}$ spans $\essAWeb[N](n)$ and concludes the proof.
\end{proof}

\begin{proposition}\label{prop:faithful2} $\phi$ is injective when restricted to morphism spaces in $\essAWebp[N]$.
\end{proposition}
\begin{proof}
Suppose a linear combination $W$ of webs in such a morphism space are sent to zero under $\phi$. Then for each web in this linear combination we pre-compose with a merge web $M$ and post-compose with a splitter web $S$ in order to obtain an endomorphism $SWM$ in $\essAWeb[N](|\vec{k}|)$. Then we have $\phi(SWM)= \phi(S)\phi(W)\phi(M)=0$ and Proposition~\ref{prop:faithful} implies that $SWM=0$. This implies $W= c MSWMS = 0$ where $c\neq 0$ is a scalar resulting from opening bigons. 
\end{proof}

\begin{proof}[Proof of Theorem~\ref{thm:faithfulness}]
Let $W$ be a linear combination of webs in some morphism space of $\essAWeb[N]$, which is sent to zero under $\phi$. Then there exists a composition of invertible bending and braiding operations similar to those used in the proof of Theorem \ref{thm:fullness} that transforms $W$ into a linear combination $W^\prime$ of webs in a morphism space as in Proposition~\ref{prop:faithful2}, which is also sent to zero under $\phi$. The proposition then implies $W^\prime=0$ and, by invertibility of the operations, $W=0$.
\end{proof}

The final result of this section is best expressed in terms of Karoubi envelopes, the definition of which we recall now.

\begin{definition}
The Karoubi completion of a category $\mathcal{C}$ is the category with objects given by pairs $(X,e)$, where $X$ is an object of $\mathcal{C}$ and $e\in \Hom_\mathcal{C}(X,X)$ an idempotent. Morphisms between $(X,e)$ and $(Y,f)$ are of the form $f\circ g \circ e$ with $g\in \Hom_\mathcal{C}(X,Y)$.
\end{definition}

If $\mathcal{C}$ is additive or monoidal, then the Karoubi completion $Kar(\mathcal{C})$ inherits these structures. If $\mathcal{C}$ is not yet additive but linear, then we pass to the additive closure before taking the idempotent completion as above. The resulting category $Kar(\mathcal{C})$ will be additive and linear. If the morphism spaces of $\mathcal{C}$ furthermore admit a $\Z$-grading $\deg$, i.e. $\mathcal{C}$ is $\Z$-pre-graded, then so will be Karoubi completion, which we denote by $\Kar(\mathcal{C})^*$. In this case, we reserve the notation $Kar(\mathcal{C})$ for the $\Z$-graded, additive, linear category whose objects are generated by formal grading shifts $\sh^k (X,e)$ of the objects $(X,e)$ in $\Kar(\mathcal{C})^*$ and morphism required to be of degree zero. I.e. $g\colon \sh^k (X,e) \to \sh^l(Y,f)$ is required to satisfy $\deg(g)=l-k$.

\begin{corollary} \label{cor:diagrep} The functor $\phi$ induces an equivalence of additive, $\C$-linear pivotal categories: 
  \[
  \Kar(\essAWeb[N])^* \simeq \reph.
  \]
\end{corollary}
\begin{proof} 
This follows from Theorems~\ref{thm:fullness} and \ref{thm:faithfulness} since $\reph$ is already idempotent complete and any of its object can be written as the direct sum of $\phi$-images of idempotents in $\essAWeb[N]$.
\end{proof}

Let $\rephp$ denote the full subcategory of $\reph$ containing only integral $\mathfrak{h}$-representations whose weights have non-negative entries.

\begin{remark} The functor $\phi$ restricts to a fully faithful functor $\essAWebp[N] \to \rephp$ that induces an equivalence of $\C$-linear monoidal categories $\Kar(\essAWebp[N])^* \simeq \rephp$. 
\end{remark}

\subsection{Proof of Lemma~\ref{lem:crossabs}}
\label{sec:proofs}
This section contains a proof of the fact that the $T_2$ projector absorbs crossings. It can be safely skipped on a first read-through.

In order to prove Lemma~\ref{lem:crossabs} we study the endomorphism algebra of the object $2$ in $\essbAWeb[N]$. For $k\geq 1$ we introduce the following notation: 

\begin{equation*}
E_k:=\begin{tikzpicture}[anchorbase, scale=.3]
    \draw (0,0) circle (1);
    \draw (0,0) circle (3);
    \draw [double,->] (0,1) to (0,3);
    \draw [very thick,directed=.55] (0,0) circle (2);
    \node at (0,-1) {$*$};
    \node at (0,-3) {$*$};
        \node at (0.5,-2.4) {\tiny $k$};
    \draw [dashed] (0,-1) to (0,-3);
  \end{tikzpicture}
  \quad,\quad
B_k:=  \begin{tikzpicture}[anchorbase, scale=.3]
    \draw (0,0) circle (1);
    \draw (0,0) circle (3);
    \draw [double,directed=.65] (0,1) to [out=90,in=270] (-.8,1.8);
     \draw [double,->] (.8,1.8) to [out=90,in=270] (0,2.8) to (0,3);
    \draw [very thick,directed=.55] (0,0) circle (2);
    \node at (0,-1) {$*$};
    \node at (0,-3) {$*$};
            \node at (0.5,-2.4) {\tiny $k$};
    \draw [dashed] (0,-1) to (0,-3);
  \end{tikzpicture}
    \quad,\quad
    A_k := \begin{tikzpicture}[anchorbase, scale=.3]
    \draw (0,0) circle (1);
    \draw (0,0) circle (3);
    \draw [double,directed=.65] (0,1) to [out=90,in=270] (.8,1.8);
     \draw [double,->] (-.8,1.8) to [out=90,in=270] (0,2.8) to (0,3);
    \draw [very thick,directed=.55,directed=.25] (0,0) circle (2);
    \node at (0,-1) {$*$};
    \node at (0,-3) {$*$};
            \node at (.8,-2.4) {\tiny $k$-$2$};
    \draw [dashed] (0,-1) to (0,-3);
  \end{tikzpicture}
  \quad,\quad
  \wrap_2:=\begin{tikzpicture}[anchorbase, scale=.3]
\draw (0,0) circle (1);
\draw (0,0) circle (3);
\draw [double,->] (0,1) .. controls (0,1.7) and (-1.6,1) .. (-1.6,0) to [out=-90,in=180] (0,-1.8) to [out=0,in=-90] (2.1,0) to [out=90,in=-90] (0,3); 
\node at (0,-1) {$*$};
\node at (0,-3) {$*$};
\draw [dashed] (0,-1) to (0,-3);
\end{tikzpicture}
\end{equation*}
Here we set $A_1=0$, $B_2=\wrap_2$ and $A_2=\id$, and doubled edges stand for 2-labeled edges. Note that in the definition of $B_k$, we haven't depicted the orientation of one of the strands: this is because it depends on $k$. More precisely, we have:
  \[
  B_1:=  \begin{tikzpicture}[anchorbase, scale=.3]
    \draw (0,0) circle (1);
    \draw (0,0) circle (3);
    \draw [double,directed=.65] (0,1) to [out=90,in=270] (-.8,1.8);
     \draw [double,->] (.8,1.8) to [out=90,in=270] (0,2.8) to (0,3);
    \draw [very thick,directed=.55,rdirected=.26] (0,0) circle (2);
    \node at (0,-1) {$*$};
    \node at (0,-3) {$*$};
    \node at (0.5,-2.4) {\tiny $1$};
    \node at (0.3,1.5) {\tiny $1$};
    \draw [dashed] (0,-1) to (0,-3);
  \end{tikzpicture}\quad,\quad
  B_2=\wrap_2,\quad
  B_k:=  \begin{tikzpicture}[anchorbase, scale=.3]
    \draw (0,0) circle (1);
    \draw (0,0) circle (3);
    \draw [double,directed=.65] (0,1) to [out=90,in=270] (-.8,1.8);
     \draw [double,->] (.8,1.8) to [out=90,in=270] (0,2.8) to (0,3);
    \draw [very thick,directed=.55,directed=.26] (0,0) circle (2);
    \node at (0,-1) {$*$};
    \node at (0,-3) {$*$};
    \node at (0.5,-2.4) {\tiny $k$};
    \draw [dashed] (0,-1) to (0,-3);
  \end{tikzpicture}
  \quad\text{if}\; k\geq 3.
\]
 Note also that $A_k=B_k=E_k=0$ for $k>N$.

\begin{lemma}\label{lem:endtwo} The following statements hold in the endomorphism algebra of the $2$-labeled upward point in $\essbAWeb[N]$:
\begin{enumerate}
\item $E_k = \delta_{k,N} c_N$,
\item $\wrap_2$ is invertible and central,
\item $B_k= A_k \wrap_2$ for $k\geq 2$
\item $B_1 A_k = - E_{k-1} + A_{k-1} \wrap_2 + A_{k+1}$ for $k\geq 2$ 
\item $B_N=E_N=c_N$ and thus $A_{N}=c_N \wrap_2^{-1}$.
\end{enumerate}
\end{lemma}
Note that only (1) and (5) depend on the value of $N$.
\begin{proof}
  (1) holds by definition of $\essbAWeb[N]$, (2) and (3) follow from isotopies. For (4) we resolve the crossing in $E_{k-1}$ to obtain:
  
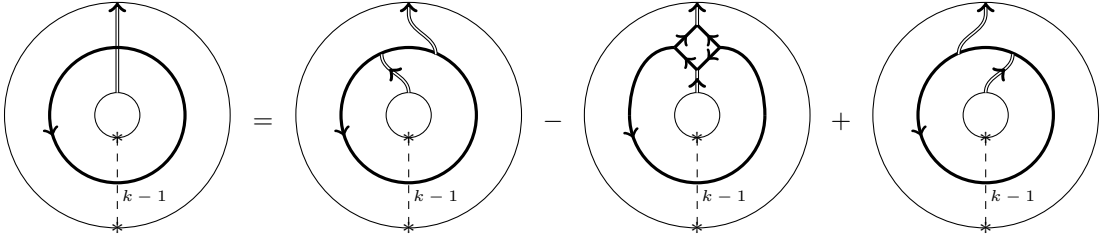
\begin{equation*}
\begin{tikzpicture}[anchorbase, scale=.3]
    \draw (0,0) circle (1);
    \draw (0,0) circle (5);
    \draw [double,->] (0,1) to (0,5);
    \draw [very thick,directed=.55] (0,0) circle (3);
    \node at (0,-1) {$*$};
    \node at (0,-5) {$*$};
    \draw [dashed] (0,-1) to (0,-5);
    \node at (1.2,-3.6) {\tiny $k-1$};
  \end{tikzpicture}
  \;\; = \;\;
  \begin{tikzpicture}[anchorbase, scale=.3]
    \draw (0,0) circle (1);
    \draw (0,0) circle (5);
     \draw [very thick,directed=.55] (0,0) circle (3);
    \draw [double,directed=.65] (0,1) to [out=90,in=270] (-1.2,2.7);
     \draw [double,->] (1.2,2.7) to [out=90,in=270] (0,4.5) to (0,5);
    \node at (0,-1) {$*$};
    \node at (0,-5) {$*$};
    \draw [dashed] (0,-1) to (0,-5);
    \node at (1.2,-3.6) {\tiny $k-1$};
  \end{tikzpicture}
 \;\; -\;\;
   \begin{tikzpicture}[anchorbase, scale=.3]
    \draw (0,0) circle (1);
    \draw (0,0) circle (5);
     \draw[very thick,directed=.11] (180:3) arc (180:0+360:3);
      \draw [very thick] (3,0) to [out=90,in=0](1,3);
      \draw [very thick] (-1,3) to [out=180,in=90](-3,0);
    \draw [double,directed=.65] (0,1) to (0,2);
    \draw[very thick, directed=.55] (0,2) to (-1,3);
    \draw[very thick, directed=.55] (1,3) to (0,2);
    \draw[very thick, directed=.55] (-1,3) to (0,4);
    \draw[very thick, directed=.55] (1,3) to (0,4);
     \draw [double,->] (0,4) to (0,5);
    \node at (0,-1) {$*$};
    \node at (0,-5) {$*$};
    \draw [dashed] (0,-1) to (0,-5);
    \node at (1.2,-3.6) {\tiny $k-1$};
  \end{tikzpicture}        
\;\;+ \;\;
  \begin{tikzpicture}[anchorbase, scale=.3]
    \draw (0,0) circle (1);
    \draw (0,0) circle (5);
     \draw [very thick,directed=.55] (0,0) circle (3);
    \draw [double,directed=.65] (0,1) to [out=90,in=270] (1.2,2.7);
     \draw [double,->] (-1.2,2.7) to [out=90,in=270] (0,4.5) to (0,5);
    \node at (0,-1) {$*$};
    \node at (0,-5) {$*$};
    \draw [dashed] (0,-1) to (0,-5);
    \node at (1.2,-3.6) {\tiny $k-1$};
  \end{tikzpicture}        
\end{equation*}
Here the first and third summands are $B_{k-1}=A_{k-1}\wrap_2$ and $A_{k+1}$ respectively. The web in the second summand simplifies as follows: 
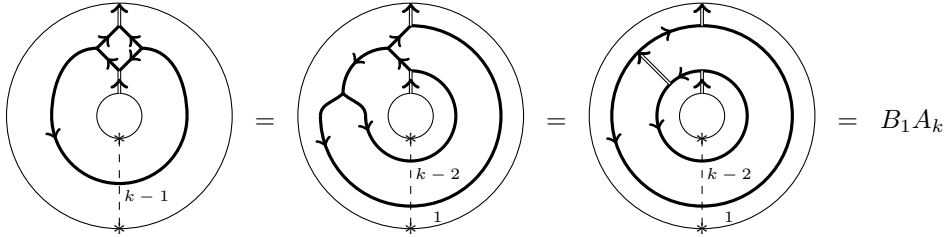
\[   \begin{tikzpicture}[anchorbase, scale=.3]
    \draw (0,0) circle (1);
    \draw (0,0) circle (5);
     \draw[very thick,directed=.11] (180:3) arc (180:0+360:3);
      \draw [very thick] (3,0) to [out=90,in=0](1,3);
      \draw [very thick] (-1,3) to [out=180,in=90](-3,0);
    \draw [double,directed=.65] (0,1) to (0,2);
    \draw[very thick, directed=.55] (0,2) to (-1,3);
    \draw[very thick, directed=.55] (1,3) to (0,2);
    \draw[very thick, directed=.55] (-1,3) to (0,4);
    \draw[very thick, directed=.55] (1,3) to (0,4);
     \draw [double,->] (0,4) to (0,5);
    \node at (0,-1) {$*$};
    \node at (0,-5) {$*$};
    \draw [dashed] (0,-1) to (0,-5);
    \node at (1.2,-3.6) {\tiny $k-1$};
  \end{tikzpicture}        
  \;\;=\;\;
   \begin{tikzpicture}[anchorbase, scale=.3]
    \draw (0,0) circle (1);
    \draw (0,0) circle (5);
     \draw[very thick,directed=.11] (180:2) arc (180:0+360:2);
     \draw[very thick,directed=.11] (180:4) arc (180:0+360:4);
      \draw [very thick] (2,0) to [out=90,in=0](0,2);
      \draw [very thick] (4,0) to [out=90,in=0](0,4);
      \draw [very thick, directed=.55] (-1,3) to [out=180,in=90](-3,1);
      \draw [very thick] (-4,0) to [out=90,in=225](-3,1);
      \draw [very thick] (-2,0) to [out=90,in=315](-3,1);
    \draw [double,directed=.65] (0,1) to (0,2);
    \draw[very thick, directed=.55] (0,2) to (-1,3);
    \draw[very thick, directed=.55] (-1,3) to (0,4);
     \draw [double,->] (0,4) to (0,5);
    \node at (0,-1) {$*$};
    \node at (0,-5) {$*$};
    \draw [dashed] (0,-1) to (0,-5);
    \node at (1.2,-2.6) {\tiny $k-2$};
    \node at (1.2,-4.5) {\tiny $1$};
  \end{tikzpicture}    
  \;\;=\;\;
    \begin{tikzpicture}[anchorbase, scale=.3]
    \draw (0,0) circle (1);
    \draw (0,0) circle (5);
     \draw [very thick,directed=.55, rdirected=.32] (0,0) circle (4);
     \draw [very thick,directed=.55, directed=.34] (0,0) circle (2);
    \draw [double,directed=.65] (0,1) to (0,2);
     \draw [double,->] (-1.42,1.42) to  (-2.84,2.84);
     \draw [double,->] (0,4) to (0,5);
    \node at (0,-1) {$*$};
    \node at (0,-5) {$*$};
    \draw [dashed] (0,-1) to (0,-5);
  \node at (1.2,-2.6) {\tiny $k-2$};
    \node at (1.2,-4.5) {\tiny $1$};
  \end{tikzpicture}
  \;\; =\;\; B_1 A_{k}
  \vspace{-.6cm} \]
\end{proof}

As a corollary, we get that the elements $A_k$ can be written in terms of powers of $B_1$ and $D_2$:
\begin{corollary} \label{cor:endtwob}	  The elements $A_k$ for $k<N+2$ satisfy the recursion $A_k := B_1 A_{k-1} -A_{k-2}\wrap_2$ for $k\geq 5$ with initial conditions $A_3=B_1$ and $A_4=B_1^2-\wrap_2$.
\end{corollary}
\begin{proof} We induct on $k$. For $k=1<N$ we use Lemma~\ref{lem:endtwo} to obtain $B_1=B_1A_2 = -E_1 +A_1 \wrap_2 +A_3=A_3$. Similarly, for $k=2<N$, we get from Lemma~\ref{lem:endtwo} that $B_1^2 - \wrap_2= B_1 A_3-A_2\wrap_2 = -E_2 + A_4=A_4$. For the recurrence relation we compute:
  $B_1 A_{k-1} -A_{k-2}\wrap_2 = -E_{k-2} + A_{k} = A_{k}$ if $5\leq k<N+2$.
\end{proof}

\begin{remark} We will now find expressions for $T_2$ which are more convenient in the following proof.
First we introduce the notation $t=\wrapi s \wrap$ for the rotation conjugate of the crossing. Now we use the facts that $\wrap^2$ is central in $\essbAWeb[N](2)$ and that $(s\wrap)^{-1}$ and $(\wrap s)$ commute to write $(s \wrap)^{-k} (\wrap s)^k = (t s)^k$, which gives: 
\[T_2=\frac{1}{N}\sum_{k=0}^{N-1} (t s)^k\]  
Further, we resolve the crossings as $s=\id_2-u$ and $t=\id_2-v$ where $u=u_1$ is the 2-labeled dumbbell and $v$ its conjugate. So we have:
\begin{equation}
\label{eqn_Ttwob}
T_2=\frac{1}{N}\sum_{k=0}^{N-1} ((\id_2-v)(\id_2-u))^k
\end{equation}
\end{remark}

If we write $M$ and $S$ for the merge and split vertices on two strands, such that $SM=u$, we get the following equality for $k\geq 2$:
\[\underbrace{\cdots v u v u}_{k \text{ factors}} = \wrap^{1-k} S B_1^{k-1} M\]  

We will now attempt to rewrite the expressions $X_n:=\wrap^{1-n} S A_{n+1} M$ in terms of powers of $B_1$. We can then use the relation $X_N=0$ to deduce a relation between compositions of the webs $u$ and $v$. 
To this end we define $R_{2k-1}:=\underbrace{(\id_2-v)(\id_2-u)\cdots (\id_2-v)}_{2k-1 \text{ factors}}$ and $S_x:=\sum_{k=1}^x R_{2k-1}$. 

\begin{lemma} \label{lem:XY}For $2\leq n \leq N$ we have:
\[X_n = \begin{cases} 
u - R_{n-1}u - u S_{n/2-1} u & \text{ even } n \\
u - u S_{(n-1)/2} u & \text{ odd } n
\end{cases}\]
\end{lemma}
\begin{proof} We will use the notation $Y_n$ for the entries on the right-hand side of the equation in the statement of the lemma. The proof of $X_n=Y_n$ proceeds by induction on $n$. For $n=2$ we have $X_2=v u = u-(\id_2-v)u = Y_2$. Similarly, for $n=3$, we have
\[ X_3=\wrap^{-2} S A_{4} M = \wrap^{-2} S (B_1^2-\wrap_2) M = uvu-u = u - u(1-v)u = Y_3\] since $u^2=2 u$ by the bigon relation. We prove the remaining cases recursively. For this, note that the elements $X_n=\wrap^{1-n} S A_{n+1} M$ inherit a recurrence relation from the elements $A_{n+1}$:
\begin{align}\label{eqn:Xrec}
X_n=\wrap^{1-n} S A_{n+1} M &= \wrap^{1-n} S B_1 A_{n} M - \wrap^{1-n} S A_{n-1} \wrap_2 M \\
\nonumber
&= \wrap^{1-n} S B_1 A_{n} M - \wrap^{3-n} S A_{n-1} M  = \begin{cases}
v X_{n-1} - X_{n-2} & \text{even } n\\
u X_{n-1} - X_{n-2} & \text{odd } n
\end{cases} 
\end{align}
Here we have used $\wrap^{1-n}SB_1=v \wrap^{2-n}S$ for even $n$ and $\wrap^{1-n}SB_1=u \wrap^{2-n} S$ for odd $n$. Now it remains to check that the $Y_n$ also satisfy this recurrence \eqref{eqn:Xrec}. Indeed, for odd $N>4$ we can verify:
\begin{align*}
Y_n-uY_{n-1}+Y_{n-2} &= (u - u S_{(n-1)/2} u) - u(u - R_{n-2}u - u S_{(n-3)/2} u) + (u - u S_{(n-3)/2} u)\\
&=  u (-S_{(n-1)/2} + R_{n-2} + S_{(n-3)/2}) u =0
\end{align*}
Here we have used that $S_x-S_{x-1}= R_{2x-1}$. In order to check the recurrence for $Y_n$ in the case of even $n$ we need an auxiliary computation. For odd $x\geq 1$ we have
\begin{align*}S_x  &= (\id_2-v)  + (\id_2-v)(\id_2-u) S_{x-1}\\
  &= (\id_2-v)  + (\id_2-v)S_{x-1}  - u S_{x-1} + v u S_{x-1}  
\\
 &= 2\id_2 -v  + (\id_2-u)S_{x-2}  - u S_{x-1} + v u S_{x-1}  
\end{align*} which implies:
\begin{align*} v u S_{n/2-1} u &= - 2u + v u +S_{n/2} u- (\id_2-u) S_{n/2-2} u +u S_{n/2-1}u\\
&= 
-2 u + v u + R_{n-1} u+ R_{n-3}u + u S_{n/2-2} u +u S_{n/2-1}u
\end{align*}
Here we have used $(\id_2-v)S_x = \id_2 + (\id_2-u)S_{x-1}$. Now we check the recurrence for even $n>3$:
\begin{align*}
Y_n-v Y_{n-1}+Y_{n-2} &= (u - R_{n-1}u - u S_{n/2-1} u) - v (u - u S_{n/2-1} u) + (u - R_{n-3}u - u S_{n/2-2} u) \\
&= 2u -v u - R_{n-1}u -R_{n-3}u - u S_{n/2-2} u- u S_{n/2-1} u   + v u S_{n/2-1} u  =0 
\end{align*}
This completes the proof of the Lemma.
\end{proof}

\begin{proof}[Proof of Lemma~\ref{lem:crossabs}]
We only prove $T_2 s=T_2$, which is equivalent to $N T_2 u=0$ by expanding the crossing and multiplying by $N$. Using \eqref{eqn_Ttwob} we compute: 
\begin{align*} 
N T_2 u =  \sum_{k=0}^{N-1} ((\id_2-v)(\id_2-u))^k u = u - \sum_{k=1}^{N-1} ((\id_2-v)(\id_2-u))^{k-1} (\id_2-v)u = u - \sum_{k=1}^{N-1} R_{2k -1} u
\end{align*}
Now note that $\id_2 = (s \wrap)^{-N} (\wrap s)^N = ((\id_2-v)(\id_2-u))^N$ implies that 
\begin{equation}
\label{eqn:rflip}
R_{2k-1} = \left( (\id_2-u) R_{2N-2k-1} (\id_2-u)\right)^{-1} = (\id_2-u) R_{2N-2k-1} (\id_2-u).
\end{equation}
Now we distinguish two cases. For even $N$ we expand:
\begin{align*} 
N T_2 u  &= u - \sum_{k=1}^{N/2} R_{2k -1} u - \sum_{k=N/2+1}^{N-1} R_{2k -1} u
\\
&= u - \sum_{k=1}^{N/2} R_{2k -1} u - (\id_2-u)\sum_{l=1}^{N/2-1} R_{2l -1}(\id_2-u)u\\
&= u - \sum_{k=1}^{N/2} R_{2k -1} u + (\id_2-u)\sum_{l=1}^{N/2-1} R_{2l -1}u = u - R_{N-1} u - u S_{N/2-1} u = X_N 
\end{align*}
Here we have used \eqref{eqn:rflip} for the second equality and Lemma~\ref{lem:XY} for the last equality. For odd $N$ we expand analogously:
\begin{align*} 
N T_2 u  &= u - \sum_{k=1}^{(N-1)/2} R_{2k -1} u - \!\!\!\!\sum_{k=(N+1)/2+1}^{N-1} R_{2k -1} u\\
&= u - \sum_{k=1}^{(N-1)/2} R_{2k -1} u + (\id_2-u)\sum_{l=1}^{(N-1)/2} R_{2k -1} u = u- u S_{(N-1)/2}u = X_N
\end{align*}
We conclude the proof by noting that $X_N=\wrap^{1-N} S A_{N+1} M=0$ in $\essbAWeb[N]$ since $A_{N+1}=0$. 
\end{proof}

\begin{remark}
The expression $X_N$ in the rewritten form in Lemma~\ref{lem:XY} expresses the longest Kazhdan--Lusztig basis element $\underline{H}_{s t s\dots}=\underline{H}_{t s t\dots}$ in the type $I_2(N)$ Hecke algebra in terms of products of $\underline{H}_s:=u$ and $\underline{H}_t:=v$, see \cite[Section 2.3]{EW}. In particular, the relation $X^N=0$ suggests that $\essbAWeb[N](2)$ is related to a quotient of the Hecke algebra, by the $2$-cell containing the basis element associated to the longest word. 
\end{remark}

\section{Categorification of power-sum symmetric polynomials}
Before turning to the topological applications of our work, we will in this section focus on identifying more precisely the structures that are categorified by the categories defined before. The main result of this section consists in a categorification of the Newton's identities for power-sum and elementary symmetric polynomials (see Theorem \ref{thm:newton}).

Let $\X=\{X_1,\dots,X_N\}$ be an alphabet of $N$ variables and denote by $\Sym(\X):=\C[X_1,\dots,X_N]^{\mathfrak{S}_N}$ the ring of symmetric polynomials in $\X$. Recall that $\Sym(\X)\cong\C[e_1(\X),\dots, e_N(\X)]$, where $e_j(\X)$ denotes the $j^{th}$ elementary symmetric polynomial in $\X$. We use the notation $h_j(\X)$ for the $j^{th}$ complete symmetric polynomial.

\begin{definition} The split Grothendieck  group of an additive category $\mathcal{C}$ is the abelian group $K_0(\mathcal{C})$ defined as the quotient of the free abelian group spanned by the isomorphism classes $[X]$ of objects $X$ of $\mathcal{C}$, modulo the ideal generated by relations of the form $[A\oplus B]=[A]+[B]$ for objects $A$, $B$ of $\mathcal{C}$. 

If $\mathcal{C}$ is monoidal, then $K_0(\mathcal{C})$ inherits a unital ring structure with multiplication $[A]\cdot[B]:=[A\otimes B]$.
\end{definition}

The following lemma is classical.
\begin{lemma} There is an isomorphism \[K_0(\repp)\otimes\C \cong K_0(\Kar(\repp))\otimes\C\cong \Sym(\X)\cong \C[e_1(\X),\dots,e_N(\X)]\]
sending the classes of the fundamental representations $[\bVn^k V]$ to the elementary symmetric polynomials $e_k(\X)$. The class of the simple representation indexed by the partition $\lambda$ is then given by the Schur polynomial $\pi_\lambda(\X)$. If one includes duals, one obtains 
\[K_0(\rep)\otimes\C\cong K_0(\Kar(\rep))\otimes\C\cong \C[\X^{\pm 1}]^{\mathfrak{S}_N} \cong \C[e_1(\X),\dots,e_{N-1}(\X),e_N^{\pm 1}(\X)].\]
\end{lemma}

For example, the classes of the symmetric and anti-symmetric power representations are related as follows.\vspace{-3mm}
\begin{equation}
\label{eqn:Grassm} h_{m+1}(\X) = \sum_{i=1}^{m+1} (-1)^i h_{m+1-i}(\X) e_i(\X)
\end{equation}

 This can also be seen in the Grothendieck group of $\Web[N]$, at the cost of passing to the Karoubi envelope. For this, we recall the symmetric clasps~\cite{Kup}, which are higher-rank analogs of Jones--Wenzl projectors, and their anti-symmetric counterparts.

\begin{definition} \label{defn:clasps} The symmetric and anti-symmetric clasps $P_m\in \Web[N]$ and $V_m\in \Web[N]$ are defined by $P_1=V_1=\id_1$ and then:
 \[
      \begin{tikzpicture}[anchorbase, scale=.3]
\fill[black,opacity=.2] (0,1) rectangle (3,3);
\draw[thick] (0,1) rectangle (3,3);
\draw [very thick] (.5,0) to (.5,1);
\draw [thick, dotted] (.7,0.5) to (1.3,.5);
\draw [very thick] (1.5,0) to (1.5,1);
\draw [very thick] (2.5,0) to (2.5,1);
\draw [very thick,->] (.5,3) to (.5,4);
\draw [thick, dotted] (.7,3.25) to (1.3,3.25);
\draw [very thick,->] (1.5,3) to (1.5,4);
\draw [very thick,->] (2.5,3) to (2.5,4);
\node at (1.5,1.9) {$P_{m+1}$};
\end{tikzpicture}
   \;\; := \;\;
   \begin{tikzpicture}[anchorbase, scale=.3]
\fill[black,opacity=.2] (0,1) rectangle (2,3);
\draw[thick] (0,1) rectangle (2,3);
\draw [very thick] (.5,0) to (.5,1);
\draw [thick, dotted] (.7,0.5) to (1.3,.5);
\draw [very thick] (1.5,0) to (1.5,1);
\draw [very thick,->] (2.5,0) to (2.5,4);
\draw [very thick,->] (.5,3) to (.5,4);
\draw [thick, dotted] (.7,3.25) to (1.3,3.25);
\draw [very thick,->] (1.5,3) to (1.5,4);
\node at (1,1.9) {$P_{m}$};
\end{tikzpicture}
    \;-\;\frac{m}{m+1}\;
       \begin{tikzpicture}[anchorbase, scale=.3]
\fill[black,opacity=.2] (0,.5) rectangle (2,1.5);
\draw[thick] (0,.5) rectangle (2,1.5);
\fill[black,opacity=.2] (0,2.5) rectangle (2,3.5);
\draw[thick] (0,2.5) rectangle (2,3.5);
\draw [very thick] (.5,0) to (.5,.5);
\draw [very thick] (1.5,0) to (1.5,.5);
\draw [double] (2,1.75) to (2,2.25);
\draw [very thick] (2.5,0) to (2.5,1.5)to [out=90,in=90] (1.5,1.5);
\draw [very thick,->] (1.5,2.5) to [out=270,in=270] (2.5,2.5) to (2.5,4); 
\draw [very thick] (.5,1.5) to (.5,2.5);
\draw [thick, dotted] (.7,2) to (1.3,2);
\draw [very thick,->] (.5,3.5) to (.5,4);
\draw [very thick,->] (1.5,3.5) to (1.5,4);
\node at (1,.95) {\tiny$P_{m}$};
\node at (1,2.95) {\tiny$P_{m}$};
\end{tikzpicture}   
\quad,\quad
   \begin{tikzpicture}[anchorbase, scale=.3]
\fill[black,opacity=.2] (0,1) rectangle (3,3);
\draw[thick] (0,1) rectangle (3,3);
\draw [very thick] (.5,0) to (.5,1);
\draw [thick, dotted] (.7,0.5) to (1.3,.5);
\draw [very thick] (1.5,0) to (1.5,1);
\draw [very thick] (2.5,0) to (2.5,1);
\draw [very thick,->] (.5,3) to (.5,4);
\draw [thick, dotted] (.7,3.25) to (1.3,3.25);
\draw [very thick,->] (1.5,3) to (1.5,4);
\draw [very thick,->] (2.5,3) to (2.5,4);
\node at (1.5,1.9) {$V_{m+1}$};
\end{tikzpicture}
   \;\; := \;\; \frac{1}{(m+1)!}\;
  \begin{tikzpicture}[anchorbase, scale=.3]
\draw [very thick,->] (.5,0) to (.5,4);
\draw [very thick] (1.5,0) to (1.5,.2) to [out=90,in=315] (.5,1);
\draw [very thick] (2.5,0) to (2.5,.2) to[out=90,in=315] (.5,1.5);
\draw [very thick,->] (.5,2.5) to [out=45,in=270](2.5,3.8) to (2.5,4);
\draw [very thick,->] (.5,3) to [out=45,in=270](1.5,3.8)to (1.5,4);
\end{tikzpicture} 
\]

Note that the clasps are related by:
  \[
      \begin{tikzpicture}[anchorbase, scale=.3]
\fill[black,opacity=.2] (0,1) rectangle (3,3);
\draw[thick] (0,1) rectangle (3,3);
\draw [very thick] (.5,0) to (.5,1);
\draw [thick, dotted] (.7,0.5) to (1.3,.5);
\draw [very thick] (1.5,0) to (1.5,1);
\draw [very thick] (2.5,0) to (2.5,1);
\draw [very thick,->] (.5,3) to (.5,4);
\draw [thick, dotted] (.7,3.25) to (1.3,3.25);
\draw [very thick,->] (1.5,3) to (1.5,4);
\draw [very thick,->] (2.5,3) to (2.5,4);
\node at (1.5,1.9) {$P_{m+1}$};
\end{tikzpicture}
\;=\;
\begin{tikzpicture}[anchorbase, scale=.3]
\fill[black,opacity=.2] (0,1) rectangle (2,3);
\draw[thick] (0,1) rectangle (2,3);
\draw [very thick] (.5,0) to (.5,1);
\draw [thick, dotted] (.7,0.5) to (1.3,.5);
\draw [very thick] (1.5,0) to (1.5,1);
\draw [very thick,->] (2.5,0) to (2.5,4);
\draw [very thick,->] (.5,3) to (.5,4);
\draw [thick, dotted] (.7,3.25) to (1.3,3.25);
\draw [very thick,->] (1.5,3) to (1.5,4);
\node at (1,1.9) {$P_{m}$};
\end{tikzpicture}
    \;-\;\frac{2 m}{m+1}
       \begin{tikzpicture}[anchorbase, scale=.3]
\fill[black,opacity=.2] (0,.5) rectangle (2,1.5);
\draw[thick] (0,.5) rectangle (2,1.5);
\fill[black,opacity=.2] (0,2.5) rectangle (2,3.5);
\draw[thick] (0,2.5) rectangle (2,3.5);
\fill[black,opacity=.2] (1.5,1.5) rectangle (3,2.5);
\draw[thick] (1.5,1.5) rectangle (3,2.5);
\draw [very thick] (.5,0) to (.5,.5);
\draw [very thick] (1.5,0) to (1.5,.5);
\draw [very thick] (2.5,0) to (2.5,1.5);
\draw [very thick] (.5,1.5) to (.5,2.5);
\draw [thick, dotted] (.7,2) to (1.3,2);
\draw [very thick,->] (.5,3.5) to (.5,4);
\draw [very thick,->] (1.5,3.5) to (1.5,4);
\draw [very thick,->] (2.5,2.5) to (2.5,4);
\node at (1,.95) {\tiny$P_{m}$};
\node at (2.25,1.95) {\tiny$V_{2}$};
\node at (1,2.95) {\tiny$P_{m}$};
\end{tikzpicture} 
\quad,\quad
      \begin{tikzpicture}[anchorbase, scale=.3]
\fill[black,opacity=.2] (0,1) rectangle (3,3);
\draw[thick] (0,1) rectangle (3,3);
\draw [very thick] (.5,0) to (.5,1);
\draw [thick, dotted] (.7,0.5) to (1.3,.5);
\draw [very thick] (1.5,0) to (1.5,1);
\draw [very thick] (2.5,0) to (2.5,1);
\draw [very thick,->] (.5,3) to (.5,4);
\draw [thick, dotted] (.7,3.25) to (1.3,3.25);
\draw [very thick,->] (1.5,3) to (1.5,4);
\draw [very thick,->] (2.5,3) to (2.5,4);
\node at (1.5,1.9) {$V_{m+1}$};
\end{tikzpicture}
   \;\; = \;\;
   \begin{tikzpicture}[anchorbase, scale=.3]
\fill[black,opacity=.2] (0,1) rectangle (2,3);
\draw[thick] (0,1) rectangle (2,3);
\draw [very thick] (.5,0) to (.5,1);
\draw [thick, dotted] (.7,0.5) to (1.3,.5);
\draw [very thick] (1.5,0) to (1.5,1);
\draw [very thick,->] (2.5,0) to (2.5,4);
\draw [very thick,->] (.5,3) to (.5,4);
\draw [thick, dotted] (.7,3.25) to (1.3,3.25);
\draw [very thick,->] (1.5,3) to (1.5,4);
\node at (1,1.9) {$V_{m}$};
\end{tikzpicture}
    \;-\;\frac{2 m}{m+1}
       \begin{tikzpicture}[anchorbase, scale=.3]
\fill[black,opacity=.2] (0,.5) rectangle (2,1.5);
\draw[thick] (0,.5) rectangle (2,1.5);
\fill[black,opacity=.2] (0,2.5) rectangle (2,3.5);
\draw[thick] (0,2.5) rectangle (2,3.5);
\fill[black,opacity=.2] (1.5,1.5) rectangle (3,2.5);
\draw[thick] (1.5,1.5) rectangle (3,2.5);
\draw [very thick] (.5,0) to (.5,.5);
\draw [very thick] (1.5,0) to (1.5,.5);
\draw [very thick] (2.5,0) to (2.5,1.5);
\draw [very thick] (.5,1.5) to (.5,2.5);
\draw [thick, dotted] (.7,2) to (1.3,2);
\draw [very thick,->] (.5,3.5) to (.5,4);
\draw [very thick,->] (1.5,3.5) to (1.5,4);
\draw [very thick,->] (2.5,2.5) to (2.5,4);
\node at (1,.95) {\tiny$V_{m}$};
\node at (2.25,1.95) {\tiny$P_{2}$};
\node at (1,2.95) {\tiny$V_{m}$};
\end{tikzpicture} 
\]
\end{definition}
It is well-known that $\phi$ sends $P_m$ and $V_m$ to the projections onto simple representations in $V^{\otimes m}$ given by the $m$-fold symmetric and anti-symmetric powers of the vector representation respectively.

\begin{theorem}
  In $\Kar(\Web[N])$ there is an isomorphism
  \[
  \oplus_{i=0}^{\lfloor \frac{k-1}{2}\rfloor} (k,V_{k-1-2i}\otimes P_{2i+1})\simeq \bigoplus \oplus_{i=0}^{\lfloor\frac{k}{2}\rfloor}(k,V_{k-2i}\otimes P_{2i})
  \]
  which categorifies \eqref{eqn:Grassm}.
\end{theorem}
The proof of this is similar to but easier than the proof of Theorem~\ref{thm:newton} below, and thus omitted.

\subsection{Categorified Newton's identities}
\label{sec:NI}

We now explicitly show that the projectors $T_m$ categorify the power-sum symmetric polynomials $p_m(\X)= X_1^m+\cdots + X_N^m$ in the same sense as the clasps $P_m$ categorify the complete symmetric polynomials. To this end, we prove that the projectors $T_m$ satisfy categorified versions of the classical Newton identities:
\begin{align}
\label{eqn:newtonid}
p_k(\X) &= (-1)^{k-1}k e_k(\X) - \sum_{j=1}^{k-1} (-1)^{k-j} e_{k-j}(\X) p_j(\X)\quad \text{for }  1\leq k
\end{align}

\begin{theorem}
\label{thm:newton}
  In $\Kar(\essAWeb[N])^*$, there is an isomorphism:
  \[
  \oplus_{i=0}^{\lfloor \frac{k-1}{2}\rfloor} (k,V_{k-1-2i}\otimes T_{2i+1})\simeq \oplus_{k}(k,V_k)\bigoplus \oplus_{i=1}^{\lfloor\frac{k}{2}\rfloor}(k,V_{k-2i}\otimes T_{2i})
  \]
\end{theorem}

\begin{proof}
The desired isomorphism takes the shape of a ``zig-zag'', i.e. each direct summand maps non-trivially to at most two direct summands on the other side. A typical segment of the zig-zag looks as follows: 
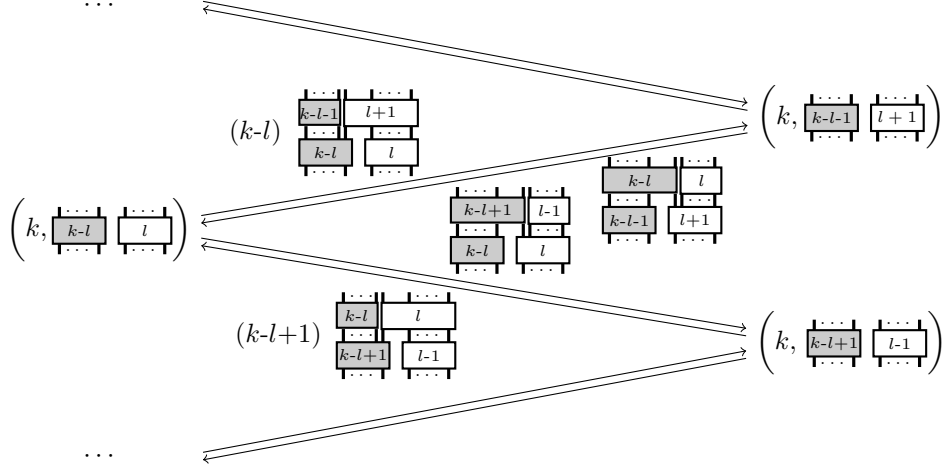
\[

  \]
  
  Above, gray-colored boxes stand for the anti-symmetric clasps $V_m$, as pictured in Definition \ref{defn:clasps}, while the other rectangles are our extremal weights projectors $T_m$. The isomorphism takes this form locally only when $l\geq 2$, and we will deal with the $l=1$ term at the end. 

We check that the composite of the maps to the right with the maps to the left induce the identity on the components of the left. For this, we compute
  \begin{align*}
      (k-l+1)\;\;
      %
    \end{align*}  
    by expanding the middle projectors on the left-hand side using the recursions in \eqref{eq:recursive} and \ref{defn:clasps}. Adding both equations, we obtain the desired equality:
      \begin{align*}
      (k-l+1)\;\;
      %
    \end{align*}
  We also check that all other components of this endomorphism of the left-hand side are zero:
  \[(k-l)\; %
       
      \]
These expressions are zero because the middle projectors can be absorbed into the top and bottom projectors respectively. The results have anti-symmetric clasps and extremal weight projectors that share two strands, which forces them to equal zero.
    
    To finish the proof, we need to look at the top and bottom ends of the zig-zag. The top end, which contains $(k, T_k)$ is treated precisely as in the generic case. The bottom end is more interesting, as it involves $k$ copies of the object $(k, V_k)$.
    
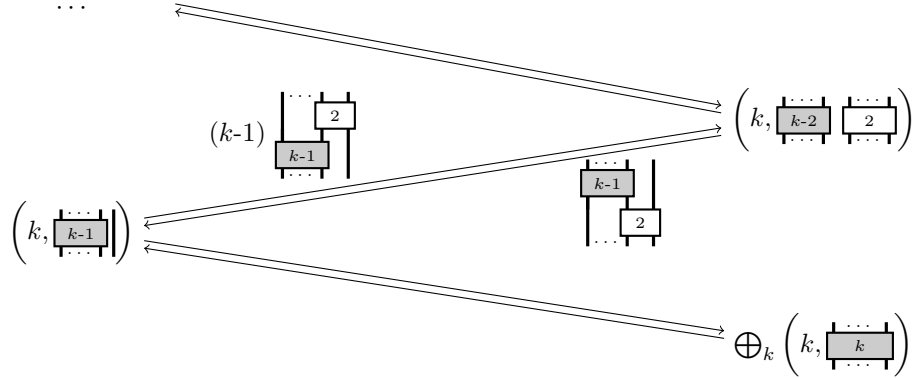
 \[
  %
  \]   
    
     In order to obtain the desired isomorphism, one needs to express the following term as a sum of $k$ orthogonal projections onto $k$ copies of the exterior power:
    \[
            %
  
         \]

         To this end, we introduce the following diagram as corresponding to taking the polynomial $P_i$ ($i=0\dots N-1$) in the wrap, thus projecting onto the $i$-th eigenspace of the wrap: 
         \[
         %
         \]
and more generally in the case of $k$ parallel strands, for a tuple $I=(i_1,\dots,i_k)$ we use the diagrams
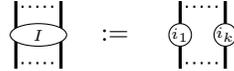
\[
         %
\]
which correspond to the projectors $P_{i_1}\otimes \cdots \otimes P_{i_k}$. In the following we consider subsets $A\subset \{0, \dots, N-1\}$ as tuples of distinct elements, with the usual ordering. We write $s A$ for the tuple obtained from this by acting with a permutation $s$. The corresponding projectors satisfy:
\[\comm{
%
 
\quad\quad \text{as well as} \quad\quad}
%
 
\;\; = \;\; 0 \quad \text{ if } i_l=i_m \text{ for } l\neq m
\]

Now one can write:
\begin{equation}
\label{eqn:lastterm}
%
\end{equation}

Here we have expanded the identity web between the two anti-symmetric clasps into a sum of projectors, noting that the clasps will kill $k$-tuples that have two equal elements in the first $k-1$ entries. The two summands correspond to the two alternatives of having a last entry $x$ that is distinct from the first $k-1$ ones, or a repeat. Now, we make use of the following equality:
\[
%
\]
to rewrite the two summands in \eqref{eqn:lastterm}. The first sum is partitioned into $k$ terms, the $j$-th of which contains those projectors $s A$ that project to the $j$-th largest element of $A$ on the right-most strand. For a fixed $A$ and $j$, there are $(k-1)!$ permutations $s$ such that $s A$ is of this type, each of which produces an identical summand by the previous equation.   

In the second summand, we similarly reorder the $x$ in $s B$ to the right-most strand, which collects together $k-1$ identical terms. 
\begin{align}\label{eqn:newtonorthog}
%
\quad &= \quad 
\left(
\sum_{j=1}^k \hspace{-.3cm}
\sum_{\substack{A\subset \{0, \dots, N-1\}\\ |A|=k,\\ x\;j\text{-th largest entry of A}}} \hspace{-.5cm}
\;(k-1)!\;
%
\right) 
\end{align}
The remaining summation in the second term is over $(k-2)$-element subsets $C$ and their permutations. Faithfulness of $\phi$ implies that the anti-symmetric clasps absorb this sum of projectors and so we get:
\[
(k-1)
\left(
\sum_{\substack{x\in \{0,\dots,N-1\} \\C\subset \{0,\dots,N-1\}\\ |C|=k-2\\  s\in \mathfrak{S}_{k-2}}}
%
\]

The $k$ terms in the first summand in \eqref{eqn:newtonorthog} are clearly orthogonal to each other and to the second summand, from which it follows that they are idempotents. It remains to argue that they are the are isomorphic to anti-symmetric clasps in the Karoubi envelope. To this end, we fix a $j$ and argue that
\[\sum_{\substack{A\subset \{0, \dots, N-1\}\\ |A|=k,\\ x\;j\text{-th largest entry of A}}} \hspace{-.5cm}
\;(k-1)!\;
%
 \quad \text{ and } \quad 
\sum_{\substack{A\subset \{0, \dots, N-1\}\\ |A|=k,\\ x\;j\text{-th largest entry of A}}} \hspace{-.5cm}
\;k!\;
%
\] 
give the desired inverse isomorphisms. So we simplify both composites:
\[\sum_{\substack{A,B\subset \{0, \dots, N-1\}\\ |A|=|B|=k,\\ x\;j\text{-th largest entry of A}\\ y\;j\text{-th largest entry of B}}}\hspace{-.5cm}
\;k!(k-1)!\;
%
\;=  \!\!\!\!
\sum_{\substack{A\subset \{0, \dots, N-1\}\\ |A|=k,\\ x\;j\text{-th largest entry of A}}} \hspace{-.5cm}
\;k!\;
%
\] 
Summands in the composite of the left-hand side are zero unless $x=y$ and $A\setminus x = B \setminus y$ and the anti-symmetric clasp is absorbed at the cost of dividing by $(k-1)!$. The second equality is just a reordering, while the third uses that the anti-symmetric projectors absorb the sum of all $A$-projectors at the cost of dividing by $k!$. For the other composite we get:
\[\sum_{\substack{A,B\subset \{0, \dots, N-1\}\\ |A|=|B|=k,\\ x\;j\text{-th largest entry of A}\\ y\;j\text{-th largest entry of B}}}\hspace{-.5cm}
\;k!(k-1)!\;
%
\;=  \!\!\!\!\sum_{\substack{A\subset \{0, \dots, N-1\}\\ |A|=k,\\ x\;j\text{-th largest entry of A}}} \hspace{-.5cm}
\;(k-1)!\;
%
\]
Here we have used the same absorption property as in the first step in the computation of the other composite.
\end{proof}

\begin{remark} In the case $k=N$ we can give an alternative characterization of the more mysterious part of the isomorphism in Theorem~\ref{thm:newton}, which involves the $N$-fold direct sum of objects $(N, V_N)$. Indeed, after \eqref{eqn:newtonorthog} and the following displayed equation, the goal was to decompose
\[%
    \] 
    into a sum of $N$ orthogonal idempotents, which are individually isomorphic to $(N, V_N)$. It is easy to check that the above is equal to
    \[\sum_{x=1}^N (\id_{N-1}\otimes \wrap^{-x})V_N(\id_{N-1}\otimes \wrap^{x}),\] which is manifestly a sum of idempotents, which are orthogonal since $V_N(\id_{N-1}\otimes \wrap^{k-l})V_N=\delta_{k,l} V_N$. 
     
\end{remark}

\begin{question}
Can the extremal weight projectors and symmetric clasps be used to give categorifications of the following identities?
\begin{equation}\label{eqn:newtonid2}
k h_k(\X) = \sum_{j=1}^{k} h_{k-j}(\X) p_j(\X)\quad \text{for } 1\leq k\leq N
\end{equation}
\end{question}

An isomorphism categorifying this identity for $k=2$ is easy to construct. For $k\geq 3$ such an isomorphism cannot be of zig-zag shape as for \eqref{eqn:newtonid}.

\subsection{Categorification of the symmetric polynomial ring}\label{sec:decat}
An easy consequence of Corollary~\ref{cor:diagrep} is the following.
\begin{lemma}\label{lem:K0} There are isomorphisms
\[K_0(\Kar(\essAWebp[N])^*)\otimes \C \cong K_0(\rephp) \otimes \C \cong \C[\X]\]
sending the object $(1,P_i)$ to $[\C\la v_i \ra]$ and further to $X_i$.
\end{lemma}
We have seen that the extremal weight projectors in $\essAWebp[N]$ categorify the power sum symmetric polynomials. However, by Lemma~\ref{lem:K0} the Grothendieck group of the Karoubi envelope of $\essAWebp[N]$ is larger than the symmetric polynomial ring $\Sym(\X)\cong K_0(\Kar(\Webp[N]))\cong K_0(\repp)$. To see this, recall that the objects in $\rephp$ are direct sums of non-negative integral $\glnn{N}$ weight spaces. However, in the Grothendieck group, such direct sums can be written as formal differences of $\glnn{N}$-representations only if they are orbits of the action for the Weyl group $\mathfrak{S}_N$. In this section, we identify a sub-category of $\essAWebp[N]$ that is $\mathfrak{S}_N$-equivariant, that contains the extremal weight projectors and has $\Sym(\X)$ as Grothendieck group. 

\begin{definition} We let $\reph^{\mathfrak{S}_N}$ denote the subcategory of $\reph$ with objects that are invariant under $\mathfrak{S}_N$ and morphisms that are $\mathfrak{S}_N$-equivariant.  
\end{definition}  

\begin{lemma}\label{lem:sympolycat} The category $\reph^{\mathfrak{S}_N}$ is semi-simple and the homomorphism \[
\Sym(\X)\cong  K_0(\repp)\otimes \C \to K_0(\reph^{\mathfrak{S}_N})\otimes \C\] induced by the inclusion is an isomorphism.
\end{lemma}
\begin{proof}
The indecomposable objects in $\reph^{\mathfrak{S}_N}$ are orbits of the form $\C\langle v_{s(\epsilon_1),\dots, s(\epsilon_{n})}| s\in \mathfrak{S}_N \rangle$. Through the braiding, such an object is isomorphic to an $\mathfrak{S}_N$-orbit of a vector $v_{0,\dots,0,1,\dots,1, \dots, N-1}$ with multiplicities $n_i$ of the weights $i$ determined by a partition $ \lambda\colon n_0\geq n_1 \geq \cdots \geq n_{N-1}$ of $n$. There are no morphisms between distinct indecomposables and their endomorphism algebras are $1$-dimensional over $\C$. This shows that $\reph^{\mathfrak{S}_N}$ is semi-simple. The isomorphism follows since the classes of these indecomposables can be expressed as linear combinations of the classes of tensor products of fundamental representations in the same way as monomial symmetric polynomials can be expressed as polynomials in elementary symmetric polynomials.
\end{proof}

We aim to describe the subcategory $\reph^{\mathfrak{S}_N}$ of $\reph$ by a subcategory of $\essAWebp[N]$.

\begin{definition}
Let $\sessAWeb$ denote the symmetric monoidal $\C$-linear subcategory of $\essAWebp[N]$ with the same objects, but with morphisms spaces generated (under tensor product and composition) by morphisms in $\Webp[N]$ and the extremal weight projectors $T_m$ for $m\geq 1$.
\end{definition}
Note that the restriction of $\phi$ to the subcategory $\sessAWeb$ has image contained in $\reph^{\mathfrak{S}_N}$.

\begin{proposition}
The functor $\phi \colon \sessAWeb \to \reph^{\mathfrak{S}_N}$  is fully faithful and induces an equivalence of $\C$-linear monoidal categories $\Kar(\sessAWeb) \simeq\reph^{\mathfrak{S}_N}$. 
\end{proposition}
\begin{proof} Faithfulness is inherited from Theorem~\ref{thm:faithfulness}. We shall prove fullness by showing that the image of $\phi$ contains the projections onto the simple objects in $\reph^{\mathfrak{S}_N}$ as identified in the proof of Lemma~\ref{lem:sympolycat}. Indeed, if $\lambda\colon n_0\geq n_1 \geq \cdots \geq n_{N-1}$ of $n$, then we will construct an idempotent morphism in $\sessAWeb$ that projects onto the $\mathfrak{S}_N$-orbit of the vector $v_{0,\dots,0,1,\dots,1, \dots, N-1}$ with weights $i$ appearing with multiplicities $n_i$.

To this end, we first define an auxiliary projector $O_n$ in $\sessAWeb$ for the case where $n_i\in \{0,1\}$ for $1\leq i \leq N$. We set $O_1=\id_1$ and $O_2= \id_2 - T_2$. For $n\geq 2$ we inductively define:

\[O_{n+1} := s_1 (\id_1\otimes O_n)s_1 (\id_1\otimes O_n)(O_n\otimes \id_1)  \]
It is easy to check that the image of $O_n$ under $\phi$ is the desired projection, and so the $O_n$ are the desired diagrammatic idempotents by faithfulness of $\phi$. It is also clear that the $O_n$ are contained in $\sessAWeb$.

Now let $\lambda\colon n_0\geq n_1 \geq \cdots \geq n_{k}$ be a partition of $n$ with $k$ non-zero parts $n_i$. Then consider the projector built as the composite of  $T_{n_1}\otimes \cdots\otimes  T_{n_k}$, the permutation given as the product of transpositions that interchange the strands in position $n_i$ and $n-i$ for $1 \leq i \leq k$, the projector $\id_{n-k}\otimes O_k$, the inverse permutation and again $T_{n_1}\otimes \cdots \otimes T_{n_k}$.

The image of this element under $\phi$ is the idempotent projecting onto the $\mathfrak{S}_N$-orbit of the vector $v_{0,\dots,0,1,\dots,1, \dots, N-1}$ with weights $i$ of multiplicities $n_i$, and by faithfulness of $\phi$ it is itself an idempotent in $\sessAWeb$.

\end{proof}

\begin{corollary} $\Kar(\sessAWeb)$ categorifies the symmetric polynomial ring $\Sym(\X)$ and its objects $T_m$ categorify the power-sum symmetric polynomials.
\end{corollary}

\begin{remark} The isomorphisms of Theorem~\ref{thm:newton}, which categorify the Newton identities, holds in $\sessAWeb$, although the direct sum decomposition $\bigoplus_k(k, V_k)$ on the right-hand side is not $\mathfrak{S}_N$-equivariant.
\end{remark}

\section{Special properties of the \texorpdfstring{$\glnn{2}$}{gl(2)} case}
\label{sec:two}
We now review some of the special properties of the extremal weight projectors in the $N=2$ case. In this context, we encode $2$-labeled edges in webs as double edges and henceforth omit the labels. For convenience, we list the $\glnn{2}$ web relations for generic $q$ separately. 

\begin{gather}
\label{eqn:circles}
\begin{tikzpicture}[fill opacity=.2,anchorbase,scale=.3]
\draw[very thick, directed=.55] (1,0) to [out=0,in=270] (2,1) to [out=90,in=0] (1,2)to [out=180,in=90] (0,1)to [out=270,in=180] (1,0);
\end{tikzpicture} 
\quad=\quad
(q+ q^{-1}) \emptyset
\quad=\quad 
\begin{tikzpicture}[fill opacity=.2,anchorbase,scale=.3]
\draw[very thick, rdirected=.55] (1,0) to [out=0,in=270] (2,1) to [out=90,in=0] (1,2)to [out=180,in=90] (0,1)to [out=270,in=180] (1,0);
\end{tikzpicture}
\quad,\quad
\begin{tikzpicture}[fill opacity=.2,anchorbase,scale=.3]
\draw[double, directed=.55] (1,0) to [out=0,in=270] (2,1) to [out=90,in=0] (1,2)to [out=180,in=90] (0,1)to [out=270,in=180] (1,0);
\end{tikzpicture} 
\quad=\quad
\emptyset
\quad=\quad \begin{tikzpicture}[fill opacity=.2,anchorbase,scale=.3]
\draw[double, rdirected=.55] (1,0) to [out=0,in=270] (2,1) to [out=90,in=0] (1,2)to [out=180,in=90] (0,1)to [out=270,in=180] (1,0);
\end{tikzpicture}
\\
\label{eqn:bigons}
\begin{tikzpicture}[anchorbase, scale=.5]
\draw [double] (.5,0) -- (.5,.3);
\draw [very thick] (.5,.3) .. controls (.4,.35) and (0,.6) .. (0,1) .. controls (0,1.4) and (.4,1.65) .. (.5,1.7);
\draw [very thick] (.5,.3) .. controls (.6,.35) and (1,.6) .. (1,1) .. controls (1,1.4) and (.6,1.65) .. (.5,1.7);
\draw [double, ->] (.5,1.7) -- (.5,2);
\end{tikzpicture}
\quad= \quad
(q+q^{-1})\;
\begin{tikzpicture}[anchorbase, scale=.5]
\draw [double,->] (.5,0) -- (.5,2);
\end{tikzpicture}
\quad,\quad
\begin{tikzpicture}[anchorbase, scale=.5]
\draw [very thick] (.5,0) -- (.5,.3);
\draw [very thick] (.5,.3) .. controls (.4,.35) and (0,.6) .. (0,1) .. controls (0,1.4) and (.4,1.65) .. (.5,1.7);
\draw [double, directed=0.55] (.5,.3) .. controls (.6,.35) and (1,.6) .. (1,1) .. controls (1,1.4) and (.6,1.65) .. (.5,1.7);
\draw [very thick, ->] (.5,1.7) -- (.5,2);
\end{tikzpicture}
\quad= \quad
\begin{tikzpicture}[anchorbase, scale=.5]
\draw [very thick,->] (.5,0) -- (.5,2);
\end{tikzpicture}
\quad= \quad
\begin{tikzpicture}[anchorbase, scale=.5]
\draw [very thick] (.5,0) -- (.5,.3);
\draw [double, directed=0.55] (.5,.3) .. controls (.4,.35) and (0,.6) .. (0,1) .. controls (0,1.4) and (.4,1.65) .. (.5,1.7);
\draw [very thick] (.5,.3) .. controls (.6,.35) and (1,.6) .. (1,1) .. controls (1,1.4) and (.6,1.65) .. (.5,1.7);
\draw [very thick, ->] (.5,1.7) -- (.5,2);
\end{tikzpicture}
\\
\label{eqn:squares}
\begin{tikzpicture}[anchorbase,scale=.5]
\draw [double] (0,0) -- (0,0.5);
\draw [very thick] (1,0) -- (1,.7);
\draw [very thick] (0,0.5) -- (1,.7);
\draw [double] (1,.7) -- (1,1.3);
\draw [very thick] (0,.5) -- (0,1.5);
\draw [very thick] (1,1.3) -- (0,1.5);
\draw [double,->] (0,1.5) -- (0,2);
\draw [very thick, ->] (1,1.3) -- (1,2);
\end{tikzpicture}
\quad = \quad
\begin{tikzpicture}[anchorbase,scale=.5]
\draw [double,->] (0,0) -- (0,2);
\draw [very thick,->] (1,0) -- (1,2);
\end{tikzpicture}
\quad,\quad
\begin{tikzpicture}[anchorbase,scale=.5]
\draw [double] (1,0) -- (1,0.5);
\draw [very thick] (0,0) -- (0,.7);
\draw [very thick] (1,0.5) -- (0,.7);
\draw [double] (0,.7) -- (0,1.3);
\draw [very thick] (1,.5) -- (1,1.5);
\draw [very thick] (0,1.3) -- (1,1.5);
\draw [double,->] (1,1.5) -- (1,2);
\draw [very thick, ->] (0,1.3) -- (0,2);
\end{tikzpicture}
\quad = \quad
\begin{tikzpicture}[anchorbase,scale=.5]
\draw [double,->] (1,0) -- (1,2);
\draw [very thick,->] (0,0) -- (0,2);
\end{tikzpicture}
\quad , \quad
\begin{tikzpicture}[anchorbase,scale=.5]
\draw [double,->] (0,0) to  (0,2);
\draw [double,->] (1,2) to (1,0);
\end{tikzpicture}
\quad =\quad
\begin{tikzpicture}[anchorbase,scale=.5]
\draw [double,->] (0,0) to (0,.5) to [out=90,in=90] (1,.5) to (1,0);
\draw [double,->] (1,2) to (1,1.5) to [out=270,in=270] (0,1.5) to (0,2);
\end{tikzpicture}
\quad , \quad
\begin{tikzpicture}[anchorbase,scale=.5]
\draw [double,<-] (0,0) to  (0,2);
\draw [double,<-] (1,2) to (1,0);
\end{tikzpicture}
\quad =\quad
\begin{tikzpicture}[anchorbase,scale=.5]
\draw [double,<-] (0,0) to (0,.5) to [out=90,in=90] (1,.5) to (1,0);
\draw [double,<-] (1,2) to (1,1.5) to [out=270,in=270] (0,1.5) to (0,2);
\end{tikzpicture}
\end{gather}

In fact, $\glnn{2}$ webs satisfy generalizations of the $1$-labeled circle relation in \eqref{eqn:circles} that we now describe.

\begin{lemma}[The delooping lemma]
\label{lem:neckcut}
Let $W$ be a $\glnn{2}$-web (in a disc or some other surface ) and denote by $c(W)$ the unoriented multi-curve obtained by erasing all $2$-labeled edges. Suppose that $c(W)$ contains a circle $c$ which bounds a disc $\cat{D}$ in the complement of $c(W)$. Then $W = (q+q^{-1})V$, where $V$ is a web that agrees with $W$ outside a neighborhood of the disc $\cat{D}$ and with underlying curve $c(V)$ obtained by removing the circle in question from $c(W)$.
\end{lemma}
\begin{proof} We only consider $W$ in a neighborhood of the disc $\cat{D}$ bounded by $c$. We will find a sequence of web relations which reduce the interaction of 2-labeled edges with $c$ until $c$ can be removed via a the first relation in~\eqref{eqn:circles}.  There are three types of interaction of $c$ with $2$-labeled edges to consider in sequence:
\begin{enumerate}
\item Any 2-labeled circle contained in $\cat{D}$ can be removed using one of the relations in \eqref{eqn:circles}, starting with an innermost one.
\item Suppose there exists a $2$-labeled edge in the interior of $\cat{D}$ with boundary on $c$. We take an innermost such edge, i.e. one which encloses a region in the disc with no other 2-labeled edges in the interior. Such an intersection edge can be removed via the bigon relations in \eqref{eqn:bigons}, provided there are no 2-labeled edges hitting the boundary of $\cat{D}$ from the outside in the relevant region. Otherwise, jump to (3) to remove external edges first. Note that they always come in pairs for orientation reasons.
\item There is a pair of 2-labeled edges, hitting $c$ from the outside $\cat{D}$, which are adjacent in the sense that an arc along $c$ connects them without hitting other $2$-labeled edges. Then one application of the saddle relations in \eqref{eqn:squares} creates a 2-labeled edge connecting two points on $c$ from the outside (see the right side of Figure~\ref{fig:compdiscint}), which can be removed as in (2). 
\end{enumerate}
This algorithm relates $W$ to a web that contains $c$ as an oriented $1$-labeled circle that can be removed via \eqref{eqn:circles}.
\end{proof}

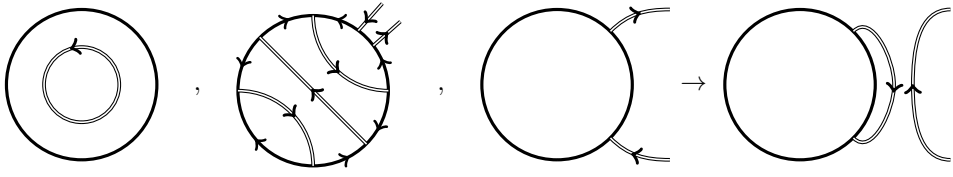
\begin{figure}[h]
\label{fig:compdiscint}
\begin{tikzpicture}[fill opacity=.2,anchorbase]
\draw[very thick] (1,0) to [out=0,in=270](2,1) to [out=90,in=0] (1,2) to [out=180,in=90] (0,1) to [out=270,in=180] (1,0);
\draw[double,directed=.55] (1,.5) to [out=0,in=270](1.5,1) to [out=90,in=0] (1,1.5) to [out=180,in=90] (.5,1) to [out=270,in=180] (1,.5);
\end{tikzpicture}
\quad , \quad
\begin{tikzpicture}[fill opacity=.2,anchorbase]
\draw[very thick,directed=.08,rdirected=.18,directed=.32,directed=.45,rdirected=.57,directed=.70,rdirected=.88] (1,0) to [out=0,in=270](2,1) to [out=90,in=0] (1,2) to [out=180,in=90] (0,1) to [out=270,in=180] (1,0);
\draw[double,directed=.55] (0,1) to [out=0,in=90](1,0);
\draw[double,directed=.55] (1,2) to [out=270,in=180](2,1);
\draw[double,rdirected=.55] (0.29,1.71) to (1.71,0.29);
\draw[double,directed=.55] (1.8,1.6) to (2.16,1.92);
\draw[double,rdirected=.55] (1.6,1.8) to (1.92,2.16);
\end{tikzpicture}
\quad , \quad
\begin{tikzpicture}[fill opacity=.2,anchorbase]
\draw[very thick] (1,0) to [out=0,in=270](2,1) to [out=90,in=0] (1,2) to [out=180,in=90] (0,1) to [out=270,in=180] (1,0);
\draw[double,directed=.55] (1.71,1.71) to [out=45,in=180](2.5,2) ;
\draw[double,rdirected=.55] (1.71,0.29) to [out=315,in=180](2.5,0) ;
\end{tikzpicture}
$\to $\;
\begin{tikzpicture}[fill opacity=.2,anchorbase]
\draw[very thick] (1,0) to [out=0,in=270](2,1) to [out=90,in=0] (1,2) to [out=180,in=90] (0,1) to [out=270,in=180] (1,0);
\draw[double,directed=.55] (1.71,1.71) to [out=45,in=90] (2.25,1) to [out=270,in=315](1.71,0.29);
\draw[double,rdirected=.55] (3,2) to [out=180,in=90] (2.5,1) to [out=270,in=180](3,0);
\end{tikzpicture}
\caption{Types of interaction of $2$-labeled edges with a $1$-labeled circle bounding a disc: internal circles, internal and external edges.}
\end{figure}

In the following, we again work in the $q=1$ specialization.  

\begin{lemma} The morphism spaces in $\essbAWebp[2]$ are spanned by outward pointing webs, except for the endomorphism space of the empty object, which is isomorphic to $\C[c_2^{\pm 1}]$.
\end{lemma}

\begin{proof} This follows from Corollaries~\ref{cor:end} and \ref{cor:out}.
\end{proof}

\subsection{Decomposing the tensor product of extremal weight projectors}
We have seen that the tensor product of extremal weight projectors $T_m\otimes T_n$ contains a copy of $T_{m+n}$. For $\glnn{2}$ we will explicitly describe the difference $T_m\otimes T_n - T_{m+n}$ in terms of the projector $T_{|m-n|}$. The situation here is very similar to the $\slnn{2}$-case investigated in \cite{QW}.

Let $\pTr_1$ denote the linear maps on the morphism spaces of $\AWebp[2]$ that act on a web $W$ by first tensoring with $\id_1$ and then pre- and post-composing the result with splitter and merge webs between the new strand and the two rightmost $1$-labeled bottom and top boundary strands if they exist---otherwise we declare the result to be zero. We use the shorthand $\pTr_n:= (\pTr_1)^n$. The following is an example of $\pTr_2$ applied to a web $W$:
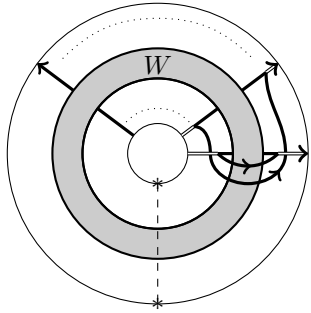
\[ 
\begin{tikzpicture}[anchorbase, scale=.4]
\draw[thick] (0,0) circle (3.5);
\fill[black,opacity=.2] (0,0) circle (3.5);
\draw[thick,fill=white] (0,0) circle (2.5);
\draw[dotted] (-1.05,1.05) to [out=45,in=180] (0,1.5) to [out=0,in=135] (1.05,1.05);
\draw [thick] (0,0) circle (2.5);
\draw[dotted] (-3.16,3.16) to [out=45,in=180] (0,4.5) to [out=0,in=135] (3.16,3.16);
\draw (0,0) circle (1);
\draw (0,0) circle (5);
\draw [very thick] (-.8,.6) to (-2,1.5);
\draw [very thick,->] (-2.8,2.1) to (-4,3);
\draw [double] (.8,.6) to (1.2,0.9);
\draw [very thick] (1.2,0.9) to (2,1.5);
\draw [very thick] (2.8,2.1) to (3.6,2.7);
\draw [double,->] (3.6,2.7) to (4,3);
\draw [very thick,directed=.55] (1.2,0.9) to [out=0,in=90] (1.75,0) to [out=270,in=180] (3,-1) to [out=0,in=270] (4.25,0) to [out=90,in=270] (3.6,2.7);
\draw [double] (1,0) to (2,0);
\draw [very thick] (2,0) to (2.5,0);
\draw [very thick] (3.5,0) to (4,0);
\draw [double,->] (4,0) to (5,0);
\draw [very thick,directed=.55] (2,0) to [out=315,in=225] (4,0);
\node at (0,-1) {$*$};
\node at (0,-5) {$*$};
\draw [dashed] (0,-1) to (0,-5);
\node at (0,2.95) {$W$};
\end{tikzpicture}
\comm{=
\begin{tikzpicture}[anchorbase, scale=.4]
\draw[thick] (0,0) circle (3.5);
\fill[black,opacity=.2] (0,0) circle (3.5);
\draw[thick,fill=white] (0,0) circle (2.5);
\draw[dotted] (-1.05,1.05) to [out=45,in=180] (0,1.5) to [out=0,in=135] (1.05,1.05);
\draw [thick] (0,0) circle (2.5);
\draw[dotted] (-3.16,3.16) to [out=45,in=180] (0,4.5) to [out=0,in=135] (3.16,3.16);
\draw (0,0) circle (1);
\draw (0,0) circle (5);
\draw [very thick] (-.8,.6) to (-2,1.5);
\draw [very thick,->] (-2.8,2.1) to (-4,3);
\draw [double] (.8,.6) to (1.4,1.05);
\draw [very thick] (1.4,1.05) to (2,1.5);
\draw [very thick] (2.8,2.1) to (3.4,2.55);
\draw [double,->] (3.4,2.55) to (4,3);
\draw [very thick,directed=.55] (1.4,1.05) to [out=0,in=270] (3.4,2.55);
\draw [double] (1,0) to (1.75,0);
\draw [very thick] (1.75,0) to (2.5,0);
\draw [very thick] (3.5,0) to (4.25,0);
\draw [double,->] (4.25,0) to (5,0);
\draw [very thick,directed=.55] (1.75,0) to [out=315,in=225] (4.25,0);
\node at (0,-1) {$*$};
\node at (0,-5) {$*$};
\draw [dashed] (0,-1) to (0,-5);
\node at (0,2.95) {$W$};
\end{tikzpicture}
}
\]
We can decompose $\pTr_n(W)=M_n (W\otimes \id_n)S_n$ where $S_n$ is a splitter web and $M_n$ is a merge web: 

\begin{equation} \label{def:splitMerge}
S_n=
\begin{tikzpicture}[anchorbase, scale=.4]
 \draw[dotted] (-1.4,1.4) to [out=45,in=180] (0,2) to [out=0,in=160] (0.68,1.86);
\draw (0,0) circle (1);
\draw (0,0) circle (3);
\draw [very thick,->] (-.8,.6) to (-2.4,1.8);
\draw [very thick,->] (.4,.91) to (1.2,2.73);
\draw [double] (.8,.6) to (1.2,0.9);
\draw [very thick,->] (1.2,0.9) to (2.4,1.8);
\draw [very thick,directed=1] (1.2,0.9) to [out=0,in=90] (1.75,0) to [out=270,in=140] (2.5,-1.6);
\draw [double] (1,0) to (2,0);
\draw [very thick,->] (2,0) to (3,0);
\draw [very thick,directed=1] (2,0) to [out=315,in=170] (2.904,-.6);
\draw [dotted] (2,1.5) to [out=-60,in=90] (2.5,0);
\draw [dotted] (2.4,-.7) to [out=-100,in=60] (2.2,-1.19);
\node at (0,-1) {$*$};
\node at (0,-3) {$*$};
\draw [dashed] (0,-1) to (0,-3);
\end{tikzpicture}
\quad
\text{and}
\quad
M_n=
\begin{tikzpicture}[anchorbase, scale=.4]
 \draw[dotted] (-1.4,1.4) to [out=45,in=180] (0,2) to [out=0,in=160] (0.68,1.86);
\draw (0,0) circle (1);
\draw (0,0) circle (3);
\draw [very thick,->] (-.8,.6) to (-2.4,1.8);
\draw [very thick,->] (.4,.91) to (1.2,2.73);
\draw [double,<-] (2.4,1.8) to (2,1.5);
\draw [very thick] (2,1.5) to (.8,.6);
\draw [very thick] (2,1.5) to [out=210,in=90] (2,0) to [out=-90,in=-50] (.7,-.71);
\draw [double,<-] (3,0) to (1.7,0);
\draw [very thick] (1.7,0) to (1,0);
\draw [very thick] (1.7,0) to [out=-135,in=10] (.95,-.3);
\draw [dotted] (2,1.5) to [out=-60,in=90] (2.5,0);
\node at (0,-1) {$*$};
\node at (0,-3) {$*$};
\draw [dashed] (0,-1) to (0,-3);
\end{tikzpicture}
\end{equation}

Recall that $\lambda$ denotes the endofunctor of $\essbAWeb[2]$ given on morphisms by tensoring with a $2$-labeled strand on the right as shown in the following: 
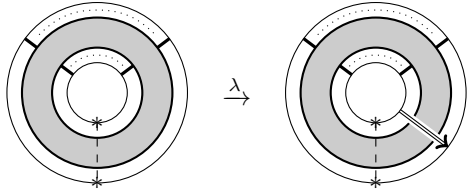
\[
\begin{tikzpicture}[anchorbase, scale=.4]
\draw[thick] (0,0) circle (2.5);
\fill[black,opacity=.2] (0,0) circle (2.5);
\draw[thick,fill=white] (0,0) circle (1.5);
\draw[dotted] (-1.93,1.93) to [out=45,in=180] (0,2.75) to [out=0,in=135] (1.93,1.93);
\draw[dotted] (-0.88,0.88) to [out=45,in=180] (0,1.25) to [out=0,in=135] (0.88,0.88);
\draw [very thick] (.8,.6) to (1.2,.9);
\draw [very thick] (-.8,.6) to (-1.2,.9);
\draw [very thick] (2,1.5) to (2.4,1.8);
\draw [very thick] (-2,1.5) to (-2.4,1.8);
\node at (0,-1) {$*$};
\node at (0,-3) {$*$};
\draw [dashed] (0,-1) to (0,-3);
\draw (0,0) circle (1);
\draw (0,0) circle (3);
\end{tikzpicture}
\quad
\xrightarrow{\lambda}
\quad 
 \begin{tikzpicture}[anchorbase, scale=.4]
\draw[thick] (0,0) circle (2.5);
\fill[black,opacity=.2] (0,0) circle (2.5);
\draw[thick,fill=white] (0,0) circle (1.5);
\draw[dotted] (-1.93,1.93) to [out=45,in=180] (0,2.75) to [out=0,in=135] (1.93,1.93);
\draw[dotted] (-0.88,0.88) to [out=45,in=180] (0,1.25) to [out=0,in=135] (0.88,0.88);
\draw [very thick] (.8,.6) to (1.2,.9);
\draw [very thick] (-.8,.6) to (-1.2,.9);
\draw [very thick] (2,1.5) to (2.4,1.8);
\draw [very thick] (-2,1.5) to (-2.4,1.8);
\draw [white, line width=.12cm] (.8,-.6) to (2.4,-1.8);
\draw [double,->] (.8,-.6) to (2.4,-1.8);
\node at (0,-1) {$*$};
\node at (0,-3) {$*$};
\draw [dashed] (0,-1) to (0,-3);
\draw (0,0) circle (1);
\draw (0,0) circle (3);
\end{tikzpicture}
\]

\begin{lemma} \label{lem:Tptr} The extremal weight projectors in $\essbAWeb[2]$ satisfy $\pTr_n(T_{m}) = \lambda^n(T_{m-n})$ for $1\leq n<m$ and also for $n=m$ if we set $T_0=2$.
\end{lemma}
\begin{proof} The cases $m=1$ and $m=2$ are easily checked, so we assume $m\geq 3$. Then we proceed by induction on $n$.
For $n=1$ we start by expanding
\begin{align*}
\pTr_1(T_{m})=\pTr_1((T_{m-1}\otimes \id_1)s_{m-1} (T_{m-1}\otimes \id_1))=\lambda(T_{m-1})\pTr_1(s_{m-1}) \lambda(T_{m-1})
\end{align*}
The result follows from the Reidemeister 1 type move $\pTr_1(s_{m-1})= \lambda(\id_{m-1})$ and idempotency of $T_{m-1}$. The induction step $n\to n+1$ is analogous, except that it additionally involves Reidemeister 2 moves between $1$- and $2$-labeled strands.
\end{proof}

\begin{lemma} For $m,n\geq 1$ we have an orthogonal decomposition of idempotents $T_m\otimes T_n = T_{m+n} + e_{m,n}$ in $\essbAWeb[2]$ where $e_{1,1}=u_1/2 + \wrapi u_1 \wrap/2$ and $e_{m,n}=(T_m\otimes T_n)u_m(T_m\otimes T_n)$ otherwise.
\end{lemma}
\begin{proof} For $m=n=1$ this follows from the explicit description of $T_2$. Otherwise we use Lemma~\ref{lem:linkedproj}:
\[T_{m+n}=(T_m\otimes T_n)s_m(T_m\otimes T_n)= (T_m\otimes T_n) - (T_m\otimes T_n)u_m(T_m\otimes T_n).\]
Since $e_{m,n}$ contains $u_m$, it is orthogonal to $T_{m+n}$. This implies that $e_{m,n}$ is an idempotent as well.
\comm{\begin{align*}
e_{m,n}^2&=(T_m\otimes T_n-T_{m+n})^2=(T_m\otimes T_n)^2-(T_m\otimes T_n)T_{m+n}-T_{m+n}(T_m\otimes T_n)+T_{m+n}^2 \\
&=T_m\otimes T_n-(T_m\otimes id_n)(id_m\otimes T_n)T_{m+n}-T_{m+n}(T_m\otimes id_n)(id_m\otimes T_n)+T_{m+n} \\
&= T_m\otimes T_n-2T_{m+n}+T_{m+n}=T_{m}\otimes T_n-T_{m+n}
\end{align*}
}
\end{proof}

\begin{lemma}\label{lem:diffproj}
For $1\leq n,m$ and $n+m\geq 3$, the idempotent $e_{m,n}$ can alternatively be written as 
\[e_{m,n}=(T_m\otimes T_n)(T_{m-r}\otimes(S_rM_r)\otimes T_{n-r})(T_m\otimes T_n)\] where $1\leq r\leq\min(m,n)$ and $M_r$ and $S_r$ are the merge and splitter webs introduced in equation \ref{def:splitMerge}.
\end{lemma}
\begin{proof} By using forkslides we can write $S_rM_r= \beta u_1\cdots u_r\beta^{-1}$ where $\beta$ is the permutation 
$(1,2r,2,2r-1,\dots,r,r+1)$. After replacing $u_i$ by $\id-s_i$, we see that $S_rM_r$ can be expressed as a signed sum of $2^{r}$ permutations, with precisely $2^{r-1}$ terms carrying minus signs. The identity appears only once with positive sign, and all other permutations $\gamma$ satisfy $\gamma\in S_{2r}\setminus (S_r\times S_r)$ and thus $(T_m\otimes T_n)(T_{m-r}\otimes\gamma\otimes T_{n-r})(T_m\otimes T_n)=T_{m+n}$ by Lemma~\ref{lem:linkedproj} and crossing absorption. This implies:
\[ (T_m\otimes T_n)(T_{m-r}\otimes(S_rM_r)\otimes T_{n-r})(T_m\otimes T_n) = (T_m\otimes T_n) + (2^{r-1}-1) T_{m+n} - 2^{r-1} T_{m+n} =  e_{m,n}  \] 
 \end{proof}

\begin{lemma}\label{lem:kariso} For $1\leq n\leq m$ we have
$(\id_{m-n}\otimes M_n)(T_m\otimes T_n)(\id_{m-n}\otimes S_n) = \lambda^n(T_{m-n})$ in $\essbAWeb[2]$. 
\end{lemma}
\begin{proof} This follows from Lemma~\ref{lem:Tptr} once we have proved that $M_n(T_n\otimes T_n)=M_n(T_n\otimes \id_n)$. The case $n=1$ is trivial, whereas for $n=2$ we have $M_2(T_2\otimes T_2)- M_2(T_2\otimes \id_2)= M_2(T_2\otimes e_{1,1})$ and it is not hard to check that the latter is zero. For the induction step we compute
\begin{align*}
M_n(T_n\otimes T_n) &= M_n(\id_{n-2}\otimes T_2\otimes T_2\otimes \id_{n-2}) (T_{n-1}\otimes \id_2\otimes T_{n-1}) \\
&=M_n(\id_{n-2}\otimes T_2\otimes \id_n)(T_{n-1}\otimes \id_2\otimes T_{n-1})
\\
&= M_n(T_{n-1}\otimes \id_2\otimes T_{n-1})(\id_{n-2}\otimes T_2\otimes \id_n)
\\
&= M_n(T_{n-1}\otimes \id_{n+1})(\id_{n-2}\otimes T_2\otimes \id_n)=  M_n(T_n\otimes \id_n)
\end{align*}
In the first and last line, we use (\ref{item:overlap}) in Theorem \ref{thm:Tm}. For the second and last line we use the case $n=2$ and the induction hypothesis for $n-1$. The third line arises from projector commutation.
\end{proof}

\begin{proposition} 
\label{prop:parprod}For $1\leq n< m$ the idempotents $e_{m,n}$, $e_{n,m}$ and $\lambda^{n}(T_{m-n})$ represent isomorphic objects in $\Kar(\essbAWeb[2])$.
\end{proposition}
\begin{proof} 
We use Lemma~\ref{lem:diffproj} to write $e_{m,n}=(T_m\otimes T_n)(T_{m-n}\otimes S_n)(T_{m-n}\otimes M_n)(T_m\otimes T_n)$. Then it is immediate from Lemma~\ref{lem:kariso} that the maps $(T_m\otimes T_n)(T_{m-n}\otimes S_n)$ and $(T_{m-n}\otimes M_n)(T_m\otimes T_n)$ are inverse isomorphisms between the elements of the Karoubi element represented by the idempotents $e_{m,n}$ and $\lambda^{n}(T_{m-n})$. The proof for $e_{n,m}$ is similar.
\end{proof}

\begin{proposition}
\label{prop:parprod2}
The idempotent $e_{m,m}$ is isomorphic to $\lambda^m(\emptyset)\oplus \sh \lambda^m(\emptyset)$ in $\Kar(\essbAWebp[2])$.
\end{proposition}
\begin{proof}
For $m=1$ we have $e_{1,1}=u_1/2 + \wrapi u_1 \wrap/2$. The two summands are orthogonal idempotents. The first is isomorphic to $\lambda(\emptyset)$ in $\Kar(\essbAWebp[2])$, while the conjugation by $\wrap$ in the second summand makes it isomorphic to $\sh \lambda(\emptyset)$ . For $m>1$ we rewrite 
\begin{align*} e_{m,m} &= (T_m\otimes T_m)(\id_1\otimes(S_{m-1}M_{m-1})\otimes \id_1)(T_m\otimes T_m)
\\
&= (T_m\otimes T_m)S_{m}M_{m}(T_m\otimes T_m)/2 + (T_m\otimes T_m)\wrapi u_{2m-1}\wrap(\id_1\otimes(S_{m-1}M_{m-1})\otimes \id_1)(T_m\otimes T_m)/2 
\\
&= \underbrace{(T_m\otimes T_m)S_{m}}_{\phi_1}\circ \underbrace{M_{m}(T_m\otimes T_m)/2}_{\psi_1} +  \underbrace{(T_m\otimes T_m)\wrapi (S_{m-1}\otimes S_1)}_{\phi_2} \circ \underbrace{(M_{m-1}\otimes M_1)\wrap(T_m\otimes T_m)/2}_{\psi_2} 
\end{align*}

The equality in the second line can be verified by inserting $\wrapi(S_{m-1}M_{m-1}\otimes T_2)\wrap$ between two factors of $T_m\otimes T_m$ and realizing that the result is zero. To prove the proposition it remains to verify that $\psi_i \phi_j = \delta_{i,j} \lambda^m(\emptyset)$. We give one example for orthogonality:
\begin{align*}
\psi_1 \phi_2 &= M_{m}(T_m\otimes T_m)\wrapi (S_{m-1}\otimes S_1)/2 \\
&= M_{m}(T_m\otimes \id_m)\wrapi (S_{m-1}\otimes S_1)/2 
= \lambda^{m-1}(\emptyset) \otimes (M_1 \wrapi S_1) =0 
\end{align*}
Here we have used the proof of Lemma~\ref{lem:kariso}, an isotopy and the essential torus relation. The proof of $\psi_2 \phi_1=0$ is analogous. $\psi_1 \phi_1= \lambda^m(\emptyset)$ follows from Lemma~\ref{lem:Tptr}. It remains to check 

\[\psi_2 \phi_2= (M_{m-1}\otimes M_1)\wrap(T_m\otimes T_m)\wrapi (S_{m-1}\otimes S_1)/2 =  \lambda^m(\emptyset).\] For $m=2$ this follows by expanding the left copy of $T_2$ and seeing that all terms except the identity term die. The result is evaluated using Lemma~\ref{lem:Tptr}. For $m\geq 3$, we use the recursion on the left copy of $T_m$, absorb the resulting copies of $T_{m-1}$ as the proof of Lemma~\ref{lem:kariso} and then simplify via Lemma~\ref{lem:Tptr}. The result is a equal to $\lambda^{m-2}(\emptyset)$ superposed with the $m=2$ case, which we have already checked.
\end{proof}

\subsection{Skeleta}
Recall that we put basepoints on the boundary components of the annulus and fix a connecting arc $\alpha$ between them, which cuts the annulus into a square -- this is drawn as a dashed line above.

\begin{lemma}\label{lem:simpleend} The endomorphism algebra of $T_m$ in $\Kar(\essbAWeb[2])^*$ is isomorphic to $\C[\wrap^{\pm 1}]$ if $m\geq 1$ and isomorphic to $\C[c_2^{\pm 1}]$ if $m=0$. In $\Kar(\essbAWeb[2])$ both are given by $\C$.
\end{lemma}

\begin{proof} 
Any endomorphism of $T_m$ is represented by a linear combination of outward pointing webs with $m$ $1$-labeled input and output strands. Such webs factor into dumbbells $u_i$ and wraps $\wrap^{\pm 1}$. Since $T_m$ kills all $u_i$, the endomorphism algebra is generated by the isomorphisms $\wrap^{\pm 1}$. Since all web relations in $\essbAWeb[2]$ preserve the flow winding number of the web around the annulus, it is clear that the elements $\wrap^{m}$ for $m\in \Z$ are linearly independent.
\end{proof}

\begin{lemma} The endomorphism algebra of $\lambda^k(T_{m})$ in $\Kar(\essbAWeb[2])^*$ is isomorphic to $\C[\lambda^k(\wrap^{\pm 1})]$ if $m\geq 1$ and isomorphic to $\C[\wrap_2^{\pm 1}]$ if $m=0$. In $\Kar(\essbAWeb[2])$ both types of endomorphism algebras are isomorphic to $\C$. 
Furthermore, in both versions of the Karoubi envelope, there are no non-zero morphisms between $\lambda^k(T_{m})$ and $\lambda^l(T_{n})$  (or their $\sh$-shifts) unless $m=n$ and $k=l$. \label{lem:simpleendo}
\end{lemma}
\begin{proof}
The first part is an immediate corollary of Lemma~\ref{lem:simpleend}, due to the fact that $\lambda$ is an auto-equivalence, which guarantees that it induces isomorphisms on endomorphism algebras. For the second part, it is clear that we need $m+2k=n+2l$ to have both objects in the same block. Now suppose, without loss of generality, that $k<l$. Then an application of the essential inverse $(\lambda^*)^k$ provides an isomorphism between the morphism space in question and $\Hom(T_{m}, \lambda^{l-k}(T_{n}))$. However, any web representing a morphism in that space necessarily contains a merge vertex, which is killed by $T_{m}$. Thus $\Hom(T_{m}, \lambda^{l-k}(T_{n}))=0$.
\end{proof}

\begin{lemma} For $m\geq 1$ we have $\Kar(\essbAWeb[2])(\sh^a \lambda^k(T_{m}),\sh^b \lambda^k(T_{m})) \cong \C\langle \lambda^k(\wrap^{b-a})\rangle$ for any $a$, $b$ and $k$. On the other hand:\[\Kar(\essbAWeb[2])(\sh^a \lambda^k(\emptyset),\sh^b \lambda^k(\emptyset))\cong \begin{cases}\C\langle \wrap_2^{(b-a)/2}\rangle & \text{ if } a-b \text{ is even and } k\geq 1\\
\C\langle c_2^{(b-a)/2}\rangle & \text{ if } a-b \text{ is even and } k=0\\
0 & \text{ if } a-b \text{ is odd}\end{cases} \]
Clearly, all such non-zero morphisms are isomorphisms.
\end{lemma}
In particular, this implies, that all objects of the form $\sh^a \lambda^k(T_{m})$ are actually isomorphic to unshifted objects $\lambda^k(T_{m})$ if $m\geq 1$. The objects $\sh^a \lambda^k(\emptyset)$, on the other hand, are isomorphic to their versions with $a=1$ or $a=0$.

\begin{lemma}\label{lem:webdecomp} Any object in $\Kar(\essbAWebp[2])$ is isomorphic to a direct sum of objects $\lambda^k(T_{m})$ for $m>1$, $k\geq 0$ or $\lambda^k(\emptyset)$ or $\sh\lambda^k(\emptyset)$ for $k \geq 0$.
\end{lemma}

\begin{proof} It suffices to decompose the objects in $\essbAWebp[2]$ into a formal direct sum of objects of the above type. Moreover, since $2$-labeled objects are isomorphic to idempotents on $1$-labeled objects in the Karoubi envelope, we only need to decompose $\id_m$. Then any idempotent endomorphism of $\id_m$ will give rise to an idempotent endomorphism of the decomposition, which is necessarily block-diagonal (there are no morphisms between distinct objects of the form $\lambda^k(T_{m-2k})$ or $\sh\lambda^{m/2}(\emptyset)$) and has entries in $\C$. Such idempotent matrices can be diagonalized, and thus decompose into objects of type $\lambda^k(T_{m-2k})$ or $\sh\lambda^{m/2}(\emptyset)$.

 The decomposition for $\id_m$ follows inductively from the parallel product formulas in Propositions \ref{prop:parprod} and \ref{prop:parprod2}.
More precisely, if we already know that $\id_{m}$ is isomorphic to a direct sum of terms $\lambda^{k_i}(T_{m-2k_i})$ and possibly $\sh \lambda^{m/2}(\emptyset)$ if $m$ is even, then $\id_{m+1}$ can be decomposed into summands  
\begin{align*}
\lambda^{k_i}(T_{m-2k_i}) \otimes \id_1 
\cong  \lambda^{k_i}(T_{m-2k_i}\otimes T_1) \cong 
\begin{cases}
\lambda^{k_i}(T_{m-2k_i+1}) \oplus  \lambda^{k_i+1}(T_{m-2k_i-1}) & m-2k_i>1\\
 \lambda^{k_i}(T_{2}) \oplus \lambda^{k_i+1}(\emptyset) \oplus \sh \lambda^{k_i+1}(\emptyset) & m-2k_i=1 
\end{cases}
\end{align*}
as well as $ \lambda^{m/2}(\emptyset) \otimes \id_1 \cong \sh \lambda^{m/2}(\emptyset) \otimes \id_1
\cong  \lambda^{m/2}(T_1)$.
\end{proof}

Since orientation of the boundary can be reversed by means of the auto-equivalences $\lambda$ and $\lambda^*$, one easily extends the previous result to the whole category:
\begin{corollary}\label{cor:webdecomp} Any object in $\Kar(\essbAWeb[2])$ is isomorphic to a direct sum of objects $\lambda^k(T_{m})$ for $m>1$, $k\in \Z$ or $\lambda^k(\emptyset)$ or $\sh\lambda^k(\emptyset)$ for $k \in \Z$. Here we identify $\lambda^{-1}=\lambda^{*}$.
\end{corollary}

\begin{definition} A subcategory $D$ of a category $C$ is a skeleton of $C$ if the inclusion of $D$ into $C$ is an equivalence of categories and additionally no two distinct objects of $D$ are isomorphic. 
\end{definition}

\begin{lemma} The full subcategory of $\Kar(\essbAWeb[2])$ containing (lexicographically ordered direct sums of) the objects $\lambda^k(T_{m})$ and $s \lambda^k(\emptyset)$ for $k,m\geq 0$ is a skeleton of $\Kar(\essbAWebp[2])$. Moreover, this skeleton is semisimple.
\end{lemma}
\begin{proof}
The inclusion of this full subcategory is essentially surjective by Lemma~\ref{lem:webdecomp}. Moreover, by Lemma~\ref{lem:simpleendo}, the decomposition of an object of $\Kar(\essbAWebp[2])$ into a lexicographically ordered direct sum of such simples is essentially unique. In particular, there are also no isomorphisms between distinct direct sums of simples. We have also seen that the endomorphism algebras of simples are isomorphic to $\C$, which implies that the skeleton is semisimple.
\end{proof}


%

%
\end{document}